\documentclass[11pt,a4paper,english]{article}   
\usepackage{graphicx}
\usepackage{latexsym}
\usepackage[english]{babel}
\usepackage{subfigure}
\usepackage[T1]{fontenc}
\usepackage[utf8]{inputenc}
\usepackage{hyperref}
\usepackage{setspace}
\usepackage{graphicx}
\usepackage{float} 
\usepackage{amssymb}
\usepackage{amsthm}
\usepackage{mathtools}
\usepackage{xcolor}
\usepackage{enumitem}
\usepackage[hmargin=2.5cm,vmargin=2.5cm]{geometry}
\usepackage{natbib}
\theoremstyle{theorem}
 \usepackage{tabulary}

\theoremstyle{theorem}
\newtheorem{lemma}{Lemma}[section]
\newtheorem*{theorem*}{Theorem}
\newtheorem{theorem}[lemma]{Theorem}
\newtheorem{prop}[lemma]{Proposition}
\newtheorem{definition}[lemma]{Definition}

\newtheorem{rem}[lemma]{Remark}

\usepackage[hmargin=2.5cm,vmargin=2.5cm]{geometry}

\usepackage{dsfont}
\usepackage{centernot}

\renewcommand{\geq}{\geqslant}
\renewcommand{\leq}{\leqslant}
\def\Q{\mathds{Q}}
\def\R{\mathds{R}}
\def\N{\mathds{N}}

\def\E{\mathbb{E}}
\def\Pp{\mathbb{P}}
\def\1{\mathds{1}}

\def\F{{\cal F}}
\def\pP{{\cal P}}
\def\T{{\cal T}}
\def\Rc{\mathbf{r}}
\def\notRc{\centernot{\mathbf{r}}}

\newcommand{\indic}[1]{\mathbf{1}\{{#1}\}}
\begin{document}
\title{Chromosome painting}
\author{Amaury Lambert, Ver\'onica Mir\'o Pina and Emmanuel Schertzer\\
\emph{Corresponding author}  \href{mailto:veronica.miropina@normalesup.org}{veronica.miropina@normalesup.org}}
\maketitle 


\section*{Abstract}
We consider a Moran model with recombination in a haploid population of size $N$.  At each birth event, with probability $1-\rho_N R$ the offspring copies one parent's chromosome, and with probability $\rho_N R$ she inherits a chromosome that is a mosaic of both parental chromosomes. We assume that at time $0$ each individual has her chromosome painted in a different color and we study the color partition of the chromosome that is asymptotically fixed in a large population, when we look at a portion of the chromosome such that $\rho := \lim_{N\to \infty} \frac{\rho_N N}{2}\to \infty$. To do so, we follow backwards in time the ancestry of the chromosome of a randomly sampled individual. This yields a Markov process valued in the color partitions of the half-line, that was introduced by \cite{esser}, in which blocks can merge and split, called the partitioning process. Its stationary distribution is closely related to the fixed chromosome in our Moran model with recombination. We are able to provide an approximation of this stationary distribution when $\rho \gg 1$ and an error bound.
This allows us to show that the distribution of the (renormalised) length of the leftmost block of the partition (i.e. the region of the chromosome that carries the same color as 0) converges to an exponential distribution. In addition, the geometry of this block can be described in terms of a Poisson point process with an explicit intensity measure. 

\section{Introduction} 
\subsection{Motivation: a Moran model with recombination} 
Genetic recombination is the mechanism by which, in species that reproduce sexually, an individual can inherit a chromosome that is a mosaic of two parental chromosomes. Many classical population genetics models ignore recombination and only focus on a single locus, i.e. a location on the chromosome with a unique evolutionary history. In this setting, many analytical results are known. For example the time to fixation (i.e. the first time at which all individuals carry the same allele) or the fixation probabilities (see for example \cite{etheridge2011}). However, understanding the joint evolution of different loci is well known to be mathematically challenging, as one needs to take into account non-trivial correlations between loci along the chromosome. For instance, loci that are close to one another are difficult to recombine, so they often inherit their genetic material from the same parent and as a consequence, often share a similar evolutionary history. On the contrary,  loci that are far from one another will tend to have different, but not independent, evolutionary histories. 

\smallskip

To visualize the questions that will be addressed in this work, let us imagine that in the ancestral population, each individual carries a single continuous chromosome painted in a distinct color.  By the blending effect of recombination, after a few generations, the chromosome of each individual looks like a mosaic of colors, each color corresponding to the genetic material inherited from a single ancestral individual. Some natural questions arise: What does the mosaic of colors that is fixed in the population look like?  How many colors are there? If the leftmost locus is red (i.e. is inherited from the individual with red chromosome in the ancestral population),  what is the amount of red in the mosaic and where are the red loci located?
These questions are interesting from a biological point of view: for example, the number of colors in the mosaic corresponds to the number of ancestors that have contributed to an extant chromosome and has been studied in \cite{WH}.  
Loci that are of the same color (i.e., that have been inherited from the same individual in the ancestral population) are called identical-by-descent (IBD). 
We note in passing that the name ``chromosome painting'' is also used for admixture mapping methods, which are statistical methods designed to infer the origin of each chunk of a genome by tracing it back to one of $k$ differentiated ancestral populations (most of the time $k=2$, sometimes more, but $k$ small). In the present work, there is one color per ancestral individual in the (single) population, while in admixture mapping, there are only $k$ colors, one for each ancestral population.

\smallskip 

It is known that changes in the population size or natural selection can alter the sizes of the IBD segments: for example genes that are under selection tend to be located within large IBD segments. This prediction can guide the detection of genes that are under selection (see for example the methods developed by \cite{sabeti} or \cite{mcquillan_runs_2008}). 
The aim of this article is to characterize the distribution of the IBD blocks along a chromosome in the absence of selection or demography. Our results may then be used as predictions under the null hypothesis, that can serve as a standard to compare against real data, e.g., to infer selection or demography.

 Also, our results may be relevant to the analysis of data obtained in experimental evolution. For example, in the experiment carried out by  \citet{teotonio}, the authors intercrossed individuals from 16 different subpopulations of the worm \textit{C. elegans}  and let the population evolve for several generations at controlled population size. Then, each individual is genotyped, 
each of its variants is mapped to one of the 16 ancestor subpopulations, so as to get a representation of each DNA sequence of each individual as a partition of the sequence into 16 colors. Again, our model (or an extension to our model accommodating for the finite number of colors), might be used as a null model whose predictions can be compared to these real color mosaics. 

\begin{figure}
\begin{center}
\includegraphics[width=10cm]{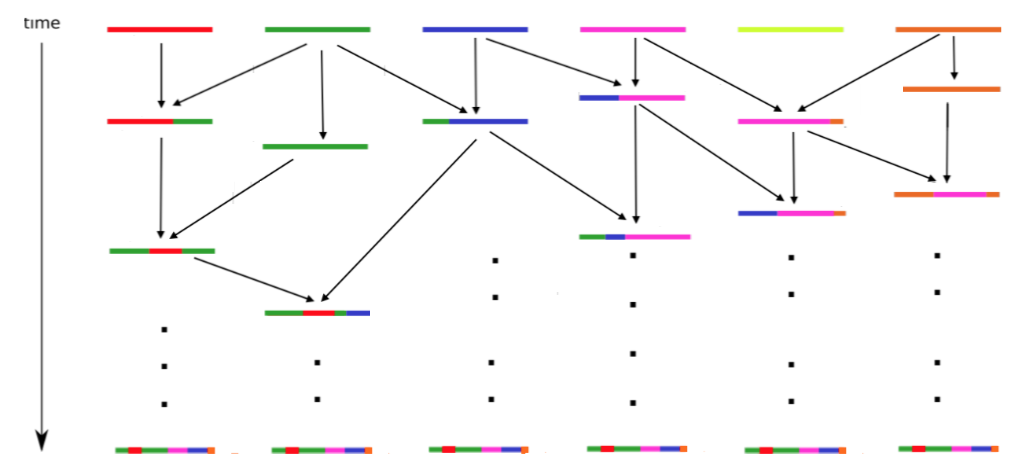}
\caption{Moran model model with recombination.}
\label{figWF}
\end{center}
\end{figure}

Sampling the chromosome, seen as a continuous, single-ended strand modelled by the positive half-line, of an individual in the present population and tracing backwards in time the ancestry of every locus yields a process valued in the partitions of $\R^+$, called the $\R^+$-partitioning process.  
Here, $x\ge 0$ and $y\ge 0$ belong to the same block of the $\R^+$-partitioning process at time $t$ if the loci at positions $x$ and $y$ on the sampled chromosome shared the same ancestor $t$ units of time ago.

Before giving a formal description of this object, we start by showing how it arises naturally from a multi-locus Moran model.
The population size is $N$ and each haploid individual carries a single linear chromosome of length $R$. At time 0, each individual has her (unique) chromosome painted in a distinct color (see Figure \ref{figWF}). 
Each individual reproduces at rate $1$, and upon reproduction, the individual chooses a random partner in the population. Let $\rho_N$ such that $\rho_N R \in (0,1)$, 
\begin{itemize}
\item With probability $1-\rho_NR$, the offspring copies one parent's chromosome (chosen uniformly at random). 
\item With probability $\rho_NR$, a recombination event occurs. We assume single-crossover recombination which means that each parental chromosome is cut into two fragments.  The position of the cutpoint (i.e. the crossover) is uniformly distributed along the chromosome (see Figure \ref{figWF}). The offspring copies the genetic material to the left of this point from one parent and the genetic material to the right of this point  from the other parent. 
\end{itemize}
 \begin{figure}
\begin{center}
\includegraphics[width=10cm]{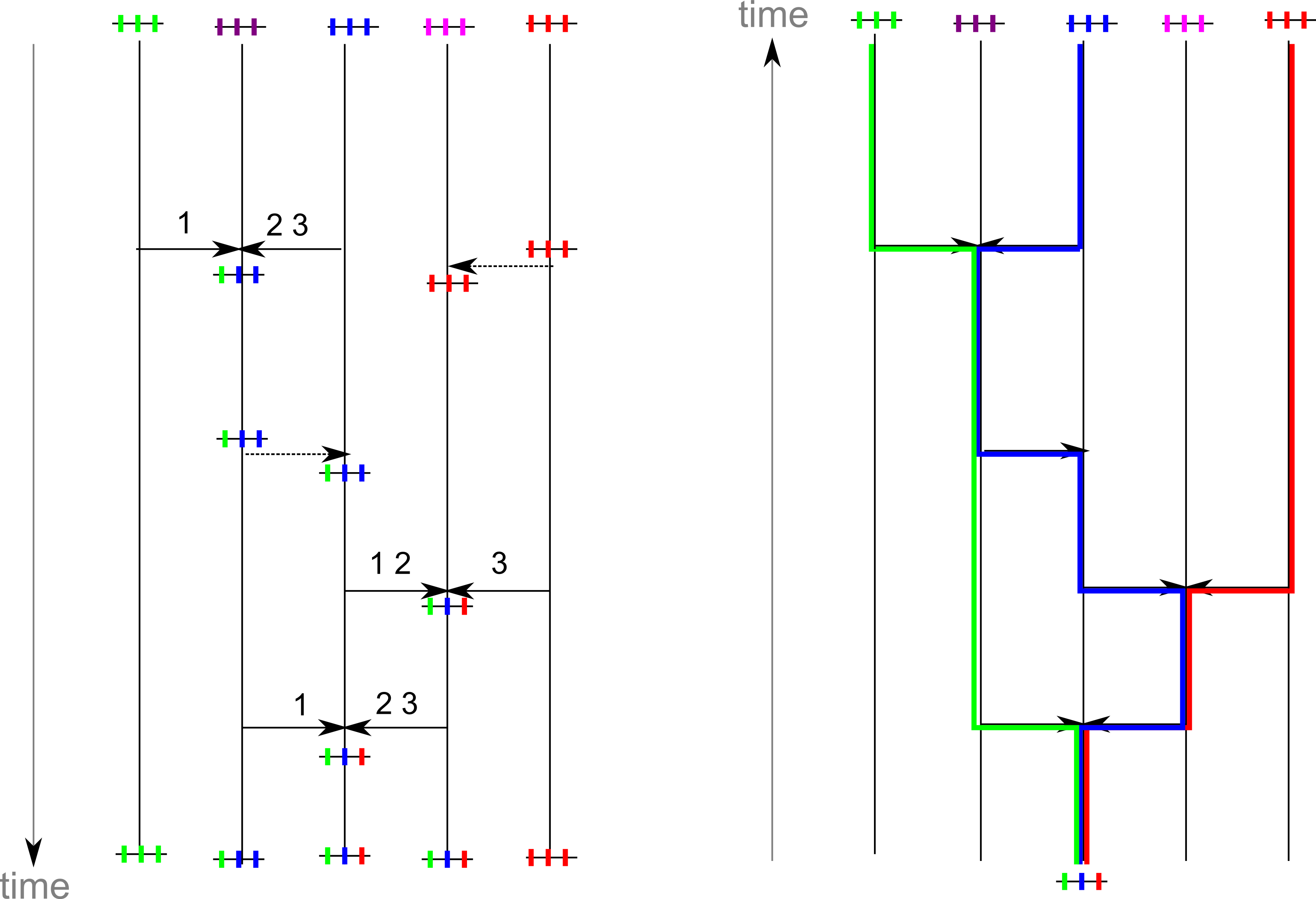}
\caption{Duality between the Moran model and the discrete partitioning process ($N=5, \ n=3$). The left panel represents a realization of the Moran model with recombination. On the top we represented the chromosomes of the different individuals at generation $0$.
Arrows represent reproduction events. For reproduction events without recombination, the tail of the arrow represents the parent from which the genetic material is inherited, and the arrow points at the individual that is replaced. For reproduction events with recombination (represented by two arrows), we indicate on top of each arrow the loci that are inherited from each parent. 
 In the second panel we show how the discrete partitioning process can be obtained by reverting time. We sample  the $3$ loci in an individual in the present population and we follow the corresponding ancestral lineages. Each time an ancestral lineage finds the tip of an arrow, it jumps to its tail. If there are two arrows, the lineage corresponding to locus $i$ jumps to the tail of the arrow whose label includes $i$.  }
 \label{moran}
\end{center}
\end{figure}
The offspring then replaces a randomly chosen individual in the population.
Because of recombination,  at time $t$ each chromosome is a mosaic of colors, each color corresponding to the genetic material inherited from one individual in the founding generation. (In other words, loci  sharing the same color are IBD.)

Let us now consider $z= \{z_0, z_1,\ldots,z_n\} \subset [0,R]$ the set of  the locations of $n+1$ loci along the chromosome (with $z_0 < z_1< \ldots < z_n$). Forward in time, the evolution of the genetic composition of the population 
can be described in terms of an $(n+1)$-locus Moran model with recombination as described in \cite{durrett, Bobrowski2010}.
Backward in time, the genealogy of those loci (sampled from the same individual) 
is described in terms of the discrete partitioning process, introduced by \cite{esser}, 
which traces the history of the $n+1$ loci under consideration (see Figure \ref{moran}).
More precisely, the discrete partitioning process (associated to $z$) is a Markov process valued in the partitions of  $z$   such that
$z_i$ and $z_j$ are in the same block at time $t$ if and only if they inherit 
their respective genetic material from the same individual $t$ units of time ago. (In other words, $z_i$
and $z_j$ are IBD if we look $t$ units of time in the past).
In a population of size $N$, it can be seen that the dynamics of the discrete partitioning process are controlled by the following transitions. In the following, we use the notation $O(f(N))$ to denote the order of the function $f(N)$ when $N\to \infty$. 
 \begin{itemize}
\item Each pair of blocks coalesces at rate $2/N + O(\rho_N/N)$.
\item Each block $b  = \{z_{i_1}, \ldots, z_{i_k}\}$ is fragmented into $\{z_{i_1}, \ldots, z_{i_j}\}$ and $\{z_{i_{j+1}}, \ldots, z_{i_k}\}$ at rate $\rho_N (z_{i_{j+1}}-z_{i_j})$.
\item Simultaneous splitting and coalescence events happen at rate $O(\rho_N/N)$.
\end{itemize}

The interesting scaling for this process is when time is accelerated by $N/2$ and the recombination probability scales with $N$ in such a way that 
\begin{equation}\lim_{N \to \infty}\rho_N N/2 = \rho, \label{eq:scaling-rho}\end{equation}
for some $\rho>0$. Recall that this corresponds to the well-known diffusion limit of the underlying Moran model with recombination (see for example \cite{durrett}).  For any  finite subset $z$ of $\R^+$, let $\pP_z$ be the set of partitions of $z$ and $\F_z$ the standard sigma algebra on $\pP_z$. It can readily be seen that the discrete partitioning process in a population  of size $N$ converges in distribution (in the Skorokhod topology) to a process  $(\Gamma^{\rho,z}_t; t \geq0)$ on ($\pP_z, \F_z)$ with the following transition rates:
\begin{itemize}
\item \textbf{Coagulation:}  
Consider  $\pi_1, \pi_2 \in \pP_z$ such that $\pi_2$ is obtained from $\pi_1$ by coalescing two of its blocks.
A transition from $\pi_1$ to $\pi_2$ occurs at rate $q(\pi_1, \pi_2) = 1$.

\item \textbf{Fragmentation:} Now take $\pi_1  \in \pP_z$ and $a$ a block of $\pi_1$ containing $k$ elements $z_{i_1}, \ldots, z_{i_k}$ such that $z_{i_1} < \ldots < z_{i_k}$. Let $j < k$.  Let $b = \{z_{i_1}, \ldots, z_{i_j}\}$ and  $c = \{z_{i_{j+1}}, \ldots, z_{i_k}\}$ and $\pi_2$, the partition obtained by fragmenting $a$ into $b$ and $c$. A transition from $\pi_1$ to $\pi_2$ occurs at rate $q(\pi_1, \pi_2) = \rho (z_{i_{j+1}} - z_{i_j})$.
\item All these events are independent and all other events have rate $0$.
\end{itemize} 
In the literature, $\Gamma^{\rho,z}$ is also referred to  as the ancestral recombination graph (ARG) \citep{hudson, ARG2, ARG3} associated to $z$ (with recombination rate $\rho$).
The following scaling property  can easily be deduced from the description of the transition rates. We assume that at time $0$ all loci are sampled in the same individual, i.e. we consider the ARG started from the coarsest partition. Then
\begin{equation}\forall R>0, \ \  \Gamma^{R, z} = \Gamma^{1, R z} \ \textrm{ in distribution.} \label{scaling0} \end{equation}
In the following, we are going to consider a high recombination regime, i.e. that $\rho$ is large.  This relation states that this is equivalent to considering that the distances between loci of interest are large. 
This regime is also in force in several other studies (e.g., to compute sampling formulas for the coalescent with recombination as done by \citealt{jenkins1, jenkins}.)

\subsection{The $\R^+$-partitioning process} 
As the goal of this article is to characterize the distribution of the IBD blocks in a continuous chromosome in an infinite population, we extend the ARG $(\Gamma^{\rho,z}_t; t \geq0)$ to the whole positive real line. To do so, we will consider partitions of $\R^+$. 
We call a maximal set of connected points belonging to the same block of the partition a \textit{segment}. A partition of ${\R^+}$ is right-continuous if the segments of the partition are left-closed (right-open) intervals and the blocks correspond to disjoint unions of such intervals. Let $\pP^{loc}$ be the set of partitions of $\R^+$ that are right-continuous and locally finite, i.e. such that each compact subset of $\R^+$ contains only a finite number of segments. 
For any finite subset $z$ of $\R^+$ and any partition $\pi \in \pP^{loc}$, $\mathrm{Rest}_z(\pi)$ is  the restriction of $\pi$ to $z$, i.e. the function $\pP^{loc} \to \pP_z$ such that for any $\pi \in \pP^{loc}$,  $\mathrm{Rest}_z(\pi)$  is the partition of $z$ induced by $\pi$. We define $\cal F$ as the $\sigma$-field on $\pP^{loc}$ generated by 
\begin{equation}
{\cal C} :=  \{ \{ \tilde \pi \in \pP^{loc}, \mathrm{Rest}_z(\tilde \pi) = \pi \}, n \in \N, \ z =\{z_0, \ldots, z_n\} \subset {\R^+},  \pi \in \pP_z \}.
\label{definitionC}
\end{equation}
Finally, for any measure $\mu$ on a measure space $(\Omega, {\cal A})$ and any $ {\cal A}$-measurable function $f$, we will denote by $f\star \mu$ the pushforward of $\mu$ i.e. the measure such that $\forall B \in  {\cal A}$, $f\star \mu(B) = \mu(f^{-1}(B))$.
\begin{theorem}
Let $\mu_0$ be a probability measure on $(\pP^{loc}, \F)$. 
There exists a unique c\`adl\`ag stochastic process valued in $(\pP^{loc}, {\cal F})$ and started  at  $\mu_0$, called the $\R^+$-partitioning process $(\Pi_t^{\rho}; t\ge 0)$ (started  at  $\mu_0$), such that for any  finite subset $z$ of $\R$, $(\mathrm{Rest}_z(\Pi_t^{\rho}); t \ge 0)$ is the ARG at rate $\rho$ for the set of loci $z$ started at $\mathrm{Rest}_z \star \mu_0$.
\label{unicityPi}
\end{theorem}

A construction of the $\R^+$-partitioning process as the projective limit of a sequence of processes that can be constructed in terms of Poisson point processes will be given in Section \ref{construction}.
The proof of this theorem of existence and uniqueness can also be found in Section \ref{construction}.  The goal of this paper is to study some properties of the invariant measure of the $\R^+$-partitioning process. The following result will be proven in Section \ref{mesure}.

\begin{theorem}\label{existencemeasure}
The $\R^+$-partitioning process $(\Pi^{\rho}_t;  t \ge 0 )$ has a unique invariant probability measure $\mu^\rho$ in $(\pP^{loc},  {\mathcal F})$.  In addition, for any finite subset $z$ of $\R^+$, 
\[  \mathrm{Rest}_z \star \mu^\rho  \ = \ \mu^{\rho,z}\]
where $\mu^{\rho,z}$ is the unique invariant measure of $\Gamma^{\rho,z}$.
\end{theorem}

\subsection{Approximation of the stationary distribution of the ARG} 
 The ARG with more than two loci is a complex process and some authors have considered that characterizing its distribution is ``computationally not tractable'' (see \cite{Bobrowski2010}).
\cite{lessard}  and \cite{esser} provided methods to compute the stationary distribution that fail when considering a large number of loci.
One of the main contributions of this paper is an explicit approximation (and an error bound for it) of the stationary distribution of the ARG $(\Gamma^{\rho,z}_t; t \geq0)$ when the typical distance between the $z_i$'s is large (or equivalently when the rate of recombination $\rho$ is large), that  is relatively easy to handle, even when we consider a large number of loci. 

\smallskip

In the following, we fix  $z = \{z_0, \ldots, z_n\} \subset \R$. We define
\[\alpha = \min_{i\neq j} |z_i-z_j|\]
and we assume that $\alpha>0$ (or equivalently that the coordinates of $z$
are pairwise distinct).  
For $r \in \{0,\ldots,n\}$, we denote by ${\cal P}_z^r$ the set of partitions of $z$ containing $n+1-r$ blocks. We say that the partitions in ${\cal P}_z^r$ are of order $r$. In particular, the only partition in  ${\cal P}_z^0$ is $\pi_0:= \{\{z_0\}, \ldots, \{z_n\}\}$, the partition made of singletons. 
Recall that, if $\pi \in \pP_z^r$, it can be obtained from the finest partition $\pi_0$ by $r$ successive coagulation events.
For example, the partition  $\{\{z_0\}, \ldots, \{z_i,z_j,z_k\},  \ldots, \{z_n\}\}$ is of order 2 and can be obtained from the finest partition in three different ways: 
\begin{eqnarray*}
\{\{z_i\}\{z_j\}\{z_k\} \ldots\} \ &\to \  \{\{z_i,z_j\}\{z_k\}\  \ldots\}  \ &\to \  \{\{z_i,z_j,z_k\}\}    \\
\{\{z_i\}\{z_j\}\{z_k\} \ldots\} \ &\to \  \{\{z_i,z_k\}\{z_j\}\  \ldots\}  \ &\to \  \{\{z_i,z_j,z_k\}\}   \\ 
\{\{z_i\}\{z_j\}\{z_k\} \ldots \}\ &\to \  \{\{z_k,z_j\}\{z_i\}\  \ldots\}  \ &\to \  \{\{z_i,z_j,z_k\}\}.  
\end{eqnarray*}
  Intuitively, when $\rho \gg 1$ or $\alpha \gg1$, fragmentation occurs much more often than coalescence so, at stationarity, the partition made of singletons is the most likely configuration and the probability of a partition decreases with its order. 
We define a ``coalescence scenario of order $r$'' as a sequence of partitions  $(s_k)_{0\le k \le r}$ such that $s_0 = \pi_0$ and for $1 \le k \le r$, $s_k$ is a partition of order $k$  that can be obtained from $s_{k-1}$ by a single coagulation event. For any  partition in $\pi \in {\cal P}_z^r$, ${\cal S}(\pi)$ is the set of coalescence scenarios of order $r$ such that $s_r = \pi$.

 For $\pi \in \pP_z^r$, let $b_1, \ldots, b_{n+1-r} $ be the blocks of $\pi$. We denote by $C(\pi)$ the cover length of $\pi$ defined as: $$C(\pi) \ := \ \sum \limits_{i}\max \limits_{x, y  \in b_i} |x-y|.$$
 In particular, the cover length of $\pi_0$ is equal to 0. 
  Notice that $C(\pi)$ is the total rate of fragmentation from state $\pi$ divided by $\rho$.
  It is important to recall that $\pi$ is not necessarily an interval partition so that $C(\pi)$ can be greater than $C(\{\{z_0, \dots, z_n\}\})$, i.e., each block of the partition can be formed of several disjoint intervals carrying the same color. 
{In \cite{WH} the authors call ``trapped material''  the non-ancestral material enclosed between two ancestral segments (or intervals). ``Trapped material'' is the reason why $C(\pi)$ can be greater than $C(\{\{z_0, \dots, z_n\}\})$. }
 
Let   $s = (s_k)_{0 \le k \le r}$  be a scenario of coalescence of order $r$,  with $1\le r \le n$. We define the energy of $s$, $E(s)$ as
$$E(s) := \prod \limits_{i =1}^r { C(s_i) }. $$
In words, the energy of a scenario is the product of the total fragmentation rates (divided by $\rho$) at each step.
 Finally, define
\begin{align}  \forall \pi \in {\cal P}_z \setminus  {\cal P}_z^0 , \ \ & \ F(\pi) := \sum_{S \in {\cal S}(\pi)} \frac1{E(S)}.
\label{def:F}
\end{align}
In the rest of the paper, we will denote by $[n]$ the set $\{1, \dots, n\}$.
\begin{theorem}
For all $n\in\N$, there exists a function $$f^n: \R^+ \setminus \{ 0 \} \cup \{\infty\}  \to \R^+, \  \ \lim_{x \to \infty} f^n(x) = 0,$$ independent of the choice of $z  = \{z_0, \ldots, z_n\}$ and $\rho$, such that
 \begin{align*}
\forall \rho>0, \ \forall k \in [n], \ \forall \pi_k \in  {\cal P}_z^k, \ \  
 \left | \mu^{\rho, z}(\pi_k) - \frac1{\rho^k} F(\pi_k) \right| \le f^n(\alpha \rho)  \frac1{\rho^k} F(\pi_k) .
\end{align*} 
\label{thm-stationary}
\end{theorem}
Recall that the RHS goes to $0$ either when  $\rho \to \infty$ or $\alpha \to \infty$.
As already mentioned (see \eqref{scaling0}), these two scaling limits are equivalent.
We let the reader refer to Section \ref{section4} for a proof of this result. 

\begin{rem}
The latter result provides an approximation of the invariant distribution of the ARG for a large recombination rate. Other studies in this regime rather focus on the genetic diversity of a sample of $n$ individuals at $k$ loci (joint site frequency spectrum, where genetic linkage induces a non-trivial correlation between sites), see e.g., \cite{jenkins1, jenkins} in the case $k=2$ and \cite{bhaskar} for general $k$. 
In those studies, the authors are interested in the observable differences between individual genomes due to mutations having occurred \emph{since} the last common ancestor at each locus.  
In \cite{jenkins}, the authors defined the so-called \textit{loose-linkage coalescent}, which is the multivariate backward-time process following the $n$ lineages at each locus.
In contrast, we are only interested in the history of these lineages \emph{beyond} the common ancestor, once the ancestral population has reached fixation.
The two approaches look at different time-scales and look at different regimes of the ancestral process (transient phase vs invariant phase of the ARG).
{More precisely, in \cite{jenkins} the state space is paths of transitions from a given partition to the finest partition and the most probable state is the path comprising only recombination events. By considering excursions that deviate from the most probable path via occasional coalescences, one can obtain an asymptotic formula for the sampling distribution. In our paper the most probable state is the finest partition, and excursions are sequences of states bringing different loci together transiently.}
\end{rem}

\subsection{Characterization of the leftmost block of the $\R^+$-partitioning process}
\label{1.5}
As an application of our approximation of $\mu^{\rho,z}$, we characterize the geometry of the leftmost block  on a large scale (NB: what we call the ``leftmost block'' is the block that contains the leftmost segment, i.e. the block containing $0$).
 Motivated by the  Moran model and the  scaling relation
(\ref{scaling0}),  without loss of generality, we study the $\R^+$-partitioning process at rate $1$ restricted to $[0,R]$.

For any partition $\pi$,  $x \sim_{\pi} y$ means $x$ and $y$ are in the same block of $\pi$.  Let $\Pi_{eq}$ be the random partition with law $\mu^1$.
Let ${\cal L}_R(0)$ be the length of the block containing $0$, rescaled by $\log(R)$. More precisely, 
$${\cal L}_R(0) =  \frac1{\log(R)} \int_{[0,R]} \mathds{1}_{\{x \sim_{\Pi_{eq}} 0\}} dx.$$ 

 We define the random measure $\vartheta^R[a,b]$ such that
\begin{equation*}
\forall a, b \in [0,1], \ a \le b, \ \ \ \vartheta^R[a,b] = \frac1{\log(R)} \int_{R^a}^{R^b} \mathds{1}_{\{x\sim_{\Pi_{eq}} 0\}}dx
\end{equation*}
so that
$\vartheta^R$ encapsulates the whole information about the positions of the loci that are IBD to 0 in the {\it logarithmic scale} (which will be seen to be the natural scaling for the partitioning process at equilibrium). 
In the following $\vartheta^R$ will be considered as a random variable valued in ${\cal M}([0,1])$, the space of locally finite measures of $[0,1]$
equipped with the weak topology (i.e. the coarsest topology making $m\to\left<m,f\right>$ continuous for every function $f$ bounded and continuous).
In the following, $\Longrightarrow$ denotes  convergence in distribution.
\begin{figure}
\begin{center}
\includegraphics[width=9cm]{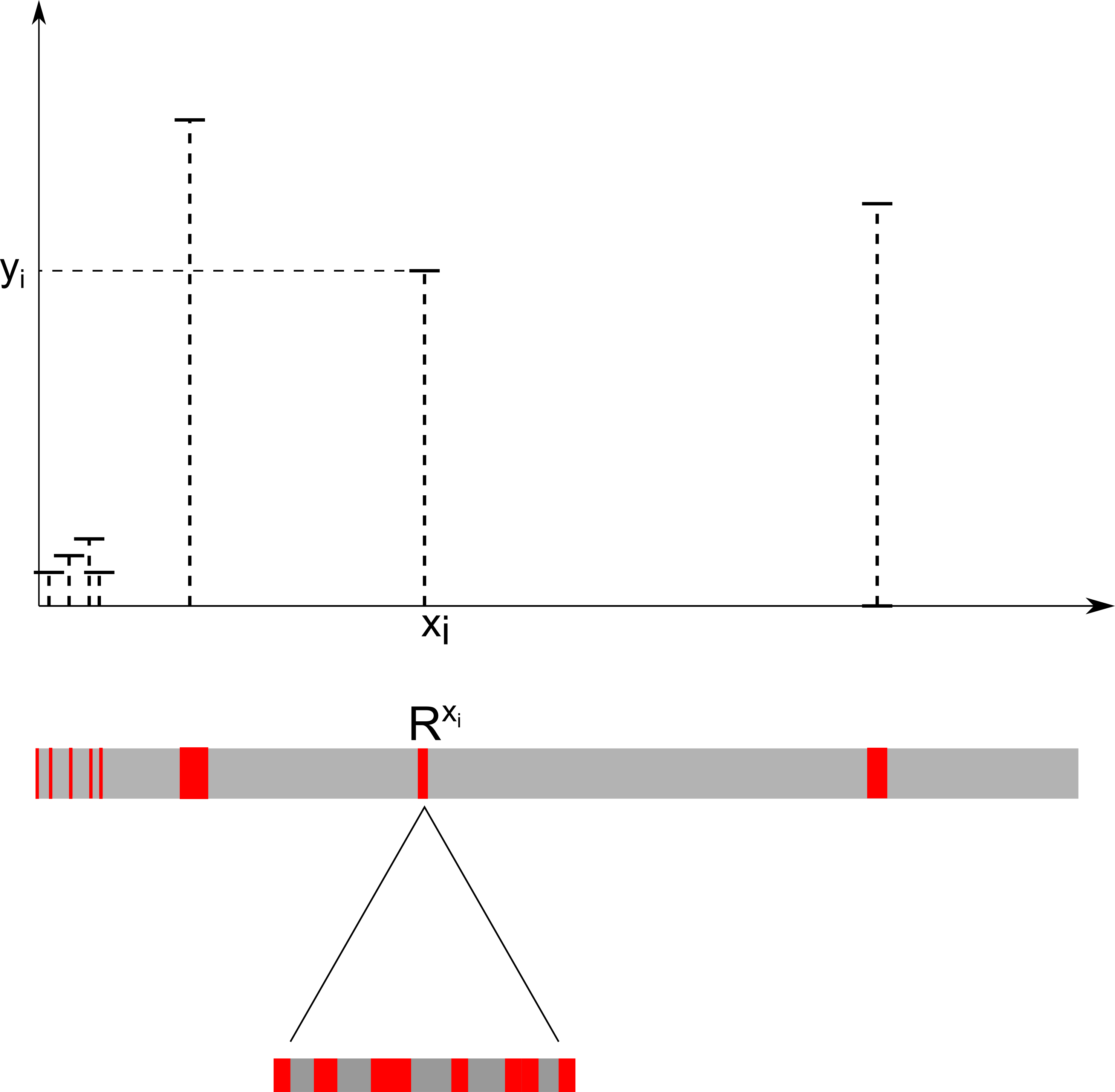}
\caption{Example of a realization of $\vartheta^\infty$ and its interpretation. Regions of the chromosome (in the log-scale) that are IBD to $0$ are represented in red. In the limit, in the logarithmic scale, those regions are shrunk to points which can have a complex geometry on a finer scale (see lower figure). $y_i$ is the amount of genetic material IBD to $0$ in the region located at $R^{x_i}$. We conjecture that the fine structure of those subregions (lower red and grey segment)
could also be described by a Poisson point process.}
\label{varthetainfty}
\end{center}
\end{figure}
\begin{theorem}\label{thm:geommetry}
Consider a Poisson point process $\cal P^{\infty}$  on $[0,1] \times \R^+$ with intensity measure 
\[\lambda(x, y) = \frac{1}{x^2} \exp(-y/x) dx dy\]
and define the random measure on ${\cal M}([0,1])$
\[\vartheta^{\infty} \ := \ \sum_{(x_i, y_i) \in {\cal P^{\infty}}} \ y_i \ \delta_{x_i}. \] 
Then
\begin{enumerate}
\item $\vartheta^R \underset{R \to \infty}{\Longrightarrow}  \vartheta^{\infty}$ in the weak topology.
\item In particular, $   \mathcal{L}_{R^{x}}(0) \underset{R \to \infty}{\Longrightarrow} \varepsilon(x)$
where $\varepsilon(x)$ denotes the exponential distribution with parameter $x$.
\end{enumerate}
\end{theorem}

Recall that $2.$ was already conjectured in \cite{WH} using simulations.
The first part of the theorem can be interpreted as follows.
As $R\to\infty$, there are distinct regions of genetic material that is IBD to 0, and in the limit, when the chromosome is seen in a logarithmic scale, those regions are shrunk to points. The locations of those regions are encapsulated by the $x_i$'s  (in the logarithmic scale) -- in other words, there is a cluster of genetic material IBD to $0$ in an interval $[R^{x_i-dx},R^{x_i+dx}]$ with $dx\ll1$  -- and the coordinate $y_i$ corresponds the amount of genetic material that is IBD to $0$ present in this region (see Figure \ref{varthetainfty}).
Note that our theorem does not give us any information about the fine geometrical structure
of the IBD points concentrated around a point $R^{x_i}$: the only information available 
is the total length of this structure.  We conjecture that the fine structure of those regions could be also described by a PPP. (See again Figure \ref{varthetainfty}).
Finally, we note that the positions of the segments (in the logarithmic scale) are given by the Poisson process of intensity $(1/x) dx$, which is known as ``the scale invariant Poisson process'' (see for example \cite{arratia}).
\begin{rem}
1. and 2. in Theorem \ref{thm:geommetry} are consistent since the random variable 
$\int_{[0,x] \times \R^+} \vartheta^\infty(du dy)$ is easily seen to be distributed as an exponential random variable with parameter $x$. Conversely, it is not hard to see that
if $(e^x;x\geq0)$ is a family of random variables such that (a) $e^x$ is distributed as an exponential 
with parameter $x$, and (b) for $y>x$, $e^y-e^x$ is independent of $(e^t; t\leq x)$, then the family 
$(e^x; x\geq0)$ is distributed as $(\int_{[0,x] \times \R^+} \vartheta^\infty(du dy); x\geq0)$. Thus, in essence, Theorem \ref{thm:geommetry} states that the length of the leftmost block is exponential, and 
that on a log-scale, the "geometries" of the cluster on different disjoint intervals are independent.
\end{rem}
We performed some numerical simulations of the partitioning process to illustrate the second part of  Theorem \ref{thm:geommetry}. Figure \ref{cluster0hist} shows how the length of the cluster covering $0$ is exponentially distributed.

\subsection{Biological relevance}
Recall that a ``morgan'' is a unit used to measure genetic distance. The distance between two loci is 1 morgan if the average number  of crossovers is 1 per reproduction event. In other words, in a population of size $N$,  if we consider the discrete partitioning process at rate $1$, two loci $z_i$ and $z_j$ are at distance $\frac{2}{N}|z_i-z_j|$ morgans.  

We studied the $\R^+$-partitioning process at rate $1$ restricted to $[0,R]$,
 which  should correspond to a portion (or frame)
of the chromosome that is of size $R/N$ morgans  (and small enough so that the single crossing-over approximation is valid).
We first let the population size $N$ tend to infinity  and then the size of the frame go to $\infty$. 
In \cite{jenkins}, the authors give an explicit rescaling of a Moran model that also leads to a strong recombination regime, namely a total recombination rate $R=O(\sqrt{N})$ as $N \to \infty$. However, in our case, to get the partitioning process from the underlying finite population model, we need to take successive limits (first $N \to \infty$ and then $R \to \infty$), but it remains unclear how the population size and the size of the observation frame should scale with one another to ensure that the approximation is correct. We conjecture that our results should hold under the strong recombination assumption of \cite{jenkins} and more generally whenever $R,N$ go to $\infty$ simultaneously under the constraint  $R/N\to 0$. 

Strong recombination rates have been reported in some species such as \textit{Drosophila melanogaster} (see \cite{chan}), but also in human populations.
In \cite{WH}, the authors explain that, as the size of chromosome 1 is 2.93 morgans and the effective population size is $N = 20000$, if one looks at a frame of this chromosome of length 1 morgan ($1/3$ of the chromosome), then $R = 20000$ (recall that $R/N$ is the chromosome length in morgans).

Our work is based on a precise analysis of the ARG. In contrast \cite{baird} have looked at a similar problem using a forward approach. They investigated how the genetic material from a given ancestor (whose chromosome is painted in red for example) is passed to its descendants in the early phase of the fixation process, always assuming strong recombination. Their analysis is based on a branching process approximation. Among other things, they provide the expected number of descendants carrying some red material and moment calculations for the distribution of the amount of red material per descendant. It would be interesting to investigate further the relation between the two approaches.

\subsection{Outline} This paper is organized as follows. In Section \ref{construction} we propose a construction of the $\R^+$-partitioning process and we prove Theorem \ref{unicityPi}. In Section \ref{mesure} we show the existence and uniqueness of a stationary distribution for this process (Theorem \ref{existencemeasure}). 
Finally, Sections \ref{section4}  and \ref{results} are devoted to the proofs of Theorems  \ref{thm-stationary} and  \ref{thm:geommetry} respectively.

\section{The $\R^+$-partitioning process}
\label{construction}
\subsection{Some preliminary definitions}
\label{partitions}

We start by defining some notation. 
 As we already mentioned in the introduction, we are going to consider partitions that are right continuous and locally finite i.e. that are in $\pP^{loc}$.
 Note that for a partition that is right continuous, infinite sequences of small intervals can only accumulate to the left of a point. 
We need to define a distance $d$ on $\pP^{loc}$. 
To do so, we start by identifying each partition in $\pP^{loc}$ to a function from ${\R^+}$ to itself. More precisely,
we define a map $\phi: \pP^{loc}\to D({\R^+}, {\R^+})$ such that, for $\pi \in \pP^{loc}$, $\phi(\pi)$ is constructed as follows. 
For each block $b$ of $\pi$ and for each  $x \in b$,  we set $\phi(\pi)(x) := \min(b)$. 
Note that $\phi$ is injective and $\forall x \in \R^+,  \ \phi(\pi)(x)  \le x$. Also, as $\pi \in \pP^{loc}$, $\phi$ is c\`adl\`ag and has a finite number of jumps in any compact set of $\R^+$. 
Now,  for any $\pi_1, \pi_2 \in \pP^{loc}$, define
\begin{equation*}d(\pi_1, \pi_2) \ := \ \int_{0}^{+\infty} | \phi(\pi_1)(x) - \phi(\pi_2)(x)| \exp(-x) dx.\end{equation*}
It can easily be checked that $d$ defines a distance on $\pP^{loc}$. 
\begin{rem} The idea is that $\pi_1$ and $\pi_2$ are close in the metric $d$ if blocks of $\pi_1$ and $\pi_2$ can be put in correspondence (via their minimum) in such a way that each pair of corresponding blocks is close in some standard meaning (like a small Hausdorff distance or a small symmetric difference), locally. Let us be more specific. 
Let $\delta$ denote the Hausdorff distance between (closed) subsets of $\R$. Assume that between two partitions $\pi_1$ and $\pi_2\in \pP^{loc}$, there is a relation $\Rc$ between their blocks such that for each pair of blocks $b_1$ and $b_2$ such that $b_1\Rc b_2$, $\delta(b_1, b_2)\le \varepsilon$, so that in particular $|\min b_1 - \min b_2 |\le \varepsilon$. Further assume that 
$$
\sum_{b_1\Rc b_2}\mu(b_1\Delta b_2) \le \eta,
$$
where the sum is taken over all pairs $(b_1, b_2)$ such that $b_i$ is a block of $\pi_i$, $i=1,2$, and $b_1\Rc b_2$, and $\mu$ is the measure with density $xe^{-x}$. 
Now for each $x$, let $b_1(x)$ and $b_2(x)$ be the blocks of $\pi_1$ and $\pi_2$ respectively, which $x$ falls into. Then
$$
|\phi(\pi_1)(x)- \phi(\pi_2)(x) |\le x\indic{b_1(x)\notRc b_2(x)} + \varepsilon \indic{b_1(x)\Rc b_2(x)}.
$$
Integrating this last inequality against $e^{-x}$ yields $d(\pi_1, \pi_2) \le \eta +  \varepsilon$. For our purpose, it would have been equivalent to use the distance 
$$
\inf_{\Rc} 
\sum_{b_1\Rc b_2}\mu(b_1\Delta b_2),
$$
where $\mu$ is any absolutely continuous, finite measure on $\R^+$. The advantage of $d$ is the simplicity of its definition and the fact that it lends itself more easily to projective limits, as the distance between two blocks does not change when considering partitions of $[0,L']$ instead of partitions of $[0,L]$ with $L'>L$.
\end{rem}

For $T> 0$, we will denote by $D([0,T], \pP^{loc})$ the Skorokhod space associated to $(\pP^{loc}, d)$ equipped with the standard Skorokhod topology.  For each partition $\pi \in \pP^{loc}$ we define a natural ordering on its blocks. We denote by $b^0, b^1, \ldots, b^i, \ldots $ the blocks of $\pi$  indexed in such a way that $\min(b^0)< \min( b^1) < \ldots$ 

The space $\pP^{loc}$ is separable under $d$. Indeed, for $n \in \N \setminus \{0\}$, let ${ \cal S}_n$ be the set of partitions in  $ \pi \in \pP^{loc}$ such that in $\pi|_{[0, n[}$ each block is a finite union of segments whose endpoints are in $[0,n[\cap\Q$ and $[n, +\infty[$ is included in a block of $\pi$.
 $S = \cup_n { \cal S}_n$ is countable and using standard methods, it can be shown that given $\pi \in \pP^{loc}$ and $\epsilon>0$, there exists a partition $\pi' \in S$ such that $d(\pi, \pi') <\epsilon$.
The space $\pP^{loc}$ is not complete but we define its completion  $\bar \pP^{loc}$.

 \begin{figure}
\begin{center}
\includegraphics[width = 8cm]{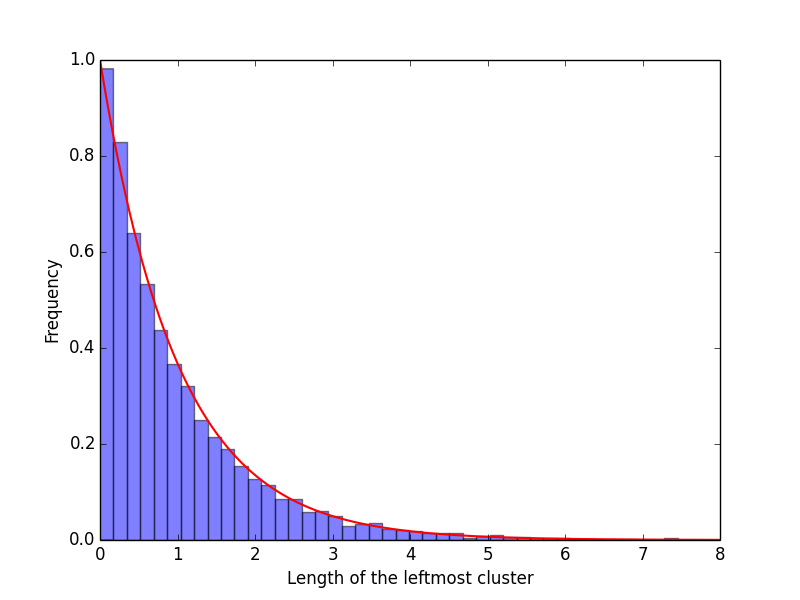}
\end{center}
\caption{\textbf{Distribution of the length of the leftmost block (R = 5000).} The blue histogram represents the empirical distribution, that was obtained by simulating the partitioning process, for a chromosome of length $R = 5000$. The number of replicates is 10000. The red curve is the probability density function of an exponential distribution of parameter 1. 
We compared the empirical distribution to an exponential distribution using a Kolmogorov-Smirnov test, which was negative, with a $p$-value of $0.15$.}
\label{cluster0hist}
\end{figure}

In the following, we  will also consider partitions of $\Q^+$.  We define $\pP^{loc}_{\Q}$ as the set of locally finite partitions of $\Q^+$ that are right continuous (in the sense that if $\Q^+ \ni x_n \downarrow x \in \Q^+$ then $x_n$ is in the same segment as $x$ for $n$ large enough) and  ${\cal F}_\Q$ the $\sigma$-field generated by
$${\cal C}_\Q =  \{ \{ \tilde \pi \in \pP_\Q, \mathrm{Rest}_z(\tilde \pi) = \pi \}, n \in \N, \ z = \{z_0, \ldots, z_n\} \subset \Q^+, \ \pi \in \pP_z \}.$$

\subsection{Definition of the $\R^+$-partitioning process}
\label{ARG}
We want to define a process on $(\pP^{loc}, {\cal F})$, called the \textit{$\R^+$-partitioning process} so that for any  finite subset $z$  of $\R^+$, the restriction on $z$ is distributed as the ARG $\Gamma^{\rho,z}$.

 For $a<b \in \R^+$, let $\pP^{loc}_{[a,b]}$ be the set of the partitions of $[a,b]$  that are right continuous and finite.
 We define the partitioning  process on $\pP^{loc}_{[0, L[}$,  $(\Pi_t^{\rho, L}; t \geq 0)$.
To do so, we set $\Pi^{\rho, L}_0 =  \pi^0, \ \pi^0 \in \pP^{loc}_{[0, L[}$ and
we assume that the blocks of this partition are indexed with the natural order defined in the previous section.
The process on $(\Pi_t^{\rho, L}; t \geq 0)$ is generated by a sequence of independent Poisson point processes as follows: 
\begin{itemize}
\item For all $i,j \in \N$, $Y^{i,j}$ is a Poisson point process of intensity $1$. For $t \in Y^{i,j}$,  at time $t^-$ there is a  \textbf{coagulation} event: blocks $b^i$ and $b^j$ are replaced by $b^i \cup b^j$. If $i$ or $j$ does not correspond to the index of any block, nothing happens.
\item For all $i \in \N$, $X^i$ is a Poisson point process on $\R^+ \times [0, L[$ with intensity $\rho \  dt \otimes dx$. The atoms of $X^i$ correspond to \textbf{fragmentation} events.  For $(t, x) \in X^i$,  if at time $t^-$, $\Pi^{\rho, L}_{t^-} = \pi$,  $\Pi^{\rho, L}_{t}$ is  equal to the coarsest common refinement of $\pi$ and $\{[0,x[,[x,L]\}$. In other words, if $b^i$ is a block of $\pi$  and  $x \in \ ] \min( b^i ),\sup( b^i )[$,  $b^{i}$ is fragmented into two blocks $b^{i, -}$ and $b^{i, +}$ such that $b^{i, -} = b^{i} \cap [0, x[$ and $b^{i, +} = b^{i} \cap [x, L]$.  Then $\Pi^{\rho, L}_{t} $ is equal to the partition obtained by replacing $b^{i}$ by $b^{i, -}$ and $b^{i, +}$.  If $x \notin \ ] \min(b^i ),\sup(b^i )[$, nothing happens.
\end{itemize}
After each event, blocks are relabelled in such a way that they remain ordered, in the sense specified above.  Recall that, with this construction, the partitions that are formed are always right continuous. 
Also, the number of blocks of $(\Pi_t^{\rho, L}; t \geq 0)$  is stochastically dominated by a birth-death \label{page:logistic} process which jumps from $n$ to $n+1$ at rate $\rho L n$ and from $n$ to $n-1$ at rate $n(n-1)/2$ with initial condition the number of blocks in $\pi^0$. These are the transition rates of the logistic branching process (\cite{lambert}) and of the block counting process of the ancestral selection graph (\cite{ASG}), which are known to remain locally bounded (and even to have $+\infty$ as entrance boundary). There is the same stochastic domination between the two processes for the numbers of jump events on any fixed time interval.
This shows that the number of blocks in $\Pi_t^{\rho, L}$ is a.s. locally bounded and since the number of segments jumps at most by $+1$ at each event, the number of segments is also a.s. locally bounded. 
So a.s. for all $t$, $\Pi^{\rho, L}_t \in \pP^{loc}_{[0, L[}$.

\smallskip 

Finally, we define the partitioning process $\Pi_t^\rho$ in $\R^+$, as the projective limit of $(\Pi^{\rho,L}_t;  t \ge 0)_{L \in \R^+}$ as $L\to  \infty$. In fact, by construction, $\forall L' > L, \forall t \ge 0,  \Pi^{\rho, L'}_t|_{[0, L[} = \Pi^{\rho, L}_t$, where  $\Pi^{\rho, L'}|_{[0, L[}$ is the natural restriction of $ \Pi^{\rho, L'}_t$ to ${[0, L[}$.

\begin{prop}The $\R^+$-partitioning process, $(\Pi_t^{\rho}; \ t\ge 0)$ with initial measure $\pi^0$ is the unique c\`adl\`ag stochastic process valued in $(\pP^{loc}, {\cal F})$ such that 
$$\forall L \geq0, \ (\Pi_t^{\rho} \cap [0, L ];  t\geq0) =  (\Pi_t^{\rho, L};  t\geq0)$$
with $\Pi_0=\pi^0$.
Further for any finite subset $z$ of $\R^+$, $\mathrm{Rest}_z(\Pi^\rho)$ is distributed as $\Gamma^{\rho,z}$, the ARG 
with initial condition $\mathrm{Rest}_z(\pi^0)$.

\label{definitionPi}
\end{prop}
 
\begin{proof}

We need to check that, for any $T>0$,   $(\Pi_t^{\rho};  \ 0\le t\le T) \in D([0,T], \pP^{loc})$ almost surely.
To do so, we need to prove that with probability $1$, for every $t \in [0,T]$, for every $\epsilon>0$, one can find $s>0$  such that $d(\Pi_t^{\rho}, \Pi_{t+s}^{\rho}) < \epsilon$.\
Fix  $\epsilon>0$ and pick $L>0$ such that $2\exp(-L)(L+1) < \epsilon$.
From the Poissonian construction, for any $T>0$,  the process $\Pi^\rho|_{[0, L[}$ has a finite number of jumps in $[0,T]$ almost surely, which happen at times $t_1, \ldots, t_n$.
We choose  $s>0$ such that $|t-s| <\min_i|t_{i+1} - t_i|$. Then  $\Pi_t|_{[0, L[} = \Pi_{t+s}|_{[0, L[} $.
As  $ \phi(\Pi_t^\rho|_{[0, L[}) = \phi(\Pi_{t}^\rho)|_{[0, L[}$,  for any $ x \in [0, L[$, $\phi(\Pi_t^\rho)(x) = \phi(\Pi_{t+s}^\rho)(x)$, so
\begin{align*}
d(\Pi^\rho_t, \Pi_{t+s}^\rho) &=  0 +  \int_{L}^{+ \infty} |\phi(\Pi^\rho_t)(x)  - \phi(\Pi_{t+s}^\rho)(x) | e^{-x} dx \\ 
& \le \  \int_L^{+ \infty} x \exp(-x)dx  \ = \  \exp(-L)(L+1)  \ < \ \epsilon,
\end{align*}
and similarly for left-hand limits.
So  $(\Pi_t^{\rho}; \ 0 \le  t\le T) \in D([0,T], \pP^{loc})$.
The fact that $\mathrm{Rest}_z(\Pi^\rho)$ is distributed as $\Gamma^{\rho,z}$, the ARG 
with initial condition $\mathrm{Rest}_z(\pi^0)$ can be readily seen from the definition.
\end{proof}

In addition, the following proposition can easily be deduced (note that the second equality is just a trivial consequence of the first one). 
Recall that the ARG $\Gamma^{\rho, z}$ has a finite state space and is irreducible. Let us denote by $\mu^{\rho, z}$ its unique invariant probability measure (that will be characterized in Section \ref{section4}). 
\begin{prop}[Consistency]
For every finite subsets $y$ and $z$ of $\R^+$ such that $y \subset z$, 
$$(\Gamma^{\rho, y}; t\geq0)  \overset{d}{=} (\Gamma^{\rho, z}|_{ y}; t \geq0),$$
where $\Gamma^{\rho, z}|_{ y}$ denotes the restriction of $\Gamma^{\rho, z}$ to $\pP_y$, and
$$ \mu^{\rho, y} =  \mathrm{Rest}_{y} \star  \mu^{\rho, z}.$$
\label{consistency}
\end{prop}
\begin{proof}[Proof of Theorem \ref{unicityPi}]
Let   $(\Pi_t; t\geq0)$ be a  c\`adl\`ag process  in $(\pP^{loc}, {\cal F})$ such that for any finite subset  $z$ of $\R$, $$(\mathrm{Rest}_z(\Pi_t);  t \geq 0)\overset{d}{=} ( \Gamma^{\rho,z}_t;  t\geq0)  \overset{d}{=} (\mathrm{Rest}_z(\Pi^\rho_t);  t \geq 0) .$$
We denote by  $\{z_i\}_{i \in \N}$ an enumeration of the rational numbers and for all $n \in \N$, we define the set $z^n := \{z_0, \ldots, z_n\}.$
For every $n >1$, we have
 $$\mathrm{Rest}_{z^n}(\Pi^\rho) =  \mathrm{Rest}_\Q(\Pi^\rho)|_{z^n} \ \textrm{ and } \  \mathrm{Rest}_{z^n}(\Pi) =  \mathrm{Rest}_\Q(\Pi)|_{z^n},$$
so we have  $$(\mathrm{Rest}_\Q(\Pi^\rho_t); t\geq0) \overset{d}{=} (\mathrm{Rest}_\Q(\Pi_t); t\geq0).$$
In particular  \begin{equation}\forall \{t_1, \ldots, t_n\} \subset \R^+ \  \ ( \mathrm{Rest}_\Q(\Pi^\rho_{t_1}),\ldots ,\mathrm{Rest}_\Q(\Pi^\rho_{t_n}) )  \overset{d}{=} ( \mathrm{Rest}_\Q(\Pi_{t_1}),\ldots ,\mathrm{Rest}_\Q(\Pi_{t_n}) ) .\label{eq:t:law} \end{equation}
  Similarly as done on p.\pageref{page:logistic}, the number of blocks of and number of events undergone by $(\mathrm{Rest}_z(\Pi_t); \ t \in [0,T])$ are stochastically dominated, uniformly in $z\subset \Q\cap [0,L]$, by those of a birth-death process which jumps from $n$ to $n+1$ at rate $\rho L n$ and from $n$ to $n-1$ at rate $n(n-1)/2$ with initial condition the number of blocks in $\Pi_0$. This shows that a.s. for all $t\ge 0$, $\mathrm{Rest}_\Q(\Pi_t) \in \pP^{loc}_\Q$ (and of course $ \mathrm{Rest}_\Q(\Pi^\rho_t) \in \pP^{loc}_\Q$). Since the partitions in $\pP^{loc}$ (resp.  $\pP^{loc}_\Q$) are right-continuous  and since $\Q$ is dense in $\R$,  for every $\bar \pi\in  \pP^{loc}_\Q$ 
there exists a unique $ \pi\in \pP^{loc}$ such that $\mathrm{Rest}_{\Q}(\pi) = \bar \pi$. In other words,  the projection map
\[ \mathrm{Rest}_{\Q} : ( \pP^{loc}, {\cal F}) \to ( \pP^{loc}_\Q, {\cal F}_\Q) \]
is bijective.   With a little bit of extra work, one can show that $\mathrm{Rest}_{\Q}^{-1}$ is measurable
so, from \eqref{eq:t:law},  $$\forall \{t_1, \ldots, t_n\} \subset \R^+, \ \ (\Pi^\rho_{t_1}, \ldots, \Pi^\rho_{t_n})  \overset{d}{=} (\Pi_{t_1}, \ldots, \Pi_{t_n}).$$
This implies that $\forall T>0$, $(\Pi_t; \ 0\le  t \le T) \overset{d}{=} (\Pi^\rho_t; \ 0\le  t \le T)$, in the Skorokhod topology $D([0,T], \bar \pP^{loc})$ (see \cite{billingsley1968convergence}, Theorem 16.6). 
So $(\Pi_t^{\rho}; \ t\geq 0)$ is the unique process in $D([0,T], \bar \pP^{loc})$  such that for any finite subset $z$ of $\R$, $(\mathrm{Rest}_z(\Pi_t^{\rho}); t \geq 0)$ is distributed as $(\Gamma^{\rho, z}_t; t\geq0)$. As $\pP^{loc} \subset \bar \pP^{loc}$, the theorem is proved. 
\end{proof}

\section{Stationary measure for the $\R^+$-partitioning process}
\label{mesure}
The goal of this section is to prove Theorem \ref{existencemeasure}.
The idea of the proof is to consider the stationary measure of the partitioning process on finite sets of rational numbers. Using Kolmogorov's extension theorem we define its unique projective limit in $\pP^{loc}_\Q$. Then, using continuity arguments, we prove that there is a unique extension of this measure to the partitions of $\R$. Let us now go into more details.
We decompose the proof into several lemmas.

\begin{lemma}
A measure $\nu$ is invariant for $(\Pi^{\rho}_t; t \geq 0 )$ iff for any finite subset $z$ of $\R^+$, $\nu \circ \mathrm{Rest}_z^{-1}$ is invariant for $(\mathrm{Rest}_z(\Pi^{\rho}_t);  t \geq 0)$.
\label{lemma-invariant}
\end{lemma}

\begin{proof}
We obviously only prove the ``if'' part. We consider a probability measure $\nu$ and for each finite $z \subset \R^+$,  we define $\nu_z : = \nu \circ \mathrm{Rest}_z^{-1}$. We assume that for any subset $z\in \R$, $\nu_z$ is invariant for $(\mathrm{Rest}_z(\Pi^{\rho}_t))$.
We assume that $\Pi_0^{\rho}= \pi^0$ is distributed according to $\nu$.
 We want to prove that 
$$\forall B \in \F, \ \ \forall t \in \R^+, \ \Pp[\Pi^{\rho}_t \in B] = \Pp[\Pi_0^{\rho} \in B].$$
As $\F$ is the $\sigma$-field generated by ${\cal C}$, and ${\cal C}$ is closed under finite intersection, we only need to prove that for any  finite subset $z$ of $\R^+$,  $$ \forall \pi \in \pP_z, \  \forall t \in \R^+, \ \Pp[\mathrm{Rest}_z(\Pi^{\rho}_t) = \pi] = \Pp[\mathrm{Rest}_z(\Pi_0^{\rho}) = \pi].$$
As $\nu_z$ is invariant for $\mathrm{Rest}_z(\Pi^{\rho})$, both terms are equal to $\nu_z(\pi)$,
which completes the proof of Lemma \ref{lemma-invariant}.
\end{proof}

\begin{lemma}\label{prop:uniqQ}
 There exists a unique probability measure $\bar \mu^{\rho}$ on $(\pP_{\Q}, \F_{\Q})$ putting weight on right continuous partitions such that, for every finite $z  \subset \Q^+$, $$\mathrm{Rest}_{z} \star \bar \mu^{\rho}   = \mu^{\rho, z} .$$
Furthermore, $\bar \mu^\rho$ only puts weight on locally finite partitions of $\Q^+$ and for every $x \in \Q^+$, $$\bar \mu^\rho(\textrm{$x$ is the extremity of a segment}) = 0.$$
\end{lemma}

\begin{proof}
From Proposition \ref{consistency}, the family $(\mu^{\rho,z}; z\subset \Q^+)$ is consistent in the sense that 
for two finite subsets $z\subset z'$ then $\mathrm{Rest}_z \star \mu^{\rho,z'}  \ = \ \mu^{\rho,z}$.
By an application of  Kolmogorov's extension theorem, there exists a unique measure $\bar \mu^\rho$
defined on $(\pP_\Q, \F_\Q)$ such that 
for every finite subset $z$ in $\Q$ we have $\mathrm{Rest}_z \star \bar \mu^\rho \ =  \ \mu^{\rho,z}$.
(To see how one can apply Kolmogorov's theorem in the context of consistent random partitions, we refer the reader to \cite{beresticky}, Proposition 2.1.)

\smallskip 

We now need to prove that $\bar \mu^{\rho}$ only puts weight on locally finite partitions of $\Q^+$.  To do so, we follow closely \cite{WH}. We fix $a,b \in \N, \ a <b$. We want to prove that, if $\pi$ is a partition of $\Q$ distributed as $\bar \mu^\rho$, then $S_{[a,b]}$, the number of segments in $\pi|_{[a,b] \cap \Q}$ is finite almost surely. To do so, we define 
\begin{align*}
\forall n \in \N \setminus \{0\}, \ \epsilon_n &:= 2^{-n}, \ \\  
X_{in} & := \mathds{1}_{((a+(i-1)\epsilon_n) \not\sim (a+i \epsilon_n))}\\
z_{in} & := (a+(i-1)\epsilon_n,  a+i\epsilon_n) \in \R^2.
\end{align*}
In words, $X_{in} = 1$ if $(i-1)\epsilon_n$ and $i\epsilon_n$ belong to different segments.
Let us compute the expectation of  $S_{[a,b]}$. Using the monotone convergence theorem we have
\begin{align*}
\E[S_{[a,b]}] & = 1 + \E[\lim_{n\to \infty} \sum_{i=1}^{\lfloor2^n(b-a)\rfloor} X_{in}]
 = 1 +  \lim_{n\to \infty} \sum_{i=1}^{\lfloor2^n(b-a)\rfloor}  \E[X_{in}] \\
& = 1 +  \lim_{n\to \infty} \sum_{i=1}^{\lfloor2^n(b-a)\rfloor}  \mu^{\rho, z_{in}}(\{a+(i-1)\epsilon_n\},\{a+i \epsilon_n\}).
\end{align*}
The ARG at rate $\rho$ for the set of loci $z_{in}$ has only two types of transitions: coagulation at rate $1$ and fragmentation at rate $\rho \epsilon_n$, so \[ \mu^{\rho, z_{in}}(\{a+(i-1)\epsilon_n\}, \{a+i \epsilon_n\}) = \frac{\rho \epsilon_n}{1 +  \rho \epsilon_n}\] which gives
\begin{align*}
\E[S_{[a,b]}]
& = 1 +  \lim_{n\to \infty} \sum_{i=1}^{\lfloor2^n(b-a)\rfloor}  \frac{\rho 2^{-n}}{1 + \rho2^{-n}} = 1 + \rho (b-a).
\end{align*}
Then $S_{[a,b]}$ is finite almost surely, which implies that $\bar \mu^{\rho}$ only puts weight on locally finite partitions of $\Q$. 
  
\smallskip 

For the last statement let $x \in \Q^+$. By the previous argument, $$\bar \mu^\rho(\textrm{$x$ is the extremity of a segment}) = \lim_{\epsilon \downarrow 0} \bar \mu^\rho(x-\epsilon \not \sim x+\epsilon) = 0,$$
which completes the proof.
\end{proof}

\begin{lemma}
There exists a unique measure $\mu^{\rho}$ on $(\pP^{loc}, \F)$ such that 
\[ \mathrm{Rest}_\Q \star \mu^\rho \ = \ \bar \mu^\rho, \]
where $\bar \mu^\rho$ is the measure defined in Lemma \ref{prop:uniqQ}. \label{34}
\end{lemma}

\begin{proof}
Let $\tilde \pP^{loc}_\Q$ the set of locally finite partitions of $\Q$  such that for all $x \in \Q^+$, $x$ is not an extremity of a segment of $\pi$.  Note that here we do not assume that the partitions of $\Q$ are right continuous. 
From the previous Lemma, $\bar \mu^\rho(\tilde \pP^{loc}_\Q) = 1$.
Similarly, let $\tilde \pP^{loc}$ be the set of elements $\pi$ of  $\pP^{loc}$ such that for all $x \in \Q^+$, $x$ is not an extremity of a segment of $\pi$. 
Since $\Q$ is dense in $\R$, it is easy to see that for every $\bar \pi\in \tilde \pP^{loc}_\Q$ 
there exists a unique $\pi\in \tilde \pP^{loc}$ such that $\mathrm{Rest}_{\Q}(\pi) = \bar \pi$. In other words,  the projection map $\mathrm{Rest}_{\Q} : (\tilde \pP^{loc}, {\cal F}) \to (\tilde \pP^{loc}_\Q, {\cal F}_\Q)$
is bijective.  (Note that the condition that there are no rational extremities for the latter statement to hold, can be understood with the following counterexample. Let $\bar \pi$ be the partition of $\Q^+$ consisting of the two blocks $[0, 1] \cap \Q$ and $]1, +\infty) \cap \Q$. Then there is no right-continuous partition $\pi \in \pP^{loc}$ such that $\mathrm{Rest}_\Q(\pi) = \bar \pi$.) With a little bit of extra work, one can show that $\mathrm{Rest}_{\Q}^{-1}$ is measurable.
As already mentioned in the proof of Theorem \ref{unicityPi},  the projection map $\mathrm{Rest}_{\Q} : (\tilde \pP^{loc}, {\cal F}) \to (\tilde \pP^{loc}_\Q, {\cal F}_\Q)$
is bijective and measurable, 
 so the measure $\mu^\rho$ defined by
$ \mu^\rho = \mathrm{Rest}_{\Q}^{-1} \star  \left [ \bar \mu^\rho(\cdot \cap \tilde \pP^{loc}_\Q) \right ]$ has mass $1$ and satisfies 
$\mathrm{Rest}_\Q \star \mu^\rho \ = \ \bar \mu^\rho$.

To prove uniqueness, let $\mu$ on $(\pP^{loc}, \F)$ such that $ \mathrm{Rest}_\Q \star \mu \ = \ \bar \mu^\rho$. 
Because $\bar \mu^\rho$ only puts weight on  $\tilde \pP^{loc}_\Q$, 
$\mathrm{Rest}_\Q \star \mu  =  \bar \mu^\rho(\cdot \cap \tilde \pP^{loc}_\Q)$.
Because $\mu$ only puts weight on  right continuous partitions,   $\mu$ only puts weight on $\tilde \pP^{loc}$ (i.e., elements with no rational extremities). 
Taking the pushforward of the two members of the previous equality by $\mathrm{Rest}_\Q^{-1}$, we get 
$$\mu(\cdot \cap \tilde \pP^{loc}) =   \mathrm{Rest}^{-1}_\Q \star (\mathrm{Rest}_\Q\star \mu)  =  \mathrm{Rest}_{\Q}^{-1} \star  \left [ \bar \mu^\rho(\cdot \cap \tilde \pP^{loc}_\Q) \right ] = \mu^\rho.$$
Since $\mu$ only puts weight on  $\tilde \pP^{loc}$, $\mu = \mu^\rho.$
 \end{proof}

\begin{proof}[Proof of Theorem \ref{existencemeasure}]
We have proved that there exists a unique probability measure $\mu^{\rho}$ on $(\pP^{loc}, \F)$ such that, for any finite subset $z$ of $\Q^+$, $\mathrm{Rest}_z\star \mu^{\rho}$ is invariant for $\left(\mathrm{Rest}_z(\Pi^{\rho}_t); t \ge 0 \right)$ (by combining Lemmas \ref{prop:uniqQ} and \ref{34}). Using Lemma \ref{lemma-invariant}, we still need to prove that the same property holds for any finite subset $z\subset \R^+$.
This will be shown by a continuity argument.

\smallskip

We fix $\rho >0$. We denote by $\Pp^\rho$ the law of the process $(\Pi_t^\rho;  t \geq0)$, with initial condition $\Pi^\rho_0$ with law $\mu^\rho$.
We also fix $z = \{z_1, \ldots, z_n\} \subset \R^+$. 
For each $z^*= \{z^*_1, \ldots, z^*_n\} \subset \Q^+$, we define a function $g^*: \pP_{z^*} \to \pP_{z}$ such that, if $\pi$ is a partition of $z^*$, $g^*(\pi)$ is the partition of $z$ such that for every $i,j \in [n]$,  $z_i \sim_{g^*(\pi)} z_j$ iff $z^*_i \sim_{\pi} z^*_j$.
For every $t>0$, we define the event 
 $$A(z^*,t)  = \{\forall s \in [0,t], \ \mathrm{Rest}_z(\Pi_s^\rho) \ = \ g^*(\mathrm{Rest}_{z^*}(\Pi_s^\rho))\}.$$

 We want to prove that for every $t>0$ and for ${\cal F}_z$-measurable bounded function $f$ on $\pP_z$,
 $$\E^\rho[f(\mathrm{Rest}_z(\Pi_t^\rho))]  \ = \   \E^\rho[f(\mathrm{Rest}_{z}(\Pi_0^\rho))].$$
 As $\mu^\rho$ is a measure on $\pP^{loc}$, for every  $\epsilon >0$  one can  find $z^*= \{z^*_1, \ldots, z^*_n\} \subset \Q^+$ such that
$$\Pp^\rho \left [A(z^*,t)^\complement \right ] ||f||_{\infty} < \epsilon/2 \ \textrm{ and } \ \left |\E^\rho\left[f(g^*(\mathrm{Rest}_{z^*}(\Pi_0^\rho)))\mathds{1}_{A(z^*,t)^\complement}\right] \right | < \epsilon/2.$$
 Then
\begin{align*}
 \E^\rho[f(\mathrm{Rest}_z(\Pi_t^\rho))] \ &= \  \E^\rho[f(\mathrm{Rest}_z(\Pi_t^\rho)) \mathds{1}_{A(z^*,t)}] + \E^\rho[f(\mathrm{Rest}_z(\Pi_t^\rho))\mathds{1}_{ A(z^*,t)^\complement}]\\
 &= \  \E^\rho \left [f(g^*(\mathrm{Rest}_{z^*}(\Pi_t^\rho)))\mathds{1}_{A(z^*,t)} \right ] + \E^\rho[f(\mathrm{Rest}_z(\Pi_t^\rho))\mathds{1}_{A(z^*,t)^\complement}].
  \end{align*}
 As $z^* \subset \Q^+$, $\mu^\rho \circ \mathrm{Rest}_{z^*}^{-1}$ is invariant for $\mathrm{Rest}_{z^*}(\Pi_t^\rho)$,
 \begin{align*}
 \E^\rho[f(g^*(\mathrm{Rest}_{z^*}(\Pi_t^\rho)))\mathds{1}_{A(z^*,t)}]
 &=    \E^\rho[f(g^*(\mathrm{Rest}_{z^*}(\Pi_0^\rho)))] -   \E^\rho[f(g^*(\mathrm{Rest}_{z^*}(\Pi_0^\rho)))\mathds{1}_{A(z^*,t)^\complement}].   \end{align*}
 Then, 
 \begin{align*} & |\E^\rho[f(\mathrm{Rest}_z(\Pi_t^\rho))] -  \E^\rho[f(g^*(\mathrm{Rest}_{z^*}(\Pi_0^\rho)))]|  \\ 
 & \leq \Pp^\rho[A(z^*,t)^\complement] ||f||_{\infty}  +   |\E^\rho[f(g^*(\mathrm{Rest}_{z^*}(\Pi_0^\rho)))\mathds{1}_{A(z^*,t)^\complement})]| < \epsilon,
 \end{align*}
 and the conclusion follows by letting $\epsilon \to 0$.
\end{proof}

To conclude this section, we state an important property of $\mu^\rho$.
\begin{prop}[Scaling]
Fix $\rho >0$.
For every $\lambda >0$, define $h_\lambda: \R \to \R$ such that $\forall x \in \R, \ h_\lambda(x) = \lambda x$. Then \[h_\lambda \star \mu^\rho  \ = \ \mu^{\lambda \rho}.\]
\label{ARG-rescale}
Similarly, for any $z \subset R$, 
\[h_\lambda \star \mu^{\rho, z}  \ = \ \mu^{\lambda \rho, z}.\]
\end{prop}
\begin{proof}
This proposition can easily be deduced from the definition of the ARG and the scaling (\ref{scaling0}) and the construction of the $\R^+$-partitioning process given in the previous section.
\end{proof}
Without loss of generality, in Section \ref{results}, we will consider the partitioning process with recombination rate $\rho = 1$. 

\section{Proof of Theorem \ref{thm-stationary}}
\label{section4}
Theorem \ref{thm-stationary}  provides an approximation of the stationary measure of the discrete partitioning process when $\rho \to \infty$ or $\alpha \to \infty$, i.e. when recombination is much more frequent than coalescence.  
In the following, we fix $z = \{z_0, \ldots, z_n\}$ a finite subset of $\R$ and we assume that $\alpha >0$. We start by defining some notation. 

 \smallskip

If $\pi_1$ and $\pi_2$ are two partitions in $\pP_z$, we define $\theta(\pi_1, \pi_2)$ as the transition rate from $\pi_1$ to $\pi_2$ in the finite partitioning process  $\Gamma^{1, z}$ with recombination rate $\rho = 1 $.
By definition, in the ARG $\Gamma^{\rho, z}$ (with recombination rate $\rho$), the transition rate from $\pi_1$ to $\pi_2$ is $\theta(\pi_1, \pi_2)$ if the transition corresponds to a coagulation event and $\rho \theta(\pi_1, \pi_2)$ if it is a fragmentation. 
It can readily be seen that,
 \begin{equation*} \forall \pi \in \pP_z^r, \  \ \sum \limits_{\tilde \pi \in \pP_z^{r-1}} \theta(\pi, \tilde \pi) \ = \ C(\pi).  \end{equation*}
 In words, when $\rho=1$, the total fragmentation rate corresponds to the cover length. For general values of $\rho$, the fragmentation rate  is the cover length multiplied by $\rho$.

Also, the total coalescence rate from a partition of order $r$ only depends on $n$ and $r$ (and not in the values of $z_0, \ldots, z_n$ and $\rho$) and is given by 
\[ \sum \limits_{\tilde \pi \in \pP_z^{r+1}} \theta(\pi, \tilde \pi) =  \gamma_r :=\frac{(n-r)(n-r+1)}{2},\] 
 where $\gamma_r$ corresponds to the number of unordered pairs of blocks in a partition of order $r$. 
 
 \smallskip

Before proving Theorem \ref{thm-stationary}, we need to prove some technical results. But to give the reader some intuition on this result, we will start by giving a brief sketch of the proof.
Until further notice, we are going to fix $\rho>0, \ k \in [n]$ and  $\pi \in \pP^k_z$ a partition of order $k\ge1$.
We define
\begin{itemize}
\item[-]  $t_0^+ = \inf \{t>0, \ \Gamma^{\rho, z}_t \ne \pi_0\}, \ t_\pi^+ = \inf \{t>0, \ \Gamma^{\rho, z}_t \ne \pi\}$,
\item[-] $\mathcal{T}_{\pi} = \inf \{t> 0, \ \Gamma^{\rho, z}_t = \pi\}$,  $\mathcal{T}_{0} = \inf \{t>t_0^+, \ \Gamma^{\rho, z}_t = \pi_0\} $.
\item[-] $T_{\pi} = \inf \{t>t_\pi^+, \ \Gamma^{\rho, z}_t = \pi\}$,  $T_{\pi, 0} = \inf \{t>t_\pi^+, \ \Gamma^{\rho, z}_t = \pi_0\}$.
\item[-] $\Pp_{\pi}$ (resp $\Pp_{0}$) denotes the law of $\Gamma^{\rho,z}$  conditioned on the initial condition $\Gamma^{\rho, z}_0 = \pi$ (resp $\Gamma^{\rho, z}_0 = \pi_0$).
\end{itemize}
Recall that the variables defined above depend on $z$ and $\rho$, but for the sake of clarity this dependence is not made explicit.  

\smallskip 

The idea behind the proof of Theorem \ref{thm-stationary} is to use excursion theory and a well known extension of Blackwell's renewal theorem \citep{blackwell} that states that 
\begin{align}
\mu^{\rho, z}(\pi) &= \frac{\E_0(Y^{\pi}_1)}{\E_0(\Delta_0)},
\label{sketch}
\end{align}
where $\Delta_0$ is the time between two renewals at $\pi_0$ and $Y^{\pi}_1$ is the time spent in $\pi$ during an excursion out of $\pi_0$.  (More precise definitions of these variables will be given in the proof of  Theorem \ref{thm-stationary}).

As we consider that $\alpha \gg1$ or  $\rho \gg1$, fragmentation occurs much more often than coalescence so
$\pi_0$ is the most likely configuration and $\Gamma^{\rho,z}$ spends most of the time at $\pi_0$. Then
$\E_0(\Delta_0)$ can be approximated by the expectation of the holding time at $\pi_0$ which is $1/\gamma_0$.
Also, in this regime, most excursions out of $\pi_0$ will only visit $\pi$ at most one time, so $\E_0(Y^{\pi}_1)$ can be approximated by $$\Pp_0[\mathcal{T}_{\pi} < \mathcal{T}_0] \frac1{\rho C(\pi)},$$
where $\Pp_0[\mathcal{T}_{\pi} < \mathcal{T}_0]$ is the probability that $\pi$ is reached during the excursion out of $\pi_0$ and  $ \frac1{\rho C(\pi)}$  is approximately the expectation of the holding time at $\pi$ when $\rho C(\pi) \gg \gamma_k$ (i.e. when recombination occurs much more often than coalescence).

The core of the proof is to compute $\Pp_0[\mathcal{T}_{\pi} < \mathcal{T}_0]$. (This will be done in Proposition \ref{prop-proba}.) To do so, we will consider $\bar \Gamma^{\rho,z}$, the embedded chain of  the ARG $\Gamma^{\rho,z}$ and we will study the law of $\bar \Gamma^{\rho,z}$ conditioned on the initial condition $\bar \Gamma^{\rho, z}_0 = \pi_0$. 
We call a ``direct path'' a trajectory that goes from $\pi_0$ to $\pi$ in only $k$ coalescence steps (without recombination events). Indirect paths  are trajectories that are longer and that contain at least one recombination event (and therefore more coalescence steps).
As we consider a high recombination regime, where coalescence occurs much less often than recombination, direct paths will be much more likely that indirect paths (this will be formalized in Lemma  \ref{lemma-proba}), and even if there are more indirect paths, their total contribution to $\Pp_0[\mathcal{T}_{\pi} < \mathcal{T}_0]$ will be negligible. (See the RHS of equation (10) and the subsequent paragraph below where we explicitly separate the contribution of direct and indirect paths and show that the latter is much smaller than the former.) In some sense, our computations are reminiscent of the method of large deviations
One of the `golden' formulations of this theory states that  ``any large deviation is done in the least unlikely of all the unlikely ways'' (\cite{denHollander} p. 10).

We can then approximate
$\Pp_0[\mathcal{T}_{\pi} < \mathcal{T}_0]$ by the sum of the probabilities of the direct paths.
Then the conclusion will follow by realizing that a direct path corresponds to a scenario of coalescence and showing that 
$\Pp_0[\mathcal{T}_{\pi} < \mathcal{T}_0]$ can be approximated by  $\frac{C(\pi)}{\rho^{k-1} \gamma_0} F(\pi)$. (This will be formalized  in Proposition \ref{prop-proba}.)
Finally, replacing $\E_0(Y^{\pi}_1)$ and $\E_0(\Delta_0)$ by their approximations in \eqref{sketch}, we find that $\mu^{\rho,z}(\pi)$ can be approximated by  $\frac{F(\pi)}{\rho^k} $.

\smallskip

 Before turning to the formal proof of Theorem \ref{thm-stationary}, we start by proving some technical results. 
  We consider  $\bar \Gamma^{\rho,z}$, the embedded chain of  the  ARG $\Gamma^{\rho,z}$.  Let $P_0$ denote the law of  $\bar \Gamma^{\rho,z}$ conditioned on $\bar \Gamma^{\rho, z}_0 = \pi_0$ and $\forall \pi' \in \pP_z, \ P_{\pi'}$ denotes the law of  $\bar \Gamma^{\rho,z}$ conditioned on $\bar \Gamma^{\rho, z}_0 = \pi'$.
  We will consider  paths that go from $\pi_0$ to $\pi$. Sets of paths are defined as follows.
 \begin{definition}
For $j \in \N \setminus \{0\}, \pi', \pi'' \in \pP_z$, we define:
\begin{align*}
G(j, \pi' \to \pi'') \ = \ \{ ( \pi^{(0)}=\pi', \pi^{(1)}, \ldots, \pi^{(j-1)}, \pi^{(j)}=\pi ''),  \\  \pi^{(1)}, \ldots, \pi^{(j-1)} \in \pP_z \setminus \{ \pi', \pi ''\}  
\mbox{ such that } \\ \ \theta(\pi^{(i)},\pi^{(i+1)}) >0 \ \ \forall i\in\{0,\cdots,j-1\}   \}.
\end{align*}
In words, $G(j, \pi' \to \pi'')$ contains every possible path (admissible for the partitioning process) that connects $\pi'$ to $\pi''$ in $j$ steps. 

For a path $p$, its length is defined as the number of steps and is denoted by $|p|$. If $p \in G(j, \pi' \to \pi'')$, $|p| = j$.
\end{definition}
For example, a path $p := \pi_0 \to \pi_1$ is of length $|p| = 1$ and a path $p:=\pi_0 \to \pi_1 \to \pi_2 \to \pi_3$  is of length $|p| = 3$.
 We are going to consider  paths $p$ that go from $\pi_0$ to $\pi$, that have at least $k$ steps  (as $\pi$ is of order $k$). 
 \begin{itemize}
 \item[-] $p$ is a \textit{direct} path if  $|p|=k$, i.e., $p$ can only  be composed of coalescence events.
\item[-] $p$ is an  \textit{indirect} path if  $|p|=k+N, \ N \in \N \setminus \{0\}.$ Indirect paths contain at least one recombination event. 
Note that the parity of the process implies that $G(k+2N+1, \pi_0 \to \pi)$ is empty.
 \end{itemize}

In the next lemma, we compare the probability of an indirect path to the probability of a direct path, and show that the latter is much more likely than the former.
\begin{lemma}
Fix $N \in \N \setminus \{0\}$. For every  path $p \in G(k+2N, \pi_0 \to \pi)$ there exists a path $\hat p \in G(k, \pi_0 \to \pi)$ such that 
\begin{align*}
\ \ \ \ \ \frac{P_0(p)}{P_0(\hat p)} \le  \left (\frac{(1 + \frac{\gamma_1}{\rho \alpha})^n}{\alpha \rho} \right)^N.
\end{align*}
\label{lemma-proba}
\end{lemma}

\begin{proof}[Proof of Lemma \ref{lemma-proba}]
We fix $N \in \N \setminus \{0\}$ and we start with proving that 
\begin{align}
 \forall ^np \in G(k+2N, \pi_0 \to \pi),  \exists \ \hat p \in G(k+2(N-1), \pi_0 \to \pi), 
  \frac{P_0(p)}{P_0(\hat p)} \le  \frac{(1 + \frac{\gamma_1}{\rho \alpha})^n}{\alpha  \rho}.
\label{claim1}
\end{align}

We consider any path $p \in G( k+2N, \pi_0 \to \pi)$ and denote it as
\begin{align*} 
p =  ( \pi_0,  \bar \pi_1, \ldots,  \bar \pi_{j},  \hat \pi_{j-1}, \pi_{i_1}, \pi_{i_2},  \ldots,  \pi   ),
\end{align*}
where the indices of the $\hat \pi,\bar \pi$'s coincide with the order of the partition (for instance, in the transition  $\bar \pi_{j}\to \hat \pi_{j-1}$, the order of the partition decreases by one unit, which corresponds to a fragmentation event).  We do not specify the order of $ \pi_{i_1}, \pi_{i_2}, \ldots$. As $N \ge 1$ there is at least one recombination event ($\bar \pi_{j}\to \hat \pi_{j-1}$). The step where the first recombination event occurs is $j+1$, where necessarily $j\le n$.
Note that any path in $ G( k+2N, \pi_0 \to \pi)$ can be written this way, as the only thing we have assumed is that there is at least one recombination event and we have not made any hypothesis on the order of $ \pi_{i_1}, \pi_{i_2}, \ldots$.
The path $p$ can be decomposed into $p_1$ and $p_2$ such that
\begin{align*} 
p_1 \in G( (j-1)+2, \pi_0 \to \hat \pi_{j-1}), \ \  &p_1 =  (\pi_0,  \bar \pi_1, \ldots, \bar \pi_{j-1}, \bar \pi_{j},  \hat \pi_{j-1} ) \\
p_2 \in G(k+2N - (j-1) -2, \hat \pi_{j-1} \to \pi), \ \ &p_2 =   (   \hat  \pi_{j-1},  \pi_{i_1}, \pi_{i_2},  \ldots,  \pi  ).
\end{align*}
In words, we decompose $p$ into two paths, $p_1$ that goes from $\pi_0$ until the first recombination event and $p_2$ that contains the rest of the path.

\smallskip

The idea now is to find a direct path $\hat p_1 \in G( j-1, \pi_0 \to \hat \pi_{j-1})$, denoted by $\hat p_1 =  ( \pi_0,  \hat  \pi_1, \ldots, \hat \pi_{j-1})$
such that
\begin{align*}
\frac{P_0(p_1)}{P_0(\hat p_1)} \le \frac1{\alpha  \rho}  \left(1 + \frac{\gamma_{1}}{ \rho\alpha}\right)^n.
\end{align*}

To do so, consider the fragmentation event that occurs between step $j$ and step $j+1$ in $p$ (when transitioning from $\bar \pi_{j}$ to $\hat \pi_{j-1}$).
$\bar \pi_j$ contains $n+1-j$ blocks and
let $(b_1, \ldots  b_{n-j}, b^*)$ be the blocks of $\bar \pi_j$ such that
 $b^*$ is  the block of $\bar \pi_j$ that is fragmented during this fragmentation event and $z_a <z_b$ the two elements of $b^*$ such that $b^*$ is fragmented between $z_a$ and $z_b$  (i.e. such that $b^*$ is fragmented into $b^*_a$ and $b^*_b$ where  $z_a$ is the rightmost element in $b^*_a$ and $z_b$ the leftmost element in $b^*_b$).
We have 
\begin{equation}\label{eq:incl0}
C(\bar \pi_j ) = C(\hat \pi_{j-1}) + z_b - z_a.
\end{equation}
Let $i^*\leq j$ be the first step of $p$ such that $z_a$ and $z_b$ are in the same block, i.e
\[ i^* = \min_{i \in [j]} \{i, \  z_a \sim_{\bar \pi_i} z_b\}.\]

We will construct  a direct path $\hat p_1 = (\hat \pi_0,\cdots,\hat \pi_{j-1})$ in such a way that 
\begin{align}
\forall \ 1 \le i < i^*, \ \ \  & C(\hat \pi_{i}) \ \le \ C(\bar \pi_i) \nonumber \\ 
\textrm{if } i^*<j-1, \ \forall \ i^*< i \le j, \ \ \ &C(\hat \pi_{i-1}) \ \le \ C(\bar \pi_i). \label{eq:incl}
\end{align}
(Note that the terminal value of $\hat p_1$ coincides with the terminal value of $p_1$  and its  length is $j-1$ instead of $j+1$).
See Figure \ref{chemins} for a concrete example. In words, we skip step $i^*$,
and rearrange the path in such a way that $\hat p_1$ is admissible, ends at $\hat \pi_{j-1}$ and the inequalities (\ref{eq:incl}) are satisfied along the way.
Formally, 
the path $\hat p_1$ is constructed as follows : 
\begin{itemize}
\item If $i^*<j-1$, for $i \in \{i^*+1, \ldots, j-1\}$, let $(b^i_1, \ldots  b^i_{n-i}, b^i_*)$ be the blocks of $\bar \pi_i$, where $ b^i_*$ is the one that contains $z_a$ and $z_b$.   
 The blocks of $\hat \pi_{i-1}$ are $(b^i_1, \ldots,  b^i_{n-i}, b^i_{n-i+1}, b^i_{n-i+2})$ such that:
\begin{itemize}
\item[-] if $z \in b^i_*$ and $z \le z_a$, $z \in b^i_{n-i+1}$.
\item[-] if $z \in b^i_*$ and $z \ge  z_b$, $z \in b^i_{n-i+2}$.
\end{itemize}
If $i^*=j-1$ we skip the present step in the construction of $\hat p$.
\item  If in $\bar \pi_{i^*-1}$,  $z_a$ is the rightmost element in its block and $z_b$ the leftmost element in its block, then we define $(\hat \pi_1, \ldots, \hat \pi_{i^*-1}) = (\bar \pi_1, \ldots, \bar \pi_{i^*-1}) $. With this construction $ \hat \pi_{i^*}$ can be obtained from $ \hat \pi_{i^*-1}$ by a coalescence event, so the path $\hat p$ is admissible for $\Gamma^{\rho, z}$.
\item Else, $(\hat \pi_1, \ldots, \hat \pi_{i^*-1})$ are constructed from $ (\bar \pi_1, \ldots, \bar \pi_{i^*-1}) $ in the following way. Let us denote by $b_a$ and $b_b$ the blocks of $\bar \pi_{i^*-1}$ that contain $z_a$ and $z_b$ respectively. For $1\le i \le {i^*-1}$,
\begin{itemize}
\item[-] If the coalescence event between $\bar \pi_{i-1}$ and $\bar \pi_i$ involves two blocks $b_c$ and $b_d$ such that  in $\bar \pi_{i^*-1}$,  $b_c, b_d \subset b_a$ (resp. $b_c, b_d \subset b_b$) and if $b_c$ contains an element that is smaller than $z_a$ and $b_d$ contains an element is larger than or equal to $z_b$, then in the coalescence step  between  $ \hat \pi_{i-1}$ and $ \hat \pi_i$,  $b_c$ (resp. $b_d$) coalesces with the block containing $z_a$ (resp. $z_b$). (And nothing happens to $b_d$ - resp. $b_c$).
\item[-] Otherwise the same coalescence event occurs  between  $\bar \pi_{i-1}$ and $\bar \pi_i$ and  between $\hat \pi_{i-1}$ and $\hat \pi_i$.
\end{itemize}
With this construction $ \hat \pi_{i^*}$ can be obtained from $\hat \pi_{i^*-1}$ by a coalescence event, and as a consequence the path $\hat p$ is admissible, in the sense that $\theta(\pi_i,\pi_{i+1})>0$ (see Figure \ref{chemins} for an example).
\end{itemize}

\smallskip 

\begin{figure}
\begin{center}
\includegraphics[width = 12cm]{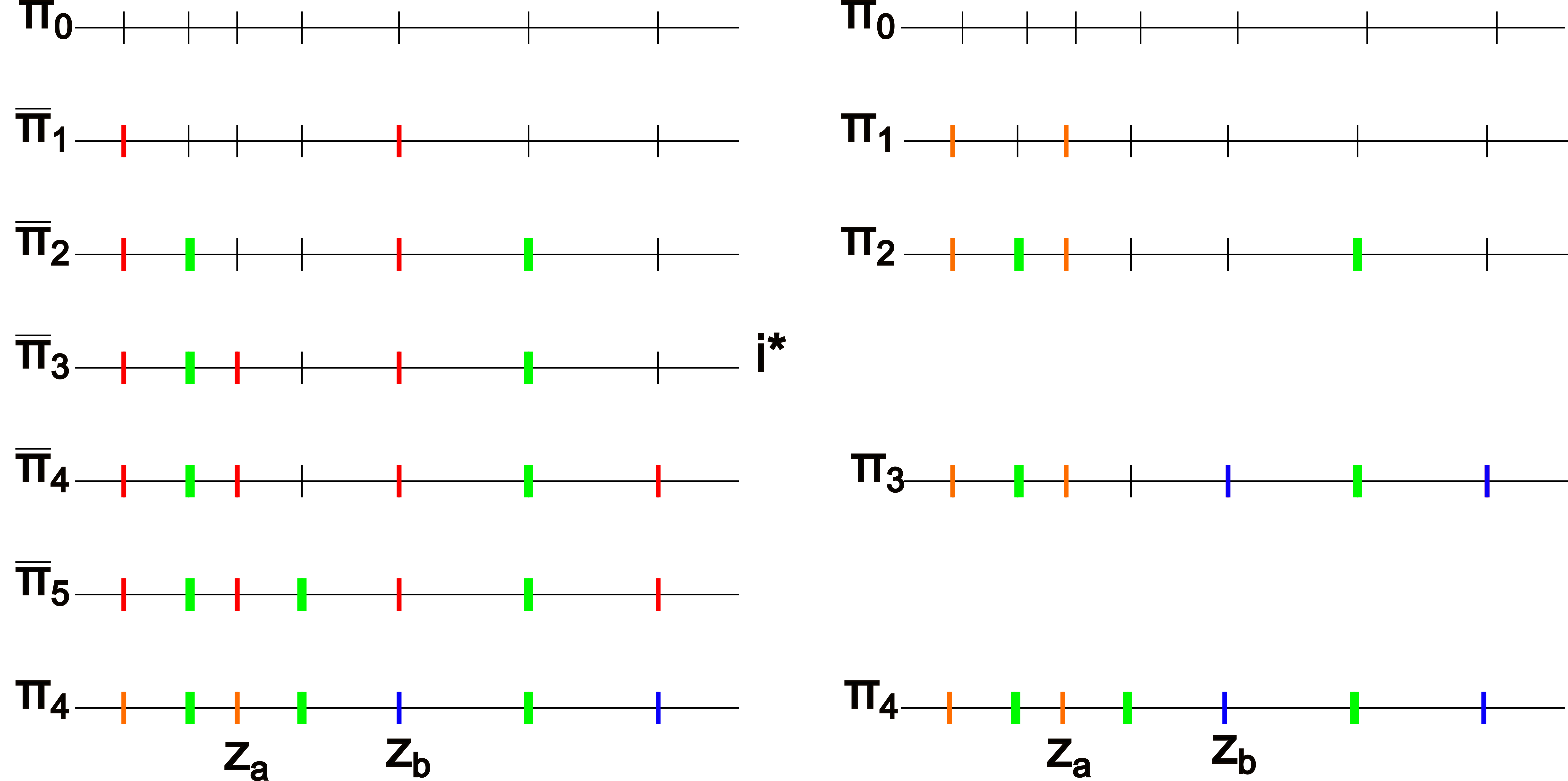}
\end{center}
\caption{Example of two paths that go from $\pi_0$ to $\pi_4$, for $n=7$. Loci in the same block are of the same color and the black loci corresponds to loci that are in singleton blocks. The path on the left corresponds $p_1 \in G(4 + 2, \pi_0 \to \pi_4)$ and the path on the right is  $\hat p_1 \in G(4, \pi_0 \to \pi_4)$ constructed from $\pi$ with the method presented above. }
\label{chemins}
\end{figure}

First,
\begin{align*}
P_0(\hat p_1) & = \frac1{\gamma_0 \rho^{{|\hat p_1|}-1}} \prod_{i = 1}^{{|\hat p_1|}-1}\frac1{C(\hat \pi_i) + \gamma_i/ \rho} \\
P_0(p_1) & =  \frac1{\gamma_0 \rho^{{|p_1|}-2}} \left ( \prod_{i = 1}^{{|p_1|}-2} \frac1{C(\bar \pi_i) + \gamma_i/  \rho} \right )\ \frac{ \rho(z_b - z_a)}{ \rho C(\bar \pi_j)+ \gamma_j},
\end{align*}
and from now we use the convention that a product running from 1 to 0 is equal to 1. Recall that $|\hat p_1| = j-1$ and $|p_1|=j+1$.
From (\ref{eq:incl}), if  $i^*<j-1$
\begin{align*}
\frac{P_0(p_1)}{P_0(\hat p_1)} & =  \frac1{ \rho} \frac1{C(\bar \pi_{i^*}) + \gamma_{i^*}}  \prod_{i = 1}^{i^*-1}\frac{C(\hat \pi_i) + \gamma_i/ \rho }{C(\bar \pi_i) + \gamma_i/ \rho}  \prod_{i = i^*+1}^{j-1}\frac{C(\hat \pi_{i-1}) + \gamma_{i-1}/ \rho}{C(\bar \pi_i) + \gamma_i/ \rho} \ \frac{ \rho(z_b - z_a)}{ \rho C(\bar \pi_j)+ \gamma_j} \nonumber \\
& \le  \frac1{\alpha \rho}  \prod_{i = i^*+1}^{j}\frac{C(\bar \pi_{i}) + \gamma_{i-1}/ \rho}{C(\bar \pi_i) + \gamma_i/ \rho}    \le  \frac1{\alpha  \rho}  \prod_{i = i^*+1}^{j-1}\frac{1 + \frac{\gamma_{i-1}}{ \rho C(\bar \pi_i) }}{1 +  \frac{\gamma_{i}}{ \rho C(\bar \pi_i) }} \nonumber \\
 &\le   \frac1{\alpha  \rho}  \prod_{i = i^*+1}^{j-1} (1 + \frac{\gamma_{i-1}}{  \rho C(\bar \pi_i) }) \le   \frac1{\alpha  \rho}  \left(1 + \frac{\gamma_{1}}{ \rho \alpha}\right)^n.
\end{align*}
where the second inequality is a consequence of (\ref{eq:incl0}) and (\ref{eq:incl}).
The case  $i^*=j-1$ follows along the same lines.

\smallskip

Let us define
\begin{align*} 
\hat p \in G(k+2, \pi_{0} \to \pi), \ \ & \hat p =   (   \pi_0, \hat  \pi_{1}, \ldots, \hat \pi_{j-1}, \pi_{i_1}, \pi_{i_2} \ldots,  \pi  ).
\end{align*}
Since $P_0(p)  \ = \ P_0(p_1) \ P_{\hat \pi _{j-1}}(p_2)$ and $P_0(\hat p) \  = \ P_0(\hat p_1) \ P_{\hat \pi _{j-1}}(p_2)$, we have
\begin{align*}
\frac{P_0(p)}{P_0(\hat p)} = \frac{P_0(p_1)}{P_0(\hat p_1)} \le \  \frac{\left(1 + \frac{\gamma_{1}}{\rho \alpha}\right)^n}{\alpha \rho},
\end{align*}
which completes the proof of \eqref{claim1}.
Lemma \ref{lemma-proba} then follows by a simple induction on $N$ using \eqref{claim1}.
\end{proof}

\begin{prop}
There exists a function $$u^n: \R^+ \setminus \{0\}\to \R^+ \cup \{\infty\}, \  \ \lim_{x \to \infty} u^n(x) = 0,$$ independent of the choice of $z$ (which is a set of cardinality $n$), $\pi$ and $\rho$, such that
\begin{equation*}
\left | \Pp_0[\mathcal{T}_{\pi} < \T_0] \ - \ \frac{C(\pi)}{\rho^{k-1} \gamma_0} F(\pi) \right |  \le  u^n(\alpha R)    \frac{C(\pi)}{\rho^{k-1} \gamma_0} F(\pi),
\end{equation*}
where $F$ is defined in \eqref{def:F}. 
\label{prop-proba}
\end{prop}

\begin{proof} 
As $\pi$ is of order $k$, a path from $\pi_0$ to $\pi$ has at least $k$ steps. In addition, as the order of the partition can only increase or decrease by $1$ at each step, a path from  $\pi_0$ to $\pi$  can only have $k + 2N$ steps, with $N \ge 0$, so
\begin{align}
\Pp_0[\T_{\pi} < \T_0] = \sum_{p \in G(k, \pi_0 \to \pi) } P_0(p) + \sum_{N\ge1} \sum_{p \in G(k + 2N, \pi_0 \to \pi) } P_0(p).
\label{somme}
\end{align}
Note that in the previous equation, we explicitly separate the LHS into two contributions: the one coming from direct paths, and the one coming from indirect paths. Using the previous estimates, our goal is to show that even if the second contribution contains more terms (high entropy), it is negligible with respect to the first one. 

We start by considering the first term in the right hand side. We consider a path $p  \in G(k, \pi_0 \to \pi)$ and we denote it by $p = (\pi_0, \pi_1, \ldots, \pi)$.
We have
\begin{align*}
&P_0(p)  = \frac1{\gamma_0} \prod_{i = 1}^{k-1} \frac1{ \rho C(\pi_i) + \gamma_i}. 
\end{align*}
Further, paths that have $k$ steps are only composed of coalescence events, and therefore $G(k, \pi_0 \to \pi)={\cal S}(\pi)$. It follows that
\begin{align*}
\sum_{p \in G(k, \pi_0 \to \pi) }P_0(p) -  \ \frac{C(\pi)}{ \rho^{k-1} \gamma_0} F(\pi)  \  & =  \ \frac{1}{ \rho^{k-1}\gamma_0} \ \sum_{s \in {\cal S}(\pi)} \left( \prod_{i = 1}^{k-1} \frac1{C(s_i) + \gamma_i/ \rho} -  \prod_{i = 1}^{k-1} \frac1{C(s_i)} \right) 
\end{align*}
 and using the fact that $\gamma_i \leq \gamma_0$, 
\begin{align*}
 \sum_{s \in {\cal S}(\pi)}  \left | \prod_{i = 1}^{k-1} \frac1{C(s_i) + \gamma_i/ \rho} -  \prod_{i = 1}^{k-1} \frac1{C(s_i)} \right| &= 
 \sum_{s \in {\cal S}(\pi)} \prod_{i = 1}^{k-1} \frac1{C(s_i)}  \left (1-\prod_{i = 1}^{k-1}\frac1{1 + \gamma_i/( \rho C(s_i))}  \right) \\
& \le  \sum_{s \in {\cal S}(\pi)} \prod_{i = 1}^{k-1} \frac1{C(s_i)}  \left (1-\left(\frac1{1 + \gamma_0/(  \rho\alpha)}  \right)^{k-1}\right) \\
&\le C(\pi) F(\pi)  \left (1-\frac1{(1 + \gamma_0/( \rho \alpha))^{k-1}}  \right)
\end{align*}
so  
\begin{align}
\left|\sum_{p \in G(k, \pi_0 \to \pi) }P_0(p) -   \frac{C(\pi)}{\rho^{k-1} \gamma_0} F(\pi) \right |    \le   \frac{C(\pi)}{\rho^{k-1}\gamma_0} F(\pi)    \left (1-\frac1{(1 + \gamma_0/( \rho \alpha))^{k-1}}  \right).
\label{Gk}
\end{align}

\smallskip

To prove Proposition  \ref{prop-proba}, we still need to consider the second term in the right hand side of \eqref{somme}. Using Lemma \ref{lemma-proba}, we have
\begin{align}
& \sum_{N\ge1} \sum_{p \in G(k + 2N, \pi_0 \to \pi) } P_0(p) \nonumber \\ & \ \ \le  \sum_{\hat p \in G(k , \pi_0 \to \pi) } P_0(\hat p) \  \left ( \sum_{N\ge1} | G(k + 2N, \pi_0 \to \pi) | \left ( \frac{(1 + \frac{\gamma_1}{\rho \alpha})^n}{\alpha  \rho} \right)^N  \right ). 
\label{Pp}
\end{align}

To compute $ | G(k + 2N, \pi_0 \to \pi) |$, let us recall that, at each step in a path:
\begin{itemize}
\item If it corresponds to a coalescence event from a partition of order $j$ there are $\gamma_j$ possibilities, and $\forall j \in \{0, \ldots, n \} \ \gamma_j \le n(n+1)$.
\item If it corresponds to a fragmentation event, there are at most $(n+1)$ blocks in the partition and each one contains at most $(n+1)$ elements, so that each block can be fragmented in $n$ different ways. 
\end{itemize}
From there, it can easily be seen that  
$$ | G(k + 2N, \pi_0 \to \pi) | \le (n(n+1))^{k+2N}.$$ 

Combining this with \eqref{Pp}, we have:
\begin{align}
\label{eq:blabla}
& \sum_{N\ge1} \sum_{p \in G(k + 2N, \pi_0 \to \pi) } P_0(p) \nonumber \\ &\le  \left (\sum_{ p \in G(k,  \pi_0 \to \pi) } P_0( p) \right ) (n(n+1))^{k} \sum_{N\ge1} \left( \frac{n^2(n+1)^2(1 + \frac{\gamma_1}{\rho \alpha})^n}{\alpha \rho}\right)^N,
\end{align}
where, for $\alpha\rho$ large enough, the sum in the RHS converges.
This result,  combined with \eqref{somme} and \eqref{Gk}, gives: 
\begin{align*}
\left |\Pp_0[\T_{\pi} < \T_0] - \frac{C(\pi)}{\rho^{k-1}\gamma_0} F(\pi)  \right | \le  \ \frac{C(\pi)}{\rho^{k-1}\gamma_0} F(\pi)   u^{n,k}(\alpha \rho)
\end{align*}
where $u^{n,k}$ only depends on $z$ via its cardinality $n$ and only depends on $\alpha$ and $\rho$ via the product $\alpha \rho$
and  vanishes at $\infty$. 
The conclusion follows by setting $u^n(\alpha \rho) = \max_{k \in [n]}(u^{n,k}(\alpha \rho))$.
\end{proof}

\begin{lemma}
For any $n \in \N$, there exist two functions $g^n$ and $h^n$ such that
\begin{align*}
 \  \ \lim_{x \to \infty} g^n(x) = 0, \ \lim_{x \to \infty} h^n(x) = 1 
 \end{align*}
independent of the choice of $z,\pi$ and $\rho$
such that 
\begin{align*}
&(i) \  \  \E_{0}[\T_0 - t_0^+] \le g^n(\alpha \rho) \\
&(ii) \ \forall k>0, \forall \pi \in \pP^k_z, \ \Pp_{\pi}[T_{\pi, 0} < T_\pi] \ge  h^n(\alpha \rho).
\end{align*}
\label{lemma-couplage}
\end{lemma}
\begin{proof}
We fix $\rho>0, n \in \N, z = \{z_0, \ldots, z_n\}, \ k \in [n], \pi \in \pP^k_z$.
The idea of the proof is to consider the stochastic process $(X^{\rho, z}_t; t \geq 0)$ valued in $\{0, \ldots, n\}$ and such that $\forall t \ge 0$,  $X^{\rho, z}_t$ is the order of the partition $\Gamma^{\rho, z}_t$. This process is not Markovian, but it can easily be compared to a Markov process $(W^{\rho,z}_t; t \geq0)$ 
in such a way  that the excursions out of $0$ of $W^{\rho, z}$  are longer than those of $X^{\rho, z}$. 
  
\smallskip 

More precisely, let $W^{\rho, z}$ be the birth-death process in $\{0,\cdots,n\}$ where all the death rates are equal to $\rho \alpha$ and the birth rate at state $k$ is $\gamma_k$ (note that $\gamma_n = 0$).
With these transition rates, for any $\pi_k \in \pP_z^k$, the total coalescence rate from $\pi_k$ for the process $\Gamma^{\rho, z}_t$ is the same as the birth rate from $k$ for $W^{\rho, z}_t$. On the other hand, the total fragmentation rate for $\Gamma^{\rho, z}$ when $\Gamma^{\rho, z}_t = \pi_k$   is equal to $\rho C(\pi_k)$ and is always higher than the death  rate at $k$ for $W^{\rho, z}_t$. 
We can find a coupling between $W^{\rho, z}$ and  $X^{\rho,z}$ such that the holding times at $0$ of the two process are the same (as the birth rate in $0$ for $W^{\rho, z}$ is the same as the coagulation rate from $0$ for $\Gamma^{\rho, z}$). In addition, during an excursion out of $0$, the holding time at $k>0$ for $X^{\rho, z}$ is shorter than the holding time at $k$ for $W^{\rho, z}$ and the embedded chain of $X^{\rho, z}$ jumps more easily to the right than the embedded chain of $W^{\rho, z}$.

\smallskip 

Let us denote by $\bar \E_0$ the  probability with respect to the distribution of $W^{\rho, z}$, conditional to $W^{\rho,z}_0 = 0$, and define 
\begin{align*}
\bar t_0^+ \ = \ \inf \{t>0, W^{\rho, z} \ne 0 \}  \  \textrm{and} \ 
\bar \T_0 \  = \ \inf \{t> \bar t_0^+, \ W^{\rho, z}_t = 0\}.
\end{align*}
By construction, we have 
\begin{align*}
\E_{0}[\T_0- t_0^+]  \ \le  \ \bar \E_0[\bar \T_0 - \bar t_0^+]. 
\end{align*}
Finally, $ \bar \E_0[\bar \T_0- \bar t_0^+]$ only depends on $\rho \alpha$ and $n$ and it can be checked that it converges to 0 as $\alpha \rho$ tends to infinity, so $(i)$ is verified. 

\smallskip 

(ii) can be handled by similar methods. Namely, let $\bar W^{\rho, z}$ denote the embedded chain of $W^{\rho, z}$
\begin{align*}
  \Pp_{\pi}[T_{\pi,0} < T_\pi] &\ge  \Pp[\bar W^{\rho, z}_0 =k, \bar W^{\rho, z}_1 = k-1, \ldots, \bar W^{\rho, z}_k = 0]\\
& = \prod_{i=1}^{k} \frac{\rho \alpha}{\rho \alpha + \gamma_i} \ge \prod_{i=1}^{n} \frac{\rho \alpha}{\rho \alpha + \gamma_i}   \underset{\rho \alpha \to  \infty}{\longrightarrow} 1 
\end{align*}
where the first  inequality is obtained by the same  argument as in (i).
This completes the proof of Lemma \ref{lemma-couplage}.
\end{proof}

We are now ready to prove the main result of this section.

\begin{proof}[Proof of Theorem \ref{thm-stationary}]

We will consider excursions of $\Gamma^{\rho, z}$ out of $\pi_0$.
Let us consider $(J_i)_{i \in \N} $ the renewal times at $\pi_0$ i.e. the successive jump times of $\Gamma^{\rho, z}$ such that $\Gamma^{\rho, z}_{J_i} = \pi_0$ (and $\Gamma^{\rho, z}_{J_i -} \ne \pi_0$).
For $i \in \N \setminus \{0\}$, let us define $$\Delta_0^i := J_i - J_{i-1}$$ 
 the time between two renewals at $\pi_0$.  The $(\Delta_0^i)_{i \in \N}$ are independent and identically distributed random variables.
Also, for $i \in \N \setminus \{0\}$, consider
\begin{align*}
Y^{\pi}_i \ & := \int_{J_{i-1}}^{J_i} \mathds{1}_{\{\Gamma^{\rho, z}_t = \pi\}} dt,
\end{align*}
which corresponds to  the time spent by $\Gamma^{\rho, z}$ in $\pi$ during the $i^{th}$ excursion out of $\pi_0$. By standard excursion theory,
the $Y^{\pi}_i$'s are independent and identically distributed random variables.
From the ergodic theorem  we have
\begin{align*}
\mu^{ \rho, z}(\pi) &= \lim_{T \to \infty} \frac1{T} \int_{0}^T \mathds{1}_{\{\Gamma^{\rho,z}_s = \pi \}} ds  \ \  \ \  \mbox{a.s.}\\
&  = \lim_{n \to \infty} \frac1{J_n}  \sum_{k=1}^{n} \int_{J_{n-1}}^{J_n}  \mathds{1}_{\{\Gamma^{\rho,z}_s = \pi \}} ds \ \  \mbox{a.s.}.
\end{align*} 
Since the excursions are independent from one another, using Blackwell's renewal theorem (\cite{blackwell}) and the law of large numbers, 
\begin{equation*}
\lim_{n \to \infty} \frac{n}{J_n} = \frac1{\E_0[\Delta^1_0]} \ \textrm{ and } \ \lim_{n \to \infty} \frac1n \sum_{k=1}^{n} \int_{J_{n-1}}^{J_n}  \mathds{1}_{\{\Gamma^{\rho,z}_s = \pi \}} ds  = \E_0[Y^{\pi}_1] \ \ \mbox{a.s.},
\end{equation*}
which easily gives
\begin{align*}
\mu^{ \rho, z}(\pi) &= \frac{\E_0[Y^{\pi}_1]}{\E_0[\Delta^1_0]}.
\end{align*}

Let $H_0$ be the holding time at $\pi_0$ and $H$ the holding time at $\pi$.
$H_0$ follows an exponential distribution of parameter $\gamma_0$. $H$ follows an exponential distribution of parameter $(\rho C(\pi) + \gamma_k)$.
By standard excursion theory,
\begin{align*}
\E[Y^{\pi}_1] \ & = \ \Pp_0[\T_\pi < \T_0]  \sum_{k\ge1}  k  \E[H] \Pp_\pi[T_{\pi} < T_0]^{k-1}  \Pp_\pi[ T_0 < T_{\pi}]   \\ & = \  \Pp_0[\T_\pi < \T_0] \frac1{( \rho C(\pi) + \gamma_k) } \frac{1}{\Pp_\pi[T_{0} < T_{\pi}]},
\end{align*}
and 
\begin{align*} 
\E[\Delta_0^1] = \E(H_0) + \E_{0}[\T_0 - t^+_0], \  \textrm{where } \ \E[H_0] = \frac1{\gamma_0}.
\end{align*}
Combining this with Proposition \ref{prop-proba}, we have
 \begin{align*}
&\left | \ \mu^{ \rho, z}(\pi) \ - \ \frac{C(\pi)F(\pi)}{\gamma_0  \rho^{k-1}} \frac1{( \rho C(\pi) + \gamma_k) } \frac{1}{\Pp_\pi[T_{0} < T_{\pi}]} \ \frac1{1/\gamma_0 +   \E_{0}[\T_0- t^+_0] }  \ \right | \\ \ \  
 & \ \ \ \ \le \ \  u^n(\alpha  \rho) \frac{C(\pi)F(\pi)}{\gamma_0  \rho^{k-1}} \frac1{( \rho C(\pi) + \gamma_k) } \frac{1}{\Pp_\pi[T_{0} < T_{\pi}]} \ \frac1{1/\gamma_0 +   \E_{0}[\T_0- t^+_0]} \\
& \ \ \ \ \le \ \  u^n(\alpha  \rho) \  \frac{F(\pi)}{  \rho^{k}} \ \frac1{1  + \frac{\gamma_k}{ \rho C(\pi)}}\  \frac{1}{\Pp_\pi[T_{0} < T_{\pi}]} \ \frac1{1+   \E_{0}[\T_0- t^+_0]\gamma_0 }
\end{align*}
 so
\begin{align*}
& \left | \  \mu^{ \rho, z}(\pi) \ - \ \frac{F(\pi)}{ \rho^k}  \ \right |  \le \ 
\ u^n(\alpha  \rho) \  \frac{F(\pi)}{  \rho^{k}}  \ \frac1{1  + \frac{\gamma_k}{ \rho C(\pi)}}  \ \frac{1}{\Pp_\pi[T_{0} < T_{\pi}]} \ \frac1{1+   \E_{0}[\T_0 -t_0^+]\gamma_0 }  \\
&   + \ \ \left | \ \frac{C(\pi)F(\pi)}{\gamma_0  \rho^{k-1}} \frac1{( \rho C(\pi) + \gamma_k) } \frac{1}{\Pp_\pi[T_{0} < T_{\pi}]} \ \frac1{1/\gamma_0 +   \E_{0}[\T_0 -t_0^+] }  \ - \ \frac{F(\pi)}{ \rho^k} \  \right| \\
&\le \frac{F(\pi)}{ \rho^k}   \ \left ( u^n(\alpha  \rho)  \  \frac{1}{\Pp_\pi[T_{0} < T_{\pi}]} \ \frac1{1 +   \E_{0}[\T_0 -t_0^+] \gamma_0} \right. \\
 &  + \ \ \left . \left | \  1\ - \   \frac1{1 + \frac{\gamma_k}{ \rho C(\pi)} } \  \frac{1}{\Pp_\pi[T_{0} < T_{\pi}]} \ \frac1{1 +   \E_{0}[\T_0 -t_0^+] \gamma_0} \  \right| \ \right )
\end{align*} 
 and using Lemma \ref{lemma-couplage} (and the fact that $ \rho C(\pi) \ge  \rho \alpha$), the term between parentheses can be bounded by $f^n(\alpha  \rho)$, where $f^n$ is independent of the choice of $z$ and $ \rho$ and is such that 
 $\lim_{x \to +\infty}f^n(x) =  0$,
which completes the proof of Theorem \ref{thm-stationary}.
\end{proof}

\section{Proof of Theorem  \ref{thm:geommetry} }
\label{results}

 Thanks to  Proposition \ref{ARG-rescale} (scaling), in this section we will assume without loss of generality that  $\rho =1$ and consider the $\R^+$-partitioning process restricted to $[0,R]$. 
 The strategy of the proof is based on the following lemma. (Note that the second point will allow us to rephrase the convergence of $\vartheta^R$ in the weak topology in terms of a moment problem).

\begin{lemma}
\begin{itemize}
\item[$(i)$] For every $k$-tuple of disjoint intervals $\{[a_i,b_i]\}_{i=1}^k$ in $[0,1]$ and any $k$-tuple of integers $\{n_i\}_{i=1}^k$  
\[ \E\left[\prod_{i=1}^k \vartheta^{\infty}([a_i,b_i])^{n_i}\right]  \ = \   \prod_{i=1}^k n_i! \ b_i^{n_i-1}(b_i-a_i).\]
\item[$(ii)$] Let $\{\nu^R\}_{R\ge0}$ be a sequence of random variables in ${ \cal M}([0,1])$ that have no atoms and such that for every $k$-tuple of disjoint intervals $\{[a_i,b_i]\}_{i=1}^k$ in $[0,1]$ and any $k$-tuple of integers $\{n_i\}_{i=1}^k$ 
\begin{equation*}\lim_{R\to\infty}\E\left[\prod_{i=1}^k \nu^{R}([a_i,b_i])^{n_i}\right] \ = \   \prod_{i=1}^k n_i! \ b_i^{n_i-1}(b_i-a_i).  
\end{equation*}
Then $\nu^R \underset{R \to \infty}{\Longrightarrow}  \vartheta^{\infty} \textrm{ in the weak topology}$. \end{itemize}
\label{lemmamoments}
\end{lemma}

\begin{proof}[Proof of Lemma \ref{lemmamoments}]
We start by proving $(i)$. We fix $k = 1$. We fix $a, b \in [0,1], a \le b$ and we compute $M$, the moment generating function of $\vartheta^{\infty}([a,b])$.
\begin{eqnarray*}
M(t) & = & \E[\exp(t \vartheta^\infty[a,b])] =  \E\left[\exp\left(t \sum \limits_{(x_i, y_i) \in {\cal P^{\infty}},  \  x_i \in [a,b] } y_i\right)\right]. \nonumber 
\end{eqnarray*}
$M(t)$ is the Laplace functional of $\cal P ^{\infty}$ for $f(x,y) = - t y$, so it is well known that:
\begin{eqnarray*}
M(t) & = & \exp \left ( -\int_{[a,b] \times \R^+} (1- e^{ty}) \lambda(x,y)dx dy \right ) \nonumber \\
& =&  \exp \left ( -\int_{[a,b]} \frac{dx}{x} \int_{\R^+} (1- e^{ty}) \  \frac1x e^{-y/x} dy \right ) \nonumber \\
& = & \exp \left ( \int_a^b \frac{tx}{1-tx} \frac{dx}{x} \right) \ = \ \exp \left( \log \left(\frac{1 - ta}{1-tb} \right) \right) \nonumber  =  \frac{1-ta}{1-tb}.
\end{eqnarray*}
Note that when $a = 0$, $M(t)$ is the moment generating function of an exponential distribution of parameter $1/b$ and $M^{(n)}(0) = n! b^n$.
When $a \ne 0$, we use a Taylor expansion of $M(t)$:
\begin{eqnarray*}
 \frac{1-ta}{1-tb} & =& (1-ta) \sum_{n = 0}^{\infty} (tb)^n
   = \sum_{n = 0}^{\infty} (tb)^n - \sum_{n = 1}^{\infty} a b^{n-1}t^n\\ 
    & =& 1 +  \sum_{n = 1}^{\infty} \frac{n!  \ b^{n-1}(b-a)}{n!}t^n, 
\end{eqnarray*}
so $M^{(n)}(0) = n! b^{n-1}(b-a)$ for $n\geq1$, which implies $(i)$.
To prove this result for $k >1$,  we use the fact that ${\cal P }^{\infty}$ is a Poisson point process so that, for any $k$-tuple of disjoint intervals $B_1, \ldots B_k$, $\vartheta^{\infty}(B_1),\ldots,  \vartheta^{\infty}(B_k)$ are mutually independent. 

\smallskip 

We now turn to the proof of $(ii)$.  Let $\{\nu^R\}_{R\ge0}$ be a sequence of random variables in ${\cal M}([0,1])$. 
Note that for every $x\in[0,1]$, $\vartheta^{\infty}$ does not put weight on  $x$ almost surely.
From \cite{kallenberg} (Theorem 16.16 page 316), it follows that proving  that $\nu^R \underset{R \to \infty}{\Longrightarrow}  \vartheta^{\infty}$ in the weak topology boils down to proving that 
$\forall n \in \N$,  for any $k$-tuple of intervals $B_1, \ldots, B_k$ 
\begin{equation}
(\nu^R(B_1), \ldots, \nu^R(B_k))   \underset{R \to \infty}{\Longrightarrow} (\vartheta^{\infty}(B_1), \ldots, \vartheta^{\infty}(B_k)). 
\label{starstar}
\end{equation}

To prove \eqref{starstar}, we use a method of moments. We will apply an extension of Carleman's condition for multi-dimensional random variables (\cite{keiber, shohat}). Fix $n,k \in \N$ and 
for a given $k$-tuple of disjoint intervals $\{[a_i,b_i]\}_{i=1}^k$, define
$$M^k_n = \sum_{i = 1}^k \E[\vartheta^{\infty}([a_i, b_i])^{n}], \ \ C = \sum_{n = 1}^{\infty} (M^k_n)^{-\frac1{2n}}.$$
The condition states that, if $C = \infty$ (for any choice of $k$ and $\{[a_i,b_i]\}_{i=1}^k$ that are not necessarily disjoint), proving \eqref{starstar} is equivalent to proving that  for $k \in \N, \ n_1, \ldots, n_k \in \N^k$
\begin{equation}
\E \left [\prod \limits_{i=1}^k \nu^R([a_i, b_i])^{n_i}\right] \underset{R \to \infty}{\longrightarrow} \E \left [ \prod \limits_{i=1}^k \vartheta^{\infty}([a_i,b_i])^{n_i}\right]. 
\label{eq-moments-infty}
\end{equation}

From $(i)$, we have $$M^k_n  \  = \   \sum_{i = 1}^k  n! \ b_i^{n-1} \ (b_i - a_i) \le k n!$$
and since
\begin{eqnarray*}
\sum_{n = 1}^{\infty} \frac1{(k n!)^{\frac1{2n}}}  
& \ge & \frac1k\sum_{n = 1}^{\infty} \frac1{ (n!)^{\frac1{2n}}}  \ge  \frac1k\sum_{n = 1}^{\infty} \frac1{n^{\frac{1}{2}}} = \infty 
\end{eqnarray*}
we get $C = \infty$ and we can apply the extension of Carleman's condition. 
We use the fact that
\[\forall  a \le b  \le c, \ \ \ \nu^R[a,c] = \nu^R[a,b] + \nu^R[b,c] \  \textrm{ and } \ \ \vartheta^{\infty}[a,c] = \vartheta^{\infty}[a,b] + \vartheta^{\infty}[b,c]\] 
so that \eqref{eq-moments-infty} reduces to the case where the intervals $\{[a_i,b_i]\}_{i=1}^k$ are pairwise disjoint. This completes the proof of Lemma \ref{lemmamoments}.
\end{proof}

\smallskip 

Since $\vartheta^R$ is absolutely continuous with respect to the Lebesgue measure,
\[\forall  a \le b  \le c, \ \ \ \vartheta^R[a,c] = \vartheta^R[a,b] + \vartheta^R[b,c],\]
so, from Lemma \ref{lemmamoments}, the proof of Theorem \ref{thm:geommetry} boils down to proving that for every $k$-tuple of  disjoint intervals $\{[a_i,b_i]\}_{i=1}^k$ in $[0,1]$ and any $k$-tuple of integers $\{n_i\}_{i=1}^k$ 
\begin{equation}\lim_{R\to\infty}\E\left[\prod_{i=1}^k \vartheta^{R}([a_i,b_i])^{n_i}\right] \ = \   \prod_{i=1}^k n_i! \ b_i^{n_i-1}(b_i-a_i). \label{eq:cv-of-moments} \end{equation}

\smallskip

The rest of this section will be dedicated to the proof of this asymptotical relation. 
We start by fixing $k \in \N$,  $n_1, \ldots, n_k \in \N^k, \ n = n_1 + \ldots + n_k$ and  $\{[a_i, b_i]\}_{i=1}^k$ a $k$-tuple of disjoint intervals.  Without loss of generality we assume $a_1 < b_1 < a_2 \ldots < a_k<b_k$. 
For any $z=\{z_0, z_1, \ldots, z_n\}\subset \R^+$ we define $c(z)$ as the coarsest partition of $z$.

\smallskip 

We start by rewriting the expectation on the left hand side of the equation \eqref{eq:cv-of-moments}
\begin{align}
& \E \left[\prod_{i=1}^k \vartheta^{R}([a_i,b_i])^{n_i}\right] = \frac1{\log(R)^n} \ \E_{\mu^1}\left[\int_{[R^{a_1},R^{b_1}]^{n_1} \times \ldots \times[R^{a_k},R^{b_k}]^{n_k}} \mathds{1}_{\{  0 \sim {z}_1 \sim \ldots \sim {z}_{n}\} } d{z}_1 \ldots d{z}_{n}\right] \nonumber \\
 &= \frac1{\log(R)^n} \int_{[R^{a_1},R^{b_1}]^{n_1} \times \ldots \times[R^{a_k},R^{b_k}]^{n_k}}  \mu^{1,z}(c(z)) dz_1\ldots dz_n,
 \label{eq:above}
\end{align}
where $\E_{\mu^{1}}$ denotes the expectation with respect to  $\mu^{1}$, $z_0 = 0$ and $\mu^{1,z}$ is defined as the invariant measure 
of the partitioning process for the set of loci $z=\{z_0,\cdots,z_n\}$ with  $\rho = 1$. 

\smallskip

Let us now give some intuition
for the rest of the section.
Let $V_R$ be the volume of the integration domain above. We have
\begin{equation*}
\E \left[\prod_{i=1}^k \vartheta^{R}([a_i,b_i])^{n_i}\right] = \frac{V_R}{\log(R)^n} \E_Z\left [\mu^{1,\{z_0,Z\}}(\{z_0, Z\})\right],
\end{equation*}
where $\E_{Z}$ denotes the expectation with respect to $Z = (Z_1, \ldots, Z_n)$ distributed as a uniform random variable on 
$[R^{a_1},R^{b_1}]^{n_1} \times \ldots \times[R^{a_k},R^{b_k}]^{n_k}$ and where we recall that $\mu^{1,z} = \mathrm{Rest}_z \star \mu^1$. 
When $R \gg1$, for a ``typical'' configuration $Z$, the distances between the $z_i$'s will be of order $R$. As $\rho =1$,  the fragmentation rates correspond to the distances between the $z_i$'s and are of order $R \gg1$, whereas the coalescence rate is always 1 for each pair of blocks. In this situation, fragmentation events occur much more often than coalescence events, which is the framework of Theorem  \ref{thm-stationary}. The main idea behind \eqref{eq:cv-of-moments} is to approximate the integrand using this theorem. 

\smallskip

Let us now go into the details of the proof. We decompose the proof into four steps.
In the following $\{[a_i, b_i]\}_{i=1}^n$ will denote a set of disjoint intervals listed in increasing order.

\smallskip

{\bf Step 1.}
Define 
\begin{align*}
C^{R}_{\beta} \ \ :=  \ \ & \{ z_1, \ldots, z_n  \in \otimes_{i=1}^n [R^{a_i-1},R^{b_i-1}]^{n_i}:  \  \textrm{ s.t. if $z_0:= 0$  }\\ & \ \ \   \forall i \ne j \in \{0, \ldots, n\}, \ |z_i - z_j| \ge \beta\} 
\end{align*}
(Note that in the rest of the proof, we will always set $z_0=0$.)
The aim of this step is to prove the following proposition.
\begin{prop} Consider $F$ defined in \eqref{def:F}.
\begin{equation}
\forall \beta \ge1, \ \ \lim_{R\to\infty}\frac1{\log(R)^n} \int_ {C^{R}_{\beta}} F(c(z))dz_1 \ldots dz_n  \ = \ \prod_{i=1}^k n_i! \ b_i^{n_i-1}(b_i-a_i).
\end{equation}
  \label{propetoile}
  \end{prop}
To see why this Proposition is useful for the proof of \eqref{eq:cv-of-moments}, we let the reader refer to Steps 2 and 3.
In the following we fix $\beta \ge 1$, and we assume that $R$ is large enough so that $\forall i \in [k], \ R^{b_i -1} > \beta R^{-1}$.
Let $\Sigma_{n}$ be the set of permutations of $[n]$. For $\sigma\in \Sigma_{n}$ define
 \begin{equation*}
 C^{R}_{\beta, \sigma} \ := \ \{z_1, \ldots, z_n \in C^{R}_{\beta}, \ z_{\sigma(1)}< \ldots.< z_{\sigma(n)}  \}.
 \end{equation*}
 Recall that, as the intervals $\{[a_i, b_i]\}_{i=1}^n$ are disjoint,  the $z_i$'s belonging to $[a_j, b_j]$ are always smaller than those belonging to  $[a_{j+1},b_{j+1}]$. This means that there are only $n_1!\ldots n_k!$ permutations for which $C^{R}_{\beta, \sigma}$ is non empty. 
Using the symmetry between the $z_i$'s belonging to the same interval, we have
  \begin{eqnarray*}
 \int_ {C^{R}_{\beta}}  F(c(z))\ d{z}_1\ldots d{z}_n &=& \sum_{\sigma \in \Sigma_n} \int_ { C^{R}_{\beta, \sigma}}  F(c(z)) \ d{z}_1\ldots d{z}_n \nonumber \\
 & = & n_1! \ldots n_k! \int_ { C^{R}_{\beta, Id}}  F(c(z))\ d{z}_1\ldots d{z}_n,
\end{eqnarray*}
where $Id$ is the identity permutation.
To prove Proposition \ref{propetoile}, it remains to show that:
\begin{equation}
  \lim \limits_{R \to \infty} \frac1{\log(R)^n}  \ \int_ {C^{R}_{\beta, Id} }  F(c(z))d z_1\ldots d z_n \ = \prod_{i=1}^k b_i^{n_i-1}(b_i-a_i). 
\label{eqID}
\end{equation}
Recall that $F$ sums over all coalescence scenarios. The idea now is to consider separately two different types of scenarios of coalescence. 
\begin{itemize}
\item[-]  ${\cal S}_C(c(z))$ corresponds to the set of the ``contiguous scenarios'' i.e. the scenarios where  blocks only coalesce with their neighbouring blocks (i.e  where at each step the block containing $z_i$ can only coalesce with the blocks containing $z_{i-1}$ or $z_{i+1}$). This is for example the case of scenarios $S_1$ and $S_2$ in Figure \ref{fig-scenarios}. {In this type of scenarios only interval partitions arise and no trapped material emerges (in the sense of  \cite{WH}).}
\item[-]  $\bar {\cal S}_C(c(z))$ contains all the other scenarios (for example $S_3$ and $S_4$ in Figure \ref{fig-scenarios}). {These scenarios include configurations with trapped material.}
\end{itemize}
We have
\begin{align}
F(c(z)) \ = &\ \sum \limits_{s \in {\cal S}_C(c(z))} \frac1{E(s)} +  \sum \limits_{s \in \bar {\cal S}_C(c(z))}\frac1{E(s)}.
\label{eqnStep1}
\end{align}
The rest of this step is devoted to the computation of the integral over $C^{R}_{\beta, Id} $ of each of the terms in the RHS of this equation.

\begin{figure}
\begin{center}
\begin{tabular}{| m{.06\textwidth} | m{.44\textwidth} | m{.5\textwidth} |}
\hline   
    \textbf{Type } & \textbf{ Scenario of coalescence  } & \textbf{ Energy } \\
\hline 
${\cal S}_1$  & \vspace{0.4cm} \includegraphics[width = 5cm]{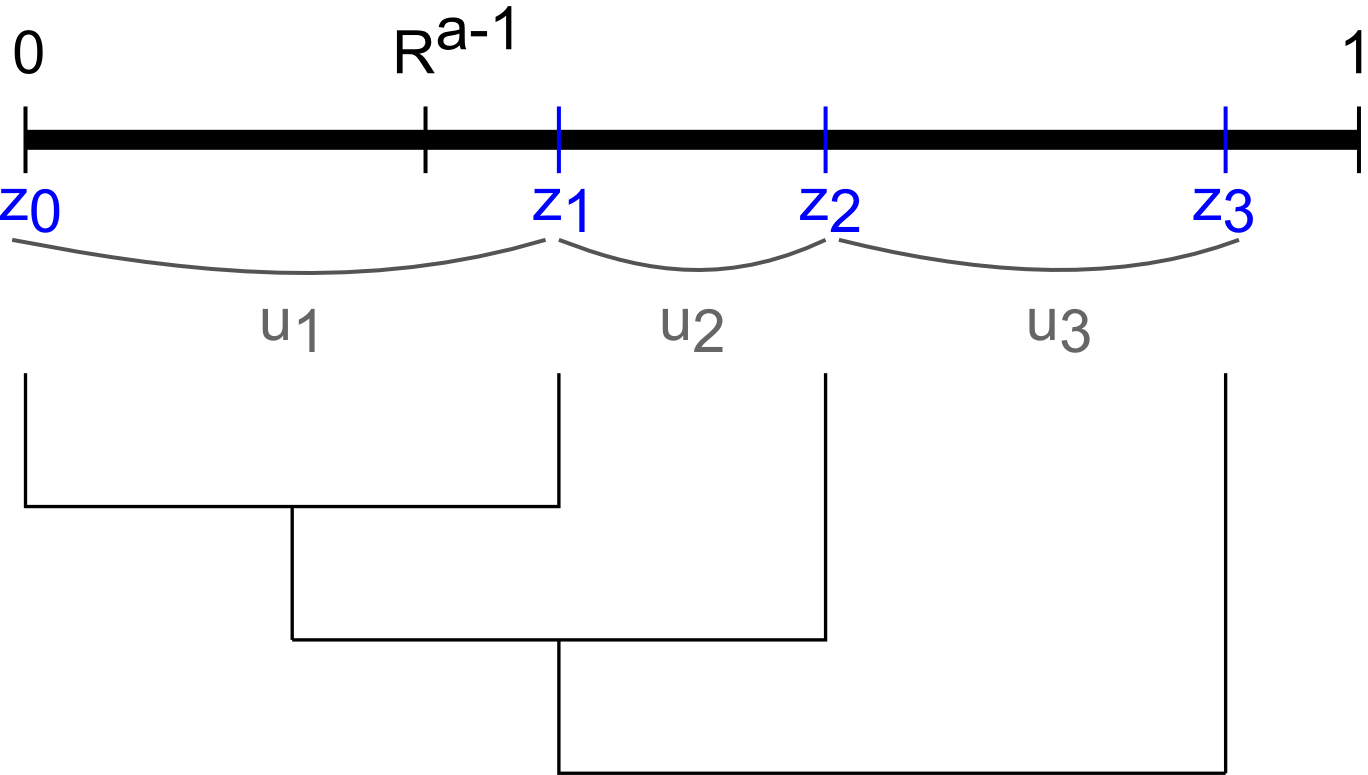} $$S_1$$ \vspace{-0.5cm}& \begin{eqnarray*}E(S_1) &=&  (z_1 - z_0) \times (z_2 - z_0) \times (z_3 - z_0) \\  &=& u_1 \times (u_1 + u_2) \times (u_1 + u_2 + u_3)\\
&& (\tau(1) = 1, \ \tau(2) = 2, \ \tau(3) = 3 ) \end{eqnarray*} \\
\hline
${\cal S}_1$ & \vspace{0.4cm} \includegraphics[width = 5cm]{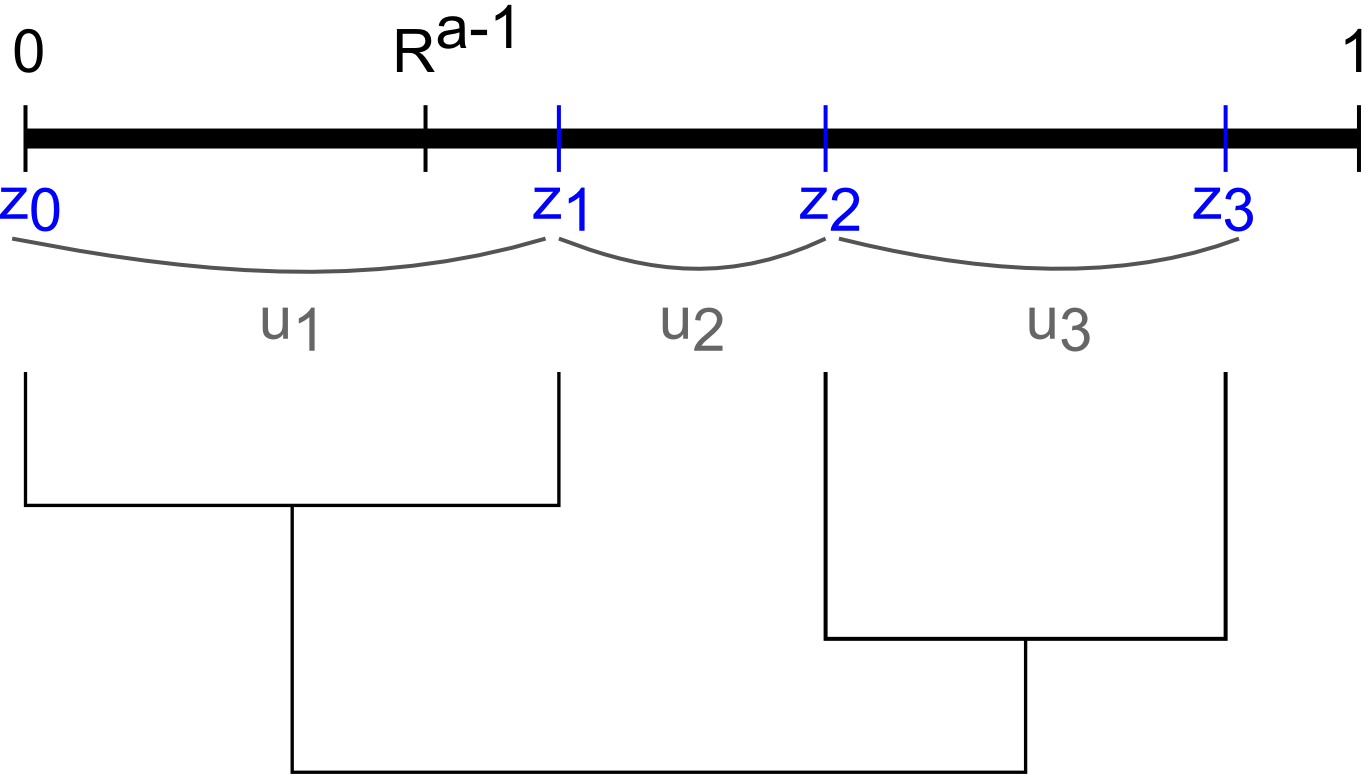} $$S_2$$ \vspace{-0.5cm}& \begin{eqnarray*}E(S_2) &=&  (z_1 - z_0)  \times ( z_1 - z_0 +  z_3 - z_2)  \\ && \times  ( z_3 - z_0) \\  &=& u_1 \times (u_1 + u_3) \times (u_1 + u_3 + u_2)\\
&& (\tau(1) = 1, \ \tau(2) = 3, \ \tau(3) = 2)
 \end{eqnarray*}\\
\hline
${\cal S}_2$ & \vspace{0.4cm} \includegraphics[width = 5cm]{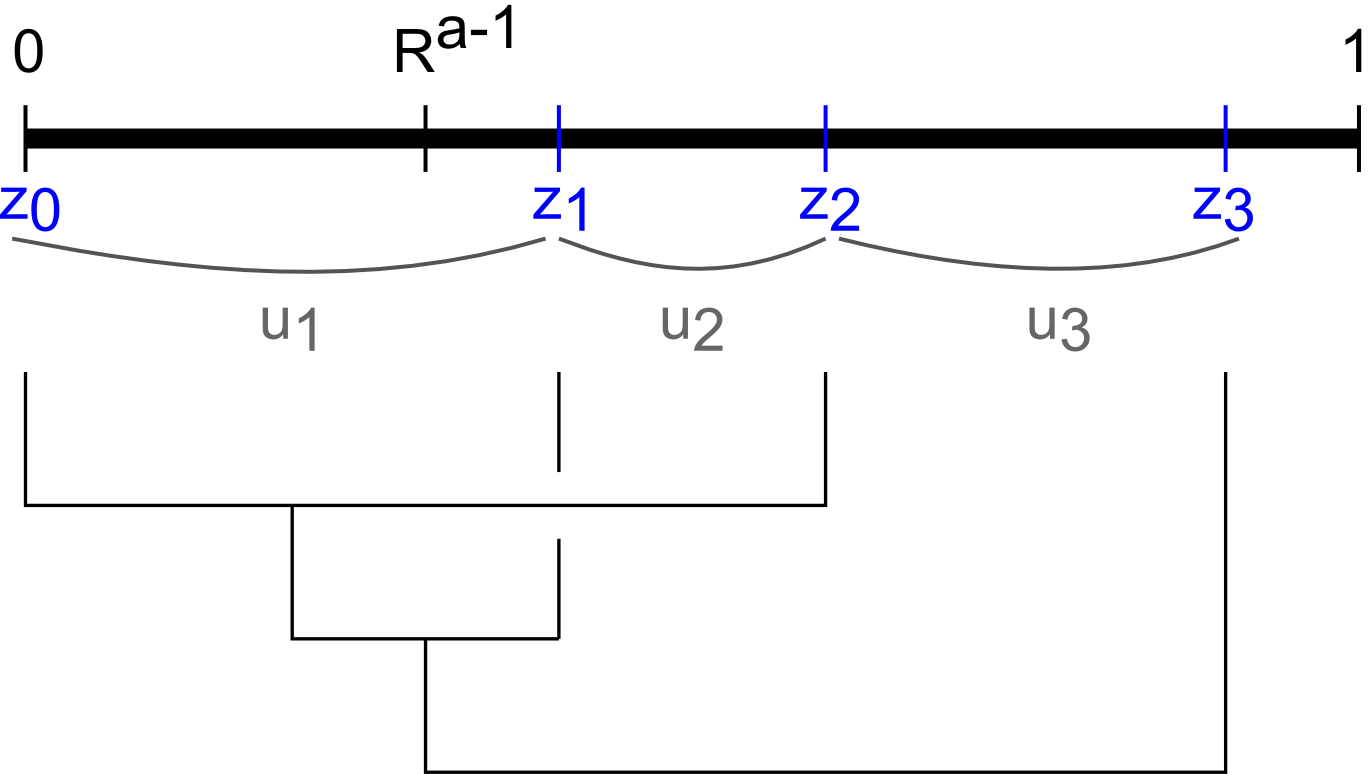} $$S_3$$ \vspace{-0.5cm}& \begin{eqnarray*}E(S_3) &=&  (z_2 - z_0)  \times ( z_2 - z_0)  \times  ( z_3 - z_0) \\  &=& (u_1 + u_2)  \times (u_1 + u_2)  \\ && \ \times (u_1 + u_2 + u_3) \end{eqnarray*} \\
\hline
${\cal S}_2$ & \vspace{0.4cm} \includegraphics[width = 5cm]{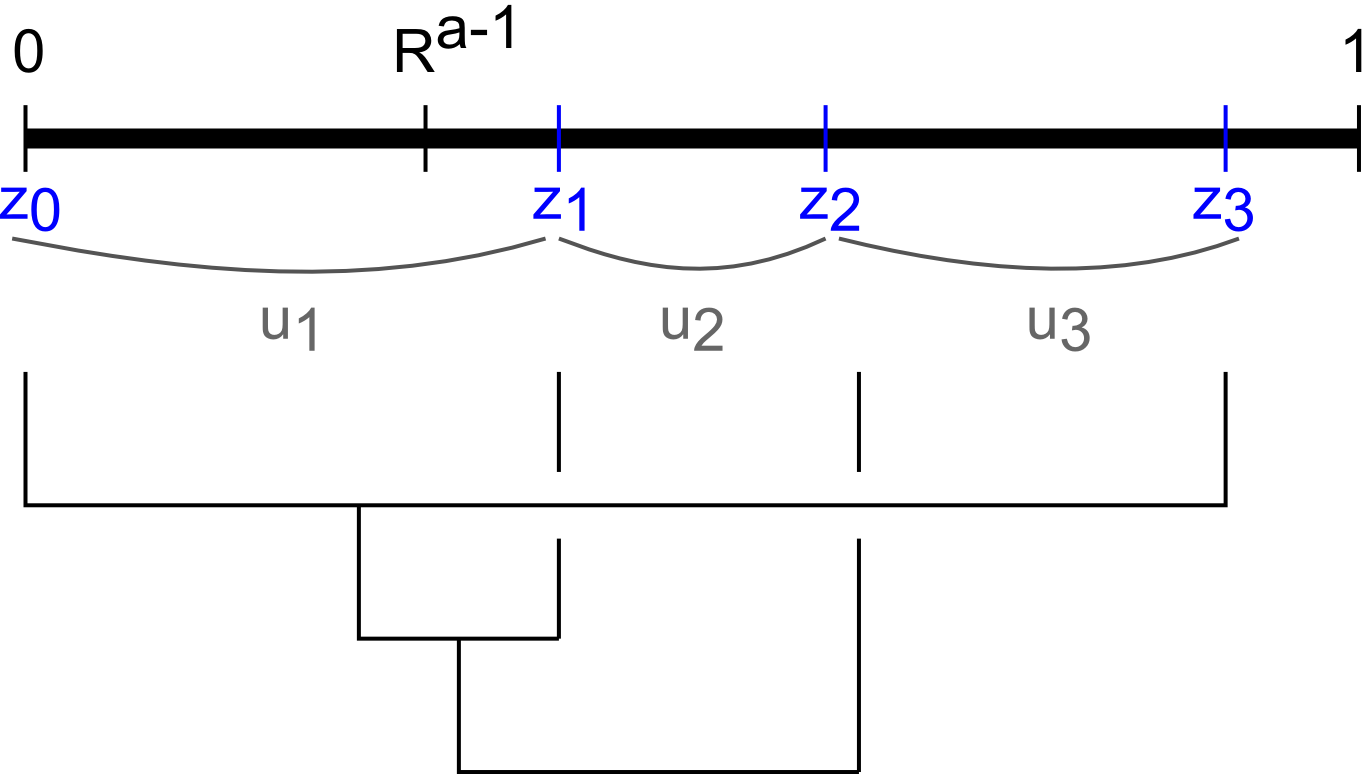} $$S_4$$ \vspace{-0.5cm}& \begin{eqnarray*}E(S_4) &=&  (z_3 - z_0)  \times ( z_3 - z_0)  \times  ( z_3 - z_0) \\  &=& (u_1 + u_2 + u_3)  \times (u_1 + u_2 + u_3)  \\ && \ \times (u_1 + u_2 + u_2) \end{eqnarray*} \\
\hline
\end{tabular}
\end{center}
\caption{Some examples of coalescence scenarios and their energy. In these examples, $k=1$, $b=1$, $a_1:=a$.}
\label{fig-scenarios}
\end{figure}

\smallskip

\textbf{Step 1.1.}
The aim of Step 1.1 is to prove the following lemma
\begin{lemma}\label{lem:target-step2}
\begin{equation*}
  \lim \limits_{R \to \infty} \frac1{\log(R)^n}  \ \int_ {C^{R}_{\beta, Id} } \sum \limits_{s \in {\cal S}_C(c(z))} \frac1{E(s)}dz_1\ldots d z_n \ = \prod_{i=1}^k b_i^{n_i-1}(b_i-a_i). 
\end{equation*}
\end{lemma}
 For each $i \in [n]$, we define $u_i  := z_{i} - z_{i - 1}$.  It is not hard to see that  each scenario $s  = (s_1, \ldots, s_n) \in {\cal S}_C(c(z))$ 
is characterized by a unique permutation $\tau \in \Sigma_n$ which specifies the order of coalescence of the successive contiguous blocks
in such a way that 
$$\frac1{ E(s)}  \ = \ \prod \limits_{i=1}^n \frac1{u_{\tau(1)} + \ldots + u_{\tau(i)}}.$$
(See Figure \ref{fig-scenarios} for some examples.)
As a consequence, we can index each contiguous scenario by a permutation, and using the change of variables $u_i  = z_{i} - z_{i - 1}$, we get
\begin{equation}
\int_ { C^{R}_{\beta, Id}} \sum \limits_{s \in {\cal S}_C(c(z))} \frac1{E(s)} \ d{z}_1\ldots d{z}_n  
 =  \int_ { U^{R}}  \sum_{\tau \in \Sigma_n}  \left ( \prod \limits_{i=1}^n \frac{du_{\tau(i)}}{u_{\tau(1)} + \ldots + u_{\tau(i)}}\right), \label{eqnU}
\end{equation}
where $U^{R}$ is defined as follows. First, let us define (see Figure \ref{U})
\begin{align*}
 &w_R(1)  :=  \max(\beta R^{-1}, R^{a_1 -1}) ,\ \  W_R(1) := \R^{b_1-1}\\ 
 \forall 2 \le i \le k, \  &w_R(i):=  R^{a_i -1} - R^{b_{i-1}-1}, \ \  W_R(i) :=  R^{b_i-1} - R^{a_{i-1}-1}\\
  \forall 1 \le i \le k, \ &L_R(i):=  R^{b_i -1} - R^{a_{i}-1}.
\end{align*}
Finally, we set $n_0 := 0$. Under the assumption that $R$ is large enough so that $\forall i \in [k], \ R^{b_i -1} > \beta R^{-1}$
\begin{align*}
U^{R}:= & \  \{ u_1, \ldots, u_n \in \otimes_{i = 1}^k \left ( [w_R(i), W_R(i)]\times[\beta R^{-1}, L_R(i)]^{n_i -1} \right)  \ : \\ & \ \forall i \in [k], \sum_{j = n_{i-1} + 2}^{n_i} u_j \le  L_R(i)  \ \textrm{ and } \  \sum_{j = 1}^{n_1 + \ldots + n_i} u_j \le  R^{b_i-1} \}
\end{align*} 
In fact, by definition of $C^{R}_{\beta, Id}$, $\forall j \in [n]$,  $\beta R^{-1} \le u_j  = z_j - z_{j-1} $. In addition, the $L_R(i)$'s correspond to the lengths of the different intervals $[R^{a_i-1}, R^{b_i-1}]$, so when $z_j$ and $z_{j-1}$ belong to the same interval,  $u_j  = z_j - z_{j-1} \le  L_R(i)$ (See Figure \ref{U}).
The $w_R(i)$'s correspond to the distance between two contiguous intervals and the $W_R(i)$'s to the maximal distance between two points of contiguous intervals. So, when $z_j$ and $z_{j-1}$ belong to different intervals then $u_j  = z_j - z_{j-1}   \in [w_R(i), W_R(i)]$ (See Figure \ref{U}). 
Finally, the last inequalities come from the fact that for $j \in \{n_{i-1}+1, \ldots, n_i\}$, the $z_i$'s belong to the same interval $[a_i, b_i]$, so the sum of their distances cannot exceed the length of the interval. In addition, the distance between $z_0$ and $z_{n_i}$ cannot exceed  $R^{b_i-1}$.
\begin{figure}
\begin{center}
\includegraphics[width=8cm]{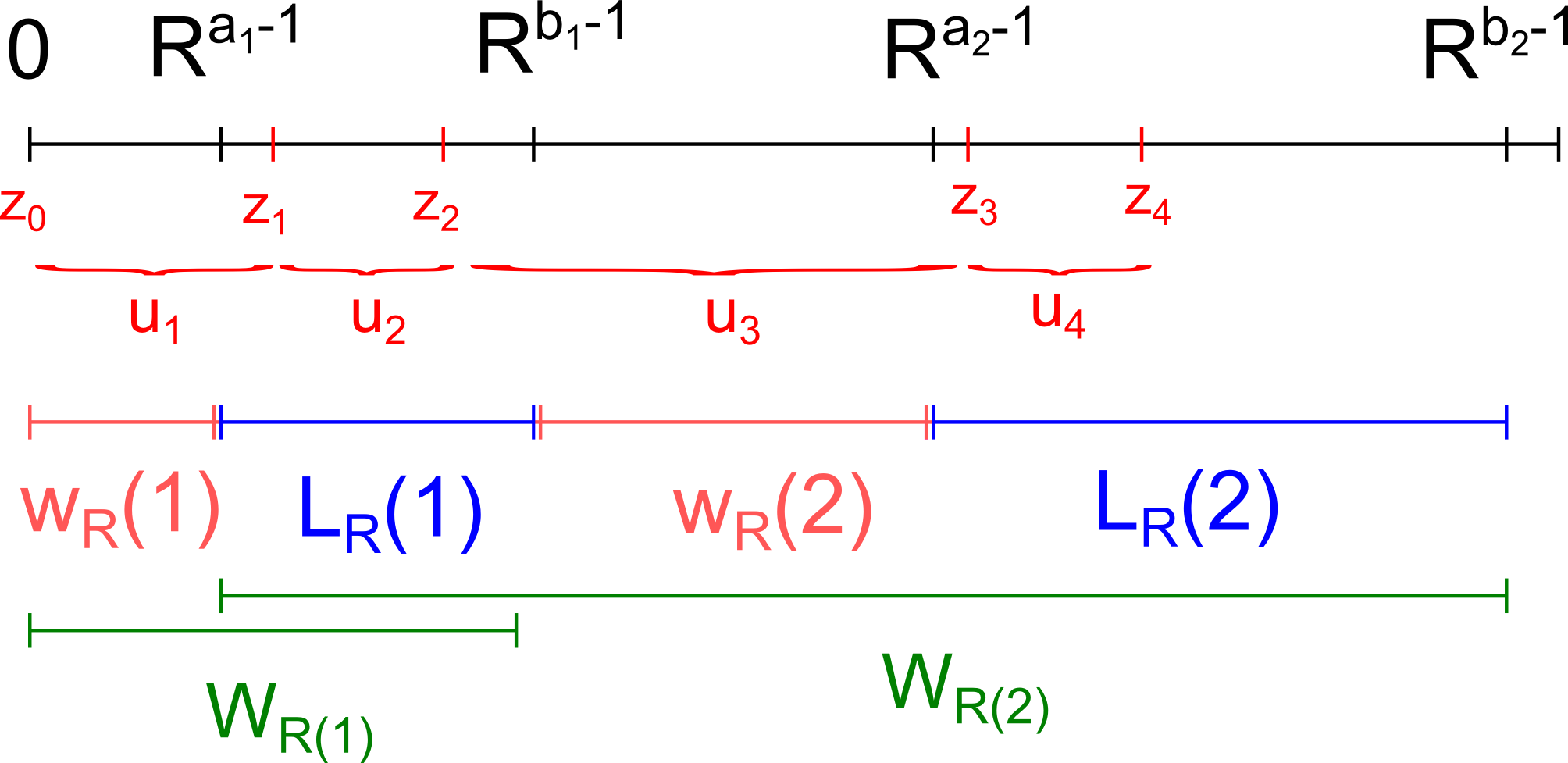}
\end{center}
\caption{The set $U^{R}$. }
\label{U}
\end{figure}
To compute the RHS of \eqref{eqnU}, we need to prove two Lemmas. 
Their proofs are  rather cumbersome, but the idea behind them is simple.
In a nutshell, the idea is that, depending on the positions of the loci (the $z_i$'s), one scenario is much more likely than the others. More precisely, for any configuration $z \in C^{R}_{\beta, Id}$, there exists a scenario $S_{min}\in {\cal S}_C(c(z))$ associated to  permutation $\tau_{min} \in \Sigma_n$ such that  
$u_{\tau_{min}(1)} \le u_{\tau_{min}(2)} \le \ldots \le u_{\tau_{min}(n)}$.
By coalescing the $u_i$'s in the increasing order, the successive cover lengths are minimised. 
This means that the only term  that doesn't vanish in the integral in  the RHS of \eqref{eqnU} is the one corresponding to this permutaiion $\tau_{min}$ and we have to compute this term.
More precisely, for any $\tau \in \Sigma_n, \ \kappa >1$, define
\begin{align*}
K^{R,\kappa}_{ \tau} \ &:= \{u_1, \ldots, u_n \in  U^{R}, \forall i \in [n], \ u_{\tau(i)} > \kappa \sum_{j = 1}^{i-1}  u_{\tau(j)} \} \\
\bar K^{R,\kappa}_\tau \ &:= U^{R} \setminus K^{R,\kappa}_{\tau}.
\end{align*}

\begin{lemma} 
We have
\begin{align*}
  \forall \kappa >1, \ \lim_{R \to \infty} \frac1{\log(R)^n}  \sum_{\tau \in \Sigma_n}  \int_ {\bar K_\tau^{R,\kappa}}\left ( \prod \limits_{i=1}^n \frac{du_{\tau(i)}}{u_{\tau(1)} + \ldots + u_{\tau(i)}}\right)   \ = \ 0 
 \end{align*}
\label{lemmaNM1}
\end{lemma}
\begin{proof}
We fix $\tau \in \Sigma_n, \ \kappa >1$.
We make the following change of variables. Let us define $\Psi^{\tau}$ such that for  $1\le i \le n$, $(\Psi^{\tau}(u_1, \ldots, u_n))_{\tau(i)} = u_{\tau(1)} + \ldots + u_{\tau(i)}$. We have
\begin{equation}
  \int_ { U^{R}}  \left ( \prod \limits_{i=1}^n \frac{du_{\tau(i)}}{u_{\tau(1)} + \ldots + u_{\tau(i)}}\right)  =  \int_ {v\in \Psi^\tau(U^{R})}  \frac{dv_{1}\ldots dv_{n}}{v_{1} \ldots v_{n}}.
 \label{eq20}
\end{equation}
In particular, as $$\forall i \in [n], \ \ u_{\tau(i)} \le \kappa \sum_{j = 1}^{i-1}  u_{\tau(j)} \ \Leftrightarrow \ \forall i \in [n], \ \ v_{\tau(i)} \le (1+\kappa)v_{\tau(i-1)},$$ we have
\begin{equation*}
  \int_ { \bar K^{R,\kappa}_{\tau}}  \left ( \prod \limits_{i=1}^n \frac{du_{\tau(i)}}{u_{\tau(1)} + \ldots + u_{\tau(i)}}\right)
 =  \int_ {V^{R,\kappa}_{\tau}}  \frac{dv_{1}\ldots dv_{n}}{v_{1} \ldots v_{n}}
\end{equation*}
where $$V^{R,\kappa}_{\tau} = \{v \in  \Psi^\tau(U^{R}), \exists i \in [n], \ v_{\tau(i)} \le (1 + \kappa)v_{\tau(i-1)} \}. $$

\smallskip

 For every $i \in [n]$, we define
\begin{align*}
 V^{R,\kappa}_{\tau}(i) \ &= \ \{v \in  \Psi^\tau(U^{R}), \  v_{\tau(i)} \le (1 + \kappa)v_{\tau(i-1)} \}. 
\end{align*}
We have
$$V^{R, \kappa}_{\tau}  = \bigcup_{i =1}^n  V^{R, \kappa}_{\tau}(i). $$ 
We fix $\tau \in \Sigma_n$ and $i \in [n]$. 
The $v_{\tau(j)}$'s are the successive cover lengths at each step of the scenario $S$ associated to $\tau$, so it can readily be seen that $\forall j \in [n], v_j \in [R^{-1}, 1]$ and $v_{\tau(1)} < \ldots < v_{\tau(n)}$,
which implies that
$$V^{R, \kappa}_{\tau}(i) \subset \{v \in [R^{-1}, 1]^n, \ v_{\tau(i-1)} < v_{\tau(i)} < (1 + \kappa) v_{\tau(i-1)} \}.$$
Recall that  $v_j \ge u_j \ge \beta R^{-1}\ge R^{-1}$, so
\begin{align*}
& \int_{V^{R, \kappa}_{\tau}(i)} \frac{dv_{\tau(1)} \ldots dv_{\tau(n)}}{v_{\tau(1)}\ldots v_{\tau(n)}} \le \left( \int_{R^{-1}}^{1} \frac{dv}{v}\right )^{n-2} \int_{R^{-1}}^{1} \frac{d v_{\tau(i-1)}}{ v_{\tau(i-1)}} \int_{ v_{\tau(i-1)}}^{  (1 + \kappa) v_{\tau(i-1)} } \frac{dv_{\tau(i)}}{v_{\tau(i)}}\\
& = \log(R)^{n-2} \int_{R^{-1}}^{1} \frac{d v_{\tau(i-1)}}{ v_{\tau(i-1)}} \log(1 + \kappa) 
 = \log(R)^{n-1}  \log (1 + \kappa),
\end{align*}
which completes the proof.
\end{proof}

\begin{lemma} 
We have
\begin{align*}
& \lim_{\kappa \to \infty} \lim_{R \to \infty} \frac1{\log(R)^n} \sum_{\tau \in \Sigma_n} \  \int_ { K^{R,\kappa}_{ \tau} }  \left ( \prod \limits_{i=1}^n \frac{du_{\tau(i)}}{u_{\tau(1)} + \ldots + u_{\tau(i)}}\right)  =  \prod_{i=1}^k b_i^{n_i-1}(b_i-a_i).
 \end{align*}
\label{lemmaNM2}
\end{lemma}

\begin{proof} 
We decompose the proof into four steps.

{\bf Step a.}
Define
\begin{align*}
X^{R} \  :=   & \  \otimes_{i = 1}^k \left ( [w_R(i), W_R(i)]\times[\beta R^{-1}, L_R(i)]^{n_i -1} \right).\
\end{align*}
Then 
\begin{align}
 \frac1{\log(R)^n}\int_{ X^{R}}  \left ( \prod \limits_{i=1}^n \frac{du_{i}}{ u_{i}}\right) 
& = \ \frac1{\log(R)^n}  \prod_{i = 1}^k \ \int_{w_R(i)}^{W_R(i)} \frac{du}{u}\left(\int_{\beta R^{-1}}^{L_R(i)}\frac{du}{u} \right)^{n_i -1}\nonumber\\
& = \ \frac1{\log(R)^n}  \prod_{i = 1}^k \log\left(\frac{W_R(i)}{w_R(i)}\right) \log\left(\frac{L_R(i)}{\beta R^{-1}}\right)^{n_i -1}\nonumber\\
& \underset{R \to \infty}{\longrightarrow}  \ \prod_{i=1}^k (b_i-a_i)b_i^{n_i-1}. \label{eq:step1}
\end{align}

{\bf Step b.} Next, for every $\kappa >1$ and for every $\tau\in\Sigma_n$, we define
\begin{align*}
  X^{R}_\tau  \ := & \ \{  u_1, \ldots, u_n \in  X^{R} \ : \ u_{\tau(1)}\le \ldots \le  u_{\tau(n)}\},\\
  \  A^\kappa_{\tau }\  :=  & \  \{ u_1, \ldots, u_n, \forall i \in [n], \ u_{\tau(i)} > \kappa \sum_{j = 1}^{i-1}  u_{\tau(j)}\}. \\
  \  X^{R, \kappa}_\tau  \  := &  \  X^{R} \cap A^\kappa_{\tau }.
\end{align*}
By reasoning along the same lines as in the proof of $\textrm{(i)}$, one can show that 
\begin{equation*}
\lim_{R\to \infty} \frac1{\log(R)^n}\left | \int_ {  X^{R}_\tau}  \left ( \prod \limits_{i=1}^n \frac{du_{\tau(i)}}{ u_{\tau(i)}}\right) - \int_ { X^{R, \kappa}_\tau  }  \left ( \prod \limits_{i=1}^n \frac{du_{\tau(i)}}{ u_{\tau(i)}}\right) \right| = 0.
\end{equation*}
From Step a, we get that for every $\kappa>1$
\begin{equation}\label{target1}\lim_{R\to \infty} \frac1{\log(R)^n} \sum_{\tau\in\Sigma_n} \int_ { X^{R, \kappa}_\tau  }  \left ( \prod \limits_{i=1}^n \frac{du_{\tau(i)}}{ u_{\tau(i)}}\right)  \ = \  \prod_{i=1}^k (b_i-a_i)b_i^{n_i-1} \end{equation}
\smallskip

{\bf Step c.} 
The aim of this step is to prove that for every $\tau\in\Sigma_n$,
\begin{align}
\lim_{R\to \infty}\frac1{\log(R)^n} \int_{  K^{R,\kappa}_{ \tau} }  \left ( \prod \limits_{i=1}^n \frac{du_{\tau(i)}}{ u_{\tau(i)}}\right)  \ & = \lim_{R\to \infty} \frac1{\log(R)^n} \int_ { X^{R, \kappa}_\tau  }  \left ( \prod \limits_{i=1}^n \frac{du_{\tau(i)}}{ u_{\tau(i)}}\right)  \label{target2}
\end{align}
From the definition of $K^{R,\kappa}_{ \tau}$, we have $K^{R,\kappa}_\tau =  X^{R, \kappa}_\tau \cap (K^{R}_1 \cap  K^{R}_2)$, where
\begin{align*}
 K^{R}_1  := &  \{ u_1, \dots, u_n,  \ \ \ \forall i \in [k], \sum_{j = n_{i-1} + 2}^{n_i} u_j \le  L_R(i)  \},\\
 K^{R}_2  := &  \{  u_1, \dots, u_n, \  \ \forall i \in [k],  \sum_{j = 1}^{n_1 + \dots + n_i} u_j \le  R^{b_i-1}  \ \}.
\end{align*} 
If
\begin{equation}
\lim_{R\to \infty}  \frac1{\log(R)^n} \left|\int_{ X^{R, \kappa}_\tau  \cap K^{R}_1}  \left ( \prod \limits_{i=1}^n \frac{du_{\tau(i)}}{ u_{\tau(i)}}\right) - \int_ { X^{R,\kappa}_\tau }  \left ( \prod \limits_{i=1}^n \frac{du_{\tau(i)}}{ u_{\tau(i)}}\right) \right| = 0
\label{K1}
\end{equation}
and 
\begin{equation}
\lim_{R\to \infty} \frac1{\log(R)^n}\left | \int_ {X^{R, \kappa}_\tau  \cap  K^{R}_2 }  \left ( \prod \limits_{i=1}^n \frac{du_{\tau(i)}}{ u_{\tau(i)}}\right) - \int_ { X^{R,\kappa}_\tau}  \left ( \prod \limits_{i=1}^n \frac{du_{\tau(i)}}{ u_{\tau(i)}}\right) \right| = 0,
\label{K2}
\end{equation}
then 
\begin{equation}
\lim_{R\to \infty} \frac1{\log(R)^n}\left | \int_ { K^{R,\kappa}_\tau }  \left ( \prod \limits_{i=1}^n \frac{du_{\tau(i)}}{ u_{\tau(i)}}\right) - \int_ { X^{R, \kappa}_\tau  }  \left ( \prod \limits_{i=1}^n \frac{du_{\tau(i)}}{ u_{\tau(i)}}\right) \right| = 0.
\label{K12}
\end{equation}

\smallskip 

We will only prove \eqref{K1}, as \eqref{K2} can be proved along the same lines. To do so, we define
\begin{align*}
Y^{R,\kappa}_\tau  :=  &  \{u_1, \ldots, u_n \in \otimes_{i = 1}^k \left ( [w_R(i), W_R(i)]\times\left[\beta R^{-1}, \frac{L_R(i)}{1+ \frac1{\kappa}}\right]^{n_i -1} \right), \\  & u_{\tau(1)} \le \ldots \le u_{\tau(n)} \} .
\end{align*} 
so that $Y^{R,\kappa}_\tau\subset X^{R}_\tau$.
Recall that $A^\kappa_{\tau }  \subset X^{R}_\tau$ so $X^{R, \kappa}_\tau = X^{R}_\tau \cap A^\kappa_{\tau }$.
By similar computations as those used in the proof of $\textrm{(i)}$, it can be shown that 
\begin{equation*}
\lim_{R\to \infty} \frac1{\log(R)^n} \int_ { X^R_\tau \setminus Y^{R,\kappa}_\tau  }   \prod \limits_{i=1}^n \frac{du_{\tau(i)}}{ u_{\tau(i)}} = 0,
\end{equation*}
so
\begin{equation}
\lim_{R\to \infty} \frac1{\log(R)^n}  \left | \int_ {  X^{R, \kappa}_\tau }  \  \prod \limits_{i=1}^n \frac{du_{\tau(i)}}{ u_{\tau(i)}} \ - \ \int_ { Y^{R,\kappa}_\tau \cap A^\kappa_{\tau }}  \  \prod \limits_{i=1}^n \frac{du_{\tau(i)}}{ u_{\tau(i)}}\right| = 0,
\label{eqnA1}
\end{equation}
and
\begin{equation}
\lim_{R\to \infty} \frac1{\log(R)^n}  \left | \int_ {  X^{R, \kappa}_\tau \cap K_1^R }  \  \prod \limits_{i=1}^n \frac{du_{\tau(i)}}{ u_{\tau(i)}} \ - \ \int_ { Y^{R,\kappa}_\tau \cap A^\kappa_{\tau}  \cap K_1^R  }  \  \prod \limits_{i=1}^n \frac{du_{\tau(i)}}{ u_{\tau(i)}}\right| = 0.
\label{eqnA2}
\end{equation}

\smallskip

Let us show that
\[ Y^{R,\kappa}_\tau\cap A^\kappa_{\tau } \cap K_1^R =  Y^{R,\kappa}_\tau\cap A^\kappa_{\tau },\]
i.e. that
 \[\forall (u_1, \ldots, u_n) \in Y^{R,\kappa}_\tau\cap A^\kappa_{\tau }, \  \ \forall i \in [k], \sum_{j = n_{i-1} + 2}^{n_i} u_j \le   L_R(i). \]  
We fix $i \in [k]$ and we define 
\[ m_i  := j \in \{n_{i-1} + 2, \ldots, {n_i} \}, \ \tau^{-1}(j) = \max\{\tau^{-1}(n_{i-1} + 2), \ldots, \tau^{-1}(n_i)\} \]
As \[u_{\tau(m_i)} > \kappa \sum_{j = 1}^{m_i-1}  u_{\tau(j)},\]
then
\[\sum_{j = n_{i-1} + 2}^{n_i} u_j \le \sum_{j = 1}^{m_i}  u_{\tau(j)}  \le \left(1 + \frac1{\kappa}\right)u_{\tau(m_i)} \le  \left(1 + \frac1{\kappa}\right)   \frac{L_R(i)}{1 + \frac1{\kappa}} = L_R(i).\]

\smallskip 

Since $Y^{R,\kappa}_\tau\cap A^\kappa_{\tau } \cap K_1^R =  Y^{R,\kappa}_\tau\cap A^\kappa_{\tau }$,   combining \eqref{eqnA1} and  \eqref{eqnA2},  \eqref{K1} is proved. Equation \eqref{K2} can be proved along the same lines, so \eqref{K12} is verified.

\smallskip 

{\bf Step d.} Finally, for any $\tau \in \Sigma_n$, for any $u_1, \ldots, u_n \in K^{R,\kappa}_{ \tau}$, we have
$$\forall i \in [n], \ u_{\tau(i)}\  \le \ u_{\tau(1)} + \ldots + u_{\tau(i-1)} + \ u_{\tau(i)} \ \le \ \left(1+\frac1{\kappa}\right) \ u_{\tau(i)},$$
which implies that 
\begin{equation*}
 \sum_{\tau \in \Sigma_n}\int_ { K^{R,\kappa}_{ \tau} }  \left ( \prod \limits_{i=1}^n \frac{du_{\tau(i)}}{ u_{\tau(i)}}\right) \ge \sum_{\tau \in \Sigma_n} \int_ { K^{R,\kappa}_{ \tau} }  \left ( \prod \limits_{i=1}^n \frac{du_{\tau(i)}}{u_{\tau(1)} + \ldots + u_{\tau(i)}}\right) 
\end{equation*}
and
\begin{equation*}
 \sum_{\tau \in \Sigma_n} \int_ { K^{R,\kappa}_{ \tau} }  \left ( \prod \limits_{i=1}^n \frac{du_{\tau(i)}}{u_{\tau(1)} + \ldots + u_{\tau(i)}}\right) 
 \ge \frac1{(1+\frac1{\kappa})^n} \sum_{\tau \in \Sigma_n} \int_ { K^{R,\kappa}_{ \tau} }  \left ( \prod \limits_{i=1}^n \frac{du_{\tau(i)}}{ u_{\tau(i)}}\right)
\end{equation*}

This, combined with Step b (see (\ref{target1})) and Step c  (see (\ref{target2}))
\begin{align*}
\ \prod_{i=1}^k (b_i-a_i)b_i^{n_i-1} &\ge
 \lim_{R\to \infty} \frac1{\log(R)^n}\sum_{\tau \in \Sigma_n} \int_ { K^{R,\kappa}_{ \tau} }  \left ( \prod \limits_{i=1}^n \frac{du_{\tau(i)}}{u_{\tau(1)} + \ldots + u_{\tau(i)}}\right) \\
& \ge \frac{\ \prod_{i=1}^k (b_i-a_i)b_i^{n_i-1}}{(1+\frac1{\kappa})^n},
 \end{align*}
 and the conclusion follows by taking $\kappa \to \infty$. 
\end{proof}

\begin{proof}[Proof of Lemma \ref{lem:target-step2}]
From \eqref{eqnU}, we have
\begin{align*}
&\lim_{R \to \infty} \frac1{\log(R)^n}\int_ { C^{R}_{\beta, Id}} \sum \limits_{s \in {\cal S}_C(c(z))} \frac1{E(s)} \ d{z}_1\ldots d{z}_n   \\ &= \lim_{R \to \infty} \frac1{\log(R)^n}   \sum_{\tau \in \Sigma_k} \int_ { U^{R}}   \prod \limits_{i=1}^n \frac{du_{\tau(i)}}{u_{\tau(1)} + \ldots + u_{\tau(i)}}  \\
&  =   \lim_{R \to \infty} \frac1{\log(R)^n}  \sum_{\tau \in \Sigma_k} \left(  \int_ { K^{R, \kappa}_\tau}  \prod \limits_{i=1}^n \frac{du_{\tau(i)}}{u_{\tau(1)} + \ldots + u_{\tau(i)}}  + \int_ { \bar K^{R, \kappa}_\tau}   \prod \limits_{i=1}^n \frac{du_{\tau(i)}}{u_{\tau(1)} + \ldots + u_{\tau(i)}}\right) . 
\end{align*}
Using Lemmas \ref{lemmaNM1} and \ref{lemmaNM2} and taking $\kappa \to \infty$ in the RHS, we have
\begin{align}
\lim_{R \to \infty} \frac1{\log(R)^n}\int_ { C^{R}_{\beta, Id}} \sum \limits_{s \in {\cal S}_C(c(z))} \frac1{E(s)} \ d{z}_1\ldots d{z}_n     \ = \ \prod_{i=1}^k b_i^{n_i-1}(b_i-a_i).
\label{eqnStep2}
\end{align}
\end{proof}

\textbf{Step 1.2.} The aim of this step is to compute the second term in the RHS of \eqref{eqnStep1}.
\begin{lemma} 
\begin{equation*}
 \lim \limits_{R \to \infty}  \frac1{\log(R)^n} \ \int_ {C^{R}_{\beta, Id} } \sum_{S \in \bar {\cal S}_C(c(z))}\frac1{E(S)} d z_1\ldots d z_n\ = \ 0.
 \end{equation*}
\label{lemmaNC}
\end{lemma}
\begin{proof}
We fix  $z = z_0, z_1, \ldots, z_n \in C^{R}_{\beta, Id}$.
We start by considering scenarios where blocks only coalesce with neighbouring blocks, except for one step. In other words, we start by considering scenarios in  $\bar {\cal S}'_C(c(z))$, the set of scenarios that contain one single coalescence event between two non-neighbouring blocks. For example, $S_3$ in Figure \ref{fig-scenarios} is in $\bar {\cal S}'_C(c(z))$.
We consider $S' = (s'_0, s'_1, \ldots, s'_n) \in \bar {\cal S}'_C(c(z))$, a scenario of coalescence in which step $j>1$, is the only coalescence event between two non neighbouring blocks. 
The idea is to compare $S'$ with a scenario $S\in {\cal S}_C(c(z))$
and use this scenario to show that  
\begin{equation}
\lim_{R \to \infty}\frac1{\log(R)^n} \int_{C^{R}_{\beta, Id}} \frac1{E(S')}  dz_1 \ldots dz_n = 0.
\label{scenarioS'} 
\end{equation}
As already argued each scenario in $ {\cal S}_C(c(z))$ is associated to a permutation $\tau$ such that, at each step $i$, the cover length increases by $u_{\tau(i)}$. 
The scenario $S$ and the corresponding permutation $\tau$ are constructed as follows (and we let the reader refer to Figure \ref{holes} for an example, where subfigures $\textrm(i), \ldots, \textrm(iv)$  correspond to each step in the following construction).
\begin{itemize}
\item[$\textrm(i)$] For $0 \le i < j$, we set $s_i = s'_i$. Before step $j$,  there are only coalescence events between neighbouring blocks in $S'$. 
$\tau(1), \ldots, \tau(j-1)$ are constructed in such a way that $$\forall 1\le i <j, \ \ C(s_i) = C(s_i') =  \sum_{k=1}^iu_{\tau(k)}.$$
\item[$\textrm(ii)$]  At step $j$, in scenario $S'$ there is a coalescence event between two non neighbouring blocks, which means that there exists ${i_1} < i_2 < \ldots < {i_\ell}$ such that $C(s'_j) = C(s'_{j-1}) + u_{i_1}+ \ldots + u_{i_\ell}$. 
$s_j$ is the partition of order $j$ such that $C(s_j)  = C(s_{j-1}) + u_{i_1}$. We set $\tau(j) := i_1$.
We have $$C(s_j) =  \sum_{k=1}^iu_{\tau(k)} \le C(s_j').$$
\item[$\textrm(iii)$]  For $j< i \le j+ \ell-1$, $\tau(i) : = i_{i-j}$, i.e., we add successively $u_{i_2}, \ldots, u_{i_\ell}$. We have $$\forall j< i \le j+ \ell-1, \  C(s_i) \le C(s_i').$$
\item[$\textrm(iv)$]  For $j+\ell \le i \le j$, the $s_i$'s are constructed as follows. Let $u_{r_1}, \ldots , u_{r_p}$ be the $u_i$'s that have not been added yet to the cover length of $S'$ (i.e. the $u_i$'s that are not in $\{ u_{\tau(1)}, \ldots , u_{\tau(j+\ell)}\}$), indexed in such a way that in $S'$, $u_{r_1}$ coalesces before $u_{r_2}$ etc $\ldots$. Then we set  $u_{\tau(j +\ell + i )} = u_{r_i}$.  In other words, the $u_{r_i}$'s are added to the cover length in $S$ in the same  order as they are added in $S'$ (see Figure \ref{holes}).
\end{itemize} 
With this construction, we have
\begin{align*}
\frac1{E(S)} =& \prod \limits_{i=1}^n \frac1{u_{\tau(1)} + \ldots + u_{\tau(i)}} 
= \frac1{v_{\tau(1)} \ldots v_{\tau(n)} }.
\end{align*}
where for $i \in [n], \ v_{\tau(i)} := \Psi^\tau(U^{R})_{\tau(i)} =  u_{\tau(1)} + \ldots + u_{\tau(i)}$. By construction, we have
\begin{align*}
\frac1{E(S')} \le &  \left(\prod \limits_{i=1}^{j-1} \frac1{C(s_i)} \right) \  \frac1{C(s'_j)}  \ \left( \prod \limits_{i=j+1}^{n} \frac1{C(s_i)}\right).
\end{align*}
Using the fact that 
\begin{align*}
C(s'_j) = u_{\tau(1)} + \ldots + u_{\tau(j-1)} + u_{i_1}+ \ldots + u_{i_\ell}  = v_{\tau(j+\ell)},
\end{align*}
we have
\begin{align*}
\frac1{E(S')} \le & \left(\prod \limits_{i=1}^{j-1} \frac1{v_{\tau(i)}} \right)\frac1{v_{\tau(j+\ell)}} \left(  \prod \limits_{i=j+1}^{n} \frac1{v_{\tau(i)}} \right) 
\le   \left(\prod \limits_{i=1}^{j-1} \frac1{v_{\tau(i)}} \right)\frac1{v_{\tau(j+1)}} \left(  \prod \limits_{i=j+1}^{n} \frac1{v_{\tau(i)}} \right), 
\end{align*}
where the last inequality comes from the fact that $v_{\tau(j+1)} < \ldots < v_{\tau(j)+l}$.
Using the same change of variables as in \eqref{eq20}, we have
\begin{eqnarray*}
\int_{C^{R}_{\beta, Id}} \frac1{E(S')}  dz_1 \ldots dz_n & \le & \int_{\Psi^{\tau}(U^{R}) }  \frac{dv_{\tau(1)} \ldots dv_{\tau(n)}}{v_{\tau(1)} \ldots v_{\tau(j-1)}  v_{\tau(j+1)}  v_{\tau(j+1)} \ldots v_{\tau(n)} }.
\end{eqnarray*}
And, from the definition of $\Psi^\tau$, it can easily be seen that for any $\tau \in \Sigma_n$ 
$$\Psi^{\tau}(U^{R}) \subset \Psi' = \{x_1, \ldots, x_n  \in  [R^{-1}, 1]^n, x_1 \le  \ldots \le  x_n\}$$
so
\begin{align*}
\int_{C^{R}_{\beta, Id}} \frac{dz_1 \ldots dz_n}{E(S')}   &\le 
 \int_{\Psi'}  \frac{dx_1 \ldots dx_n}{x_1 \ldots x_{j-1}  x_{j+1}  x_{j+1} \ldots x_n }\\
 & = \int_{R^{-1}}^1 \frac{dx_n}{x_n}   \ldots   \int_{R^{-1}}^{x_{j+2}} \frac{dx_{j+1}}{x_{j+1}^2} \int_{R^{-1}}^{x_{j+1}} dx_{j}  \int_{R^{-1}}^{x_{j}} \frac{dx_{j-1}}{x_{j-1}}\ldots \int_{R^{-1}}^{x_2} \frac{dx_1}{x_1}  \\
 & =  \int_{R^{-1}}^1 \frac{dx_n}{x_n}   \ldots   \int_{R^{-1}}^{x_{j+2}} \frac{dx_{j+1}}{x_{j+1}^2} \int_{R^{-1}}^{x_{j+1}}  \frac{\log(Rx_{j})^{j-1}}{(j-1)!} dx_{j}    \\
&  \le  \frac{\log(R)^{j-1}}{(j-1)!} \int_{R^{-1}}^1 \frac{dx_n}{x_n}   \ldots   \int_{R^{-1}}^{x_{j+2}} \frac{x_{j+1} -1/R}{x_{j+1}^2} \ dx_{j+1}   \\
&  \le  \log(R)^{j-1} \int_{R^{-1}}^1 \frac{dx_n}{x_n}   \ldots   \int_{R^{-1}}^{x_{j+2}} \frac{dx_{j+1}}{x_{j+1}}  \ = \  \log(R)^{n-1},
\end{align*}
which completes the proof of \eqref{scenarioS'}.

\begin{figure}
\begin{center}
\includegraphics[width=9cm]{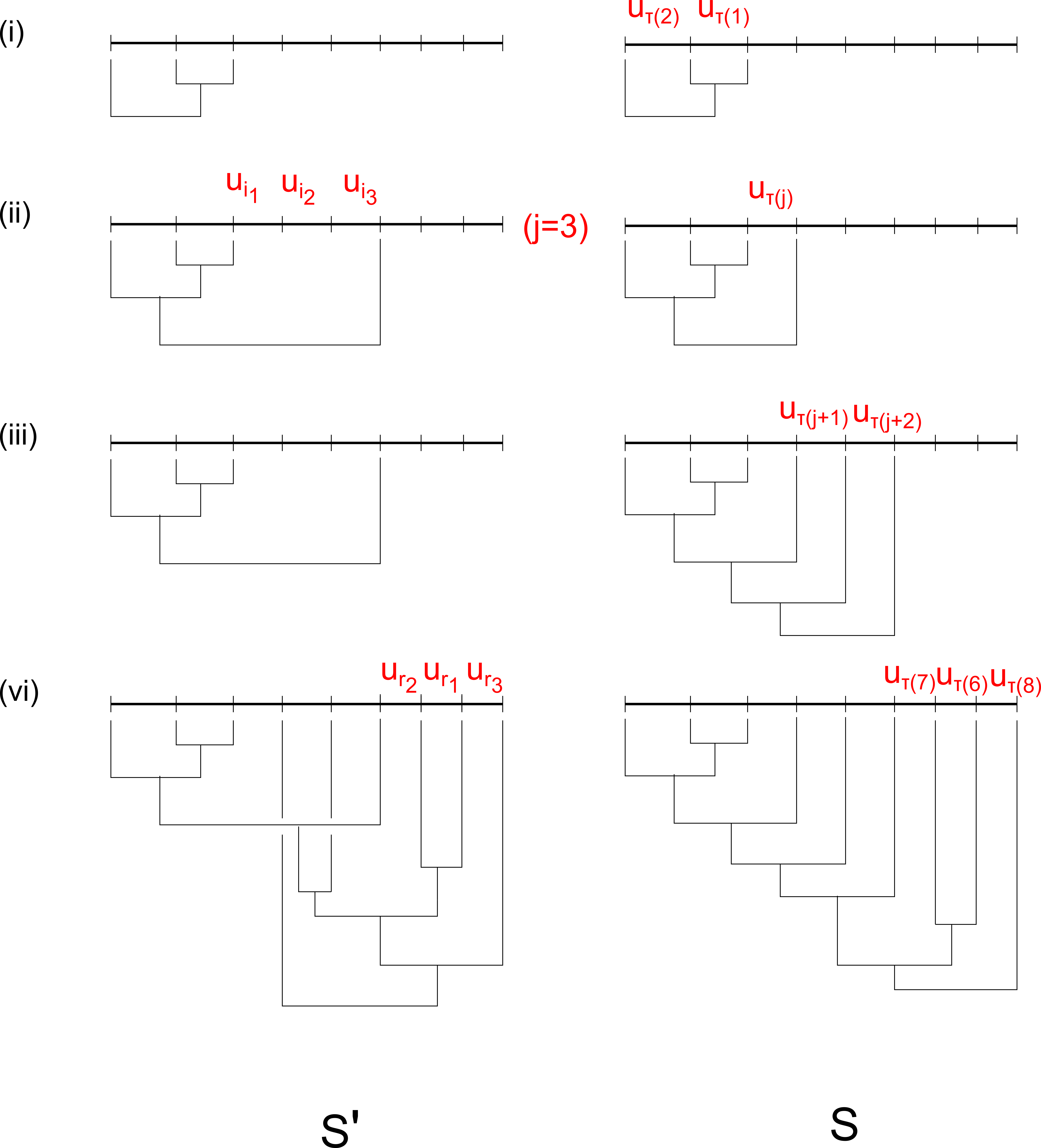}
\end{center}
\caption{Example of the construction of a scenario $S\in  {\cal S}_C(c(z))$ from a scenario $S'\in \bar {\cal S}'_C(c(z))$. The left-hand side corresponds to $S'$ and the right-hand side to $S$. Steps $\textrm{(i)} \ldots \textrm{(iv)}$  correspond to the steps in  construction of $S$ from $S'$.  }
\label{holes}
\end{figure}

\smallskip

To complete the proof of Lemma \ref{lemmaNC}, we are going  show that, for every scenario $S^2 \in \bar {\cal S}_C(c(z))$ with more than one step of coalescence between non-contiguous scenarios, there exist a scenario $S^{3} \in  \bar{\cal S}'_C(c(z))$ such that $E(S^3) \le E(S^2)$. 
We fix $S^2 = (s^2_0, s^2_1, \ldots, s^2_n) \in \bar {\cal S}_C(c(z))$, and the idea is to construct $S^{3}= (s^{3}_0, s^{3}_1, \ldots, s^{3}_n)  \in \bar {\cal S}'_C(c(z))$ along the same lines as in Step 1. 
Let $j_1$ the first step of coalescence between non contiguous blocks in $S^3$ and $j_2$ the second one. 
\begin{itemize}
\item For $0\le i < j_2$, $s^{3}_i := s^2_i$. In words, we copy all the steps, including $j_1$, the first step of coalescence between non neighbouring blocks. 
\item Steps $(s^{3}_{j_2}, \ldots, s^{3}_n)$ are obtained from $(s^2_{j_2},  \ldots, s^2_n)$ in the same way as $S'$ was obtained from $S$ in Step 1.
\end{itemize}
With this construction, $S^{3} \in \bar {\cal S}'_C(c(z))$ (there is only one step of coalescence between non neighbouring blocks, which is $j_1$) and we have $\forall i \in [n], \ C(s^2_i) \le C(s^3_i)$, so 
\begin{align*}
\int_{C^{R}_{\beta, Id}} \frac1{E(S^2)}  dz_1 \ldots dz_n &  \le \int_{C^{R}_{\beta, Id}} \frac1{E(S^3)}  dz_1 \ldots dz_n.
\end{align*}
As $S^{3} \in \bar {\cal S}'_C(c(z))$, combining the previous equation with \eqref{scenarioS'}, for every scenario  $S \in \bar {\cal S}_C(c(z))$, 
\begin{equation*}
\lim_{R \to \infty}\frac1{\log(R)^n} \int_{C^{R}_{\beta, Id}} \frac1{E(S)}  dz_1 \ldots dz_n = 0,
\end{equation*}
which completes the proof Lemma \ref{lemmaNC}.
\end{proof}

\begin{proof}[Proof of Proposition \ref{propetoile}]
This is a direct consequence of Lemmas  \ref{lem:target-step2} and \ref{lemmaNC} and \eqref{eqnStep1}.
\end{proof}

\smallskip

{\bf Step 2.}
Define
\begin{align*}
D^{R}_{\beta} \ \ := \ \ & \{ z_1, \ldots, z_n \in {[R^{a_1},R^{b_1}]^{n_1} \times \ldots \times[R^{a_k},R^{b_k}]^{n_k}}, \textrm{ s.t. if $z_0:= 0$  }\\ & \ \ \   \forall i \ne j \in \{0, \ldots, n\}, \ |z_i - z_j| \ge \beta\} \\
 I^{R}_{\beta} \  \ := \ \ &  \frac1{\log(R)^n} \int_{D^{R}_{\beta}} \mu^{1,z}(c(z)) dz_1\ldots dz_n, 
\end{align*}

Using scaling (see Proposition \ref{ARG-rescale}) and a change of variables, we have:
\begin{equation}  I^{R}_{\beta} \ = \   \frac{R^n}{\log(R)^n} \ \int_ {C^{R}_{\beta}} \mu^{R,z }(c(z) ) dz_1 \ldots dz_n. \label{changevariables}\end{equation}

Recall that $c(z)$ is a partition of order $n$, so from Theorem \ref{thm-stationary}, we have: 
\begin{equation*}
 \left | I^{R}_{\beta}- \frac1{\log(R)^n} \int_ {C^{R}_{\beta}} F(c(z)) \ d{z}_1\ldots d{z}_n  \right | \ \le  \ \frac{f^n(\beta)}{\log(R)^n} \int_ {C^{R}_{\beta}} F(c(z)) \ d{z}_1\ldots d{z}_n 
\end{equation*}
    and $f^n(\beta) \underset{\beta \to \infty}{\longrightarrow}0$.  By taking successive limits, first $R \to \infty$ 
    and then $\beta\to\infty$,  using Proposition \ref{propetoile}
    \begin{equation} 
\lim_{\beta\to\infty} \lim_{R\to\infty}  I^{R}_{\beta}  \ = \ \prod_{i=1}^k n_i! \ b_i^{n_i-1}(b_i-a_i). 
\label{eq:target1} 
\end{equation}

\smallskip

{\bf Step 3.} 
The aim of this step is to show that we can now approximate  $\E \left[\prod_{i=1}^k \vartheta^{R}([a_i,b_i])^{n_i}\right]$ by $I^{R}_{\beta}$. In fact, $I^{R}_{\beta} $
can be obtained from   $\E \left[\prod_{i=1}^k \vartheta^{R}([a_i,b_i])^{n_i}\right]$  by removing a small fraction of the integration domain (see \eqref{eq:above}). More precisely we will show that
 \begin{lemma}
$$\forall \beta \ge 1, \ \lim_{R \to \infty} \left( \E \left[\prod_{i=1}^k \vartheta^{R}([a_i,b_i])^{n_i}\right] - I^{R}_{\beta} \right)   = 0.$$
 \label{lemma:domains}
\end{lemma}
\begin{proof}
We fix $k \in \N, \ n_1, \ldots, n_k \in \N, \ n = n_1 + \ldots + n_k, \  a_1, \ldots, a_k \in [0,1]$,  $b_1, \ldots, b_k\in [0,1], \ a_1 < b_1 < a_2 < b_2 \ldots <a_k < b_k, \ \beta \ge 1$.

\smallskip 
Let us define
\begin{align*}
\dot \Delta^{R}_{\beta} := \{  z_1,\ldots, {z}_n \in [R^{a_1},R^{b_1}]^{n_1} \times \ldots \times[R^{a_k},R^{b_k}]^{n_k}, \ \textrm{ such that if } z_0:= 0, \\   \exists \ i,j \in \{0,\ldots, n\}, \ |{z}_i-{z}_{j}| < { \beta} \}.
\end{align*}
We have 
\begin{equation*}
  \E \left[\prod_{i=1}^k \vartheta^{R}([a_i,b_i])^{n_i}\right] - I^{R}_{\beta} =   \ \frac1{\log(R)^n} \int_ {\dot \Delta^{R}_{\beta}} \mu^{ 1, z }(c(z) ) \ d{z}_1\ldots d{z}_n.
 \end{equation*}
Lemma  \ref{lemma:domains} can be reformulated as follows 
\begin{equation*}
   \lim \limits_{R \to \infty}  \ \frac1{\log(R)^n} \int_ {\dot \Delta^{R}_{\beta}} \mu^{ 1, z}(c(z) ) \ d{z}_1\ldots d{z}_n   \ = \ 0.
\end{equation*}
By symmetry, proving this result reduces to proving that
\begin{equation*}
   \lim \limits_{R \to \infty}  \ \frac1{\log(R)^n} \int_ {\Delta^{R}_{\beta}} \mu^{ 1, z }(c(z) ) \ d{z}_1\ldots d{z}_n   \ = \ 0,
\end{equation*}
where 
\begin{align*}
\Delta^{R}_{\beta} := \dot \Delta^{R}_{\beta} \bigcap \{  z_1,\ldots, {z}_n, \  z_0:=0 < z_1 < \ldots < z_n \}.
\end{align*}

\smallskip

Let $\mathds{S}$ be the set of all subsets of $[n]$ containing at most $n-1$ elements. 
For $S \in \mathds{S}$, define 
\begin{align*}
\Delta^{R}_{\beta, S} = \Delta^{R}_{\beta}\bigcap \{ z_1,\ldots,z_n :\  \forall i  \in S, \ |z_{i} - z_{i-1}| \ge \beta, \ \forall i \notin S,  \ |z_{i} - z_{i-1}|<  \beta \}.
\end{align*}
in such a way that 
$$\Delta^{R}_{\beta} =   \bigcup \limits_ {S \in \mathds{S}}    \Delta^{R}_{\beta, S} .$$
It follows that
\begin{align*}
 \int_ {\Delta^{R}_{\beta}} \mu^{1,  z }(c(z) ) \ d{z}_1\ldots d{z}_n   \ &=  \   \sum \limits_ {S \in \mathds{S}}   \int_ {\Delta^{R}_{\beta, S} }    \mu^{ 1, z }(c(z) ) \ d{z}_{1}\ldots d{z}_{n}  \nonumber \\
 &\le \  \sum \limits_ {S \in \mathds{S}}  \int_ {\Delta^{R}_{\beta, S} } \mu^{1, z }(\pi_S ) \ d{z}_{1}\ldots d{z}_{n}, \nonumber 
\end{align*}
where $\forall S \in \mathds{S}$, 
$\pi_S  = \{ \pi \in {\cal P}_z, \ \forall i,j \in S, \ \ z_i \sim_{\pi} z_j\}$ (and where the inequality follows from the fact that $c(z) \in \pi_S$).
We define  ${z}^{S} := \{z_i,  \ i \in S\}$. Proposition \ref{consistency} gives $\mu^{1, z }(\pi_S ) = \mu^{1, {z}^{S} }(c(z^{S}) )$.

Define
\begin{align*}
\bar \Delta^{R}_{\beta, S}  := \{ (z_i)_{i \in S} : \  \exists \ (z_j)_{j \in [n] \setminus S}, \ \ (z_1, \ldots, z_n) \in \Delta^{R}_{\beta, S} \}.
\end{align*}
We let the reader convince herself that, for any $S \in \mathds{S}$ there exists $m_1, \ldots, m_k \in \N, \ m_1 + \ldots + m_k = |S|$, such that $\bar \Delta^{R}_{\beta, S}$ can be rewritten as
\begin{align*}
\bar \Delta^{R}_{\beta, S} \ = & \  \{ \bar z_1, \ldots \bar z_{|S|}  \in  [R^{a_1-1},R^{b_1-1}]^{m_1} \times \ldots \times[R^{a_k-1},R^{b_k-1}]^{m_k} \ : \\  & \ \bar z_0:=0  \le \bar z_{1} \le \ldots \le  \bar z_{|S|} \textrm{ and } \forall  \ i,j \in S, \ i\ne j, \   |{\bar z}_i-{\bar z}_j| > \beta \}.
\end{align*}
This allows us to rewrite the previous inequality as
\begin{align}
\int_ {\Delta^{R}_{\beta}} \mu^{ 1, z}(c(z) ) \ d{z}_1\ldots d{z}_n \ & \le \
  \sum \limits_ {S \in \mathds{S}}    \int_{\Delta^{R}_{\beta, S} }  \mu^{1, {z}^{S} }(c(z^S) )
 \underset{i \in S}{\underbrace{dz_{1} \ldots dz_{i}\ldots }} \ \underbrace{ \ldots dz_j \ldots }_{j \not \in S}
\nonumber \\ 
  &= \ \sum \limits_ {S \in \mathds{S}}    \int_{\bar \Delta^{R}_{\beta, S} }  \  \mu^{1, {z}^{S}}(c(\bar z) ) d\bar z_{1} \ldots d\bar z_{|S|} \left( \prod_{ j  \notin S} \int_{z_{j-1}}^{z_{j-1} + \beta} dz_j \right) 
     \nonumber\\
    &= \ \beta^{n-|S|} \ \sum \limits_ {S \in \mathds{S}}    \int_{\bar \Delta^{R}_{\beta, S} } 
  \  \mu^{1, {z}^{S}}(c(\bar z) ) d\bar z_{1} \ldots d\bar z_{|S|}, 
\label{majorationDelta}
\end{align} 
where $\bar z = (\bar z_0, \ldots, \bar z_{|S|})$ and $\bar z_0 = 0$.
From \eqref{eq:target1}, we have
$$\lim_{R \to \infty} \frac{1}{\log(R)^{|S|}} \int_{\bar \Delta^{R}_{\beta, S} } 
  \  \mu^{1, {z}^{S}}(c(\bar z) ) d\bar z_{1} \ldots d\bar z_{|S|} \ = \ \prod_{i \in [k], \ m_i \ne 0} b_i^{m_i-1}(b_i - a_i). $$ 
\begin{equation*}
\textrm{As } |S|<n, \ \lim_{R \to \infty} \frac1{\log(R)^n}  \sum \limits_ {S \in \mathds{S}}  \int_{\bar \Delta^{R}_{\beta, S} } 
  \  \mu^{1, {z}^{S} }(c(\bar z) ) d\bar z_{1} \ldots d\bar z_{|S|} \  \  =  \ 0,
 \end{equation*}
 which combined with \eqref{majorationDelta} concludes the proof of the Lemma \ref{lemma:domains}.
 \end{proof}

\smallskip

{\bf Step 4. Conclusion.} 
Combining \eqref{eq:target1}  and  with  Lemma \ref{lemma:domains}  (Step 3), we have proved that for every $k$-tuple of disjoint intervals $\{[a_i,b_i]\}_{i=1}^k$ and any $k$-tuple of integers $\{n_i\}_{i=1}^k$ 
\begin{equation*}
\lim_{R\to\infty}\E\left[\prod_{i=1}^k \vartheta^{R}([a_i,b_i])^{n_i}\right] \ = \   \prod_{i=1}^k n_i! \ b_i^{n_i-1}(b_i-a_i).
 \end{equation*}
So, using Lemma \ref{lemmamoments}, $\vartheta^R$ converges to $\vartheta^{\infty}$ in distribution in the weak topology.
In particular, we have
 \begin{equation*}
{\cal L}_R(0) = \vartheta^R[0,1] +   \frac1{\log(R)} \int_{[0,1]} \mathds{1}_{\{x \sim_{\pi} 0\}} dx
\label{eqn-exp}\end{equation*} 
and $$\frac1{\log(R)} \int_{[0,1]} \mathds{1}_{\{x \sim_{\pi} 0\}} dx \ \le \ \frac1{\log(R)} \underset{R \to \infty}{\longrightarrow} 0 $$
so, using equation \eqref{eq-moments-infty}, we have  $$\forall n \in \N, \  \lim_{R \to \infty} \E[{\cal L}_R(0)^n] = \lim_{R \to \infty} \E[\vartheta^R[0,1]^n] = n! $$
which are the moments of the exponential distribution of parameter 1. As in the proof of Lemma \ref{lemmamoments}, using Carleman's condition (for $k =1$), this implies that ${\cal L}_R(0)$ converges in distribution to an exponential distribution of parameter 1.

\section*{Acknowledgements}
The authors would like to thank Henrique Teot\'onio, Mathieu Tiret and Fr\'ed\'eric Hospital for many interesting and inspiring discussions, as well as two insightful reviewers for their thorough reading of the present work. They also want to thank the \emph{Center for Interdisciplinary Research in Biology} (CIRB, Coll\`ege de France) for funding.

\bibliographystyle{spbasic}      
\bibliography{BiblioRecombin}

\begin{thebibliography}{27}
\providecommand{\natexlab}[1]{#1}
\providecommand{\url}[1]{\texttt{#1}}
\expandafter\ifx\csname urlstyle\endcsname\relax
  \providecommand{\doi}[1]{doi: #1}\else
  \providecommand{\doi}{doi: \begingroup \urlstyle{rm}\Url}\fi

\bibitem[Arratia(1998)]{arratia}
R.~Arratia.
\newblock On the central role of the scale invariant poisson processes on
  (0,infty).
\newblock In \emph{Microsurveys in discrete probability ({P}rinceton, {NJ},
  1997)}, volume~41 of \emph{DIMACS Ser. Discrete Math. Theoret. Comput. Sci.},
  pages 21--41. Amer. Math. Soc., Providence, RI, 1998.

\bibitem[Baird et~al.(2003)Baird, Barton, and Etheridge]{baird}
S.J. Baird, N.H. Barton, and A.M. Etheridge.
\newblock The distribution of surviving blocks of an ancestral genome.
\newblock \emph{Theoretical Population Biology}, 64\penalty0 (4):\penalty0
  451--71, 2003.

\bibitem[Berestycki(2009)]{beresticky}
N.~Berestycki.
\newblock Recent progress in coalescent theory.
\newblock In \emph{Ensaios Matem\'aticos [Mathematical Surveys] 16}. Sociedade
  Brasileira de Matem\'atica, 2009.

\bibitem[Bhaskar and Song(2012)]{bhaskar}
A.~Bhaskar and Y.~S. Song.
\newblock Closed-form asymptotic sampling distributions under the coalescent
  with recombination for an arbitrary number of loci.
\newblock \emph{Advances in Applied Probability}, \penalty0 (44):\penalty0
  391--407, 2012.

\bibitem[Billingsley(1999)]{billingsley1968convergence}
Patrick Billingsley.
\newblock \emph{Convergence of probability measures}.
\newblock Wiley Series in Probability and Statistics: Probability and
  Statistics. John Wiley \& Sons Inc., New York, second edition, 1999.
\newblock ISBN 0-471-19745-9.
\newblock A Wiley-Interscience Publication.

\bibitem[Blackwell(1948)]{blackwell}
D.~Blackwell.
\newblock A renewal theorem.
\newblock \emph{Duke Math. J.}, 15\penalty0 (1):\penalty0 145--150, 1948.

\bibitem[Bobrowski et~al.(2010)Bobrowski, Wojdy{\l}a, and
  Kimmel]{Bobrowski2010}
A.~Bobrowski, T.~Wojdy{\l}a, and M.~Kimmel.
\newblock {Asymptotic behavior of a Moran model with mutations, drift and
  recombination among multiple loci}.
\newblock \emph{Journal of Mathematical Biology}, 61\penalty0 (3):\penalty0
  455--473, Sep 2010.

\bibitem[Chan et~al.(2012)Chan, Jenkins, and Song]{chan}
A.~H. Chan, P.~A. Jenkins, and Y.~S. Song.
\newblock Genome-wide fine-scale recombination rate variation in drosophila
  melanogaster.
\newblock \emph{PLoS Genetics}, e1003090\penalty0 (8), 2012.

\bibitem[den Hollander(2000)]{denHollander}
F.~den Hollander.
\newblock Large deviations.
\newblock In \emph{Fields Institute Monographs}, volume~14. American
  Mathematical Society, 2000.

\bibitem[Durrett(2008)]{durrett}
R.~Durrett.
\newblock \emph{{Probability Models for DNA Sequence Evolution}}.
\newblock Springer, 2 edition, 2008.

\bibitem[Esser et~al.(2016)Esser, Probst, and Baake]{esser}
M.~Esser, S.~Probst, and E.~Baake.
\newblock Partitioning, duality, and linkage disequilibria in the moran model
  with recombination.
\newblock \emph{Journal of mathematical biology}, 73\penalty0 (1):\penalty0
  161—197, July 2016.

\bibitem[Etheridge(2011)]{etheridge2011}
A.~Etheridge.
\newblock \emph{Some Mathematical Models from Population Genetics: {\'E}cole
  D'{\'E}t{\'e} de Probabilit{\'e}s de Saint-Flour XXXIX-2009}.
\newblock Lecture Notes in Mathematics. Springer, 2011.

\bibitem[Griffiths(1991)]{ARG2}
R.~C. Griffiths.
\newblock The two-locus ancestral graph.
\newblock In I.V. Basawa and R.~L. Taylor, editors, \emph{Selected Proceeedings
  of the Symposium on Applied Probability}, pages 100--117. Institute of
  Mathematical Statistics, 1991.

\bibitem[Griffiths et~al.(2016)Griffiths, Jenkins, and Lessard]{lessard}
R.~C. Griffiths, P.~A. Jenkins, and S.~Lessard.
\newblock { A coalescent dual process for a Wright-Fisher diffusion with
  recombination and its applications to haplotype partitioning}.
\newblock \emph{Theor. Popul. Biol.}, 112:\penalty0 126--138, 2016.

\bibitem[Griffiths and Marjoram(1997)]{ARG3}
R.C. Griffiths and P.~Marjoram.
\newblock An ancestral recombination graph.
\newblock In P.~Donnelly and S.~Tavar\'e, editors, \emph{Progress in Population
  Genetics and Human Evolution, IMA Volumes in Mathematics and its
  Applications}, volume~87, pages 257--270. 1997.

\bibitem[Hudson(1983)]{hudson}
R.R. Hudson.
\newblock Properties of the neutral model with intragenic recombination.
\newblock \emph{Theor. Popul. Biol.}, 23\penalty0 (2):\penalty0 213--201, 1983.

\bibitem[Jenkins and Song(2010)]{jenkins1}
P.A. Jenkins and Y.S. Song.
\newblock An asymptotic sampling formula for the coalescent with recombination.
\newblock \emph{Ann. Appl. Probab.}, 20\penalty0 (3):\penalty0 1005--1028,
  2010.

\bibitem[Jenkins et~al.(2015)Jenkins, Fearnhead, and Song]{jenkins}
P.A. Jenkins, P.~Fearnhead, and Y.S. Song.
\newblock Tractable diffusion and coalescent processes for weakly correlated
  loci.
\newblock \emph{Electron. J. Probab.}, 58\penalty0 (20):\penalty0 25, 2015.

\bibitem[Kallenberg(2002)]{kallenberg}
O.~Kallenberg.
\newblock \emph{Foundations of modern probability}.
\newblock Probability and its Applications (New York). Springer-Verlag, New
  York, second edition, 2002.
\newblock ISBN 0-387-95313-2.

\bibitem[Kleibler and Stoyanov(2013)]{keiber}
C.~Kleibler and J.~Stoyanov.
\newblock Multivariate distributions and the moment problem.
\newblock \emph{Journal of Multivariate Analysis}, 113:\penalty0 7--18, 2013.

\bibitem[Lambert(2005)]{lambert}
A.~Lambert.
\newblock The branching process with logistic growth.
\newblock \emph{Ann. Appl. Prob.}, 15:\penalty0 1506--1535, 2005.

\bibitem[McQuillan et~al.(2008)McQuillan, Leutenegger, Abdel-Rahman, Franklin,
  Pericic, and Barac-Lauc]{mcquillan_runs_2008}
R.~McQuillan, A.-L. Leutenegger, R.~Abdel-Rahman, C.~S. Franklin, M.~Pericic,
  and L.~et~al. Barac-Lauc.
\newblock Runs of homozygosity in european populations.
\newblock \emph{The American Journal of Human Genetics}, 83\penalty0
  (3):\penalty0 359--372, 2008.

\bibitem[Neuhauser and Krone(1997)]{ASG}
C.~Neuhauser and S.M. Krone.
\newblock The genealogy of samples in models with selection.
\newblock \emph{Genetics}, 2\penalty0 (145):\penalty0 519--34, 1997.

\bibitem[Sabeti et~al.(2002)Sabeti, Reich, Higgins, Levine, Richter, Schaffner,
  Gabriel, Platko, Patterson, McDonald, Ackerman, Campbell, Altshuler, Cooper,
  Kwiatkowski, Ward, and Lander]{sabeti}
P.C. Sabeti, D.E. Reich, J.M. Higgins, H.Z. Levine, D.J. Richter, S.F.
  Schaffner, S.B. Gabriel, J.V. Platko, N.J. Patterson, G.J. McDonald, H.C.
  Ackerman, S.J. Campbell, D.~Altshuler, R.~Cooper, D.~Kwiatkowski, R.~Ward,
  and E.S. Lander.
\newblock Detecting recent positive selection in the human genome from
  haplotype structure.
\newblock \emph{Nature}, 419:\penalty0 832--837, 2002.

\bibitem[Shohat and Tamarkin(1950)]{shohat}
J.A. Shohat and J.D. Tamarkin.
\newblock \emph{The Problem of Moments}.
\newblock American Mathematical Society, revised edition, 1950.

\bibitem[Teot\'onio et~al.(2017)Teot\'onio, Estes, Phillips, and
  Baer]{teotonio}
H.~Teot\'onio, S.~Estes, P.~C. Phillips, and C.~F. Baer.
\newblock Experimental evolution with caenorhabditis nematodes.
\newblock \emph{Genetics}, 2\penalty0 (206):\penalty0 691--716, 2017.

\bibitem[Wiuf and Hein(1997)]{WH}
C.~Wiuf and H.~Hein.
\newblock {On the number of ancestor to a DNA sequence}.
\newblock \emph{Genetics}, 147:\penalty0 1459--1468, 1997.

\end{thebibliography}

\end{document}